\documentclass[reqno]{amsart}
\usepackage{amssymb}
\usepackage{graphicx}
\usepackage{amsmath}
\usepackage{mathbbol}

\usepackage[usenames, dvipsnames]{color}
\usepackage{verbatim}
\usepackage{mathrsfs}
\usepackage{bm}
\usepackage{cite}
\usepackage{color}
\allowdisplaybreaks[4]

\numberwithin{equation}{section}

\newtheorem{theorem}{Theorem}[section]
\newtheorem{corollary}[theorem]{Corollary}
\newtheorem{lemma}[theorem]{Lemma}
\newtheorem{prop}[theorem]{Proposition}

\theoremstyle{definition}
\newtheorem{remark}[theorem]{Remark}

\theoremstyle{definition}

\theoremstyle{definition}

\makeatletter
\def\dashint{\operatorname%
{\,\,\text{\bf-}\kern-.98em\DOTSI\intop\ilimits@\!\!}}
\makeatother

\def\\det{\text{\det}}

\def\.5{\frac{1}{2}}

\newcommand{\RN}[1]{%
  \textup{\uppercase\expandafter{\romannumeral#1}}%
}

\newcommand{\Div}{\operatorname{div}}

\renewcommand{\epsilon}{\varepsilon}

\newcounter{marnote}

\begin{document}

\title[Estimates for stress concentration of the  Stokes flow ]{Estimates for Stress Concentration between Two Adjacent Rigid Inclusions in Two-dimensional Stokes Flow}

\author[H.G. Li]{Haigang Li}
\address[H.G. Li]{School of Mathematical Sciences, Beijing Normal University, Laboratory of MathematiCs and Complex Systems, Ministry of Education, Beijing 100875, China.}
\email{hgli@bnu.edu.cn}

\author[L.J. Xu]{Longjuan Xu}
\address[L.J. Xu]{Academy for Multidisciplinary Studies, Capital Normal University, Beijing 100048, China.}
\email{ljxu311@163.com}


\date{\today} 

\begin{abstract}
It is vital important in material sciences and fluid mechanics to study the field enhancements in the narrow region between two inclusions. Complex fluids including particle suspensions usually result in complicated flow behavior. In this paper we establish the pointwise upper bounds of the gradient and the second-order partial derivatives for the Stokes flow when two rigid particles are closely spaced suspending in an open bounded domain and away from the boundary in dimension two. Moreover, the lower bounds of the gradient estimates at the narrowest place of the neck region show the optimality of the blow-up rate. These results are valid for inclusions with arbitrary shape.

\end{abstract}

\maketitle

\section{Introduction}

\subsection{Background and Problem Formulation}\label{subsec1.1}
It is well known that Babu\u{s}ka et al. \cite{Bab} numerically studied the initiation and growth of damage in composite materials and observed that the stress still remains bounded even if the distance $\varepsilon$ between inclusions tends to zero. Stimulated by \cite{Bab}, there has been significant progress in the understanding of the field enhancement or the stress concentration in the last two decades in the theory of partial differential equations and numerical analysis. The numerical observation of the boundedness of the stress in \cite{Bab} was rigorously proved by Bonnetier and Vogelius \cite{BV}, Li and Vogelius \cite{LV}, Li and Nirenberg \cite{LN} that the gradient of the solution is uniform bounded with respect to $\varepsilon$. These bounded estimates in \cite{BV,LV,LN} of course depend on the ellipticity of the coefficients. In order to investigate the enhancement caused by the smallness of the interparticles, the coefficients occupying in the inclusions are mathematically assumed to degenerate to infinity or zero, then the situation becomes quite different. In  the context of the Lam\'e system,  the blow-up rate of the stress in between two stiff inclusions is $\varepsilon^{-1/2}$ in dimension two \cite{BLL,KY}, $(\varepsilon|\ln\varepsilon|)^{-1}$ in dimension three \cite{BLL2,Li2021}.  Similar results have been obtained for the scalar case, describing enhancement of the electric field by the perfect conductivity problem  \cite{AKL,AKL3,BC,BLY,BT,Yun,Dong,Li-Li}, and by the  insulated problem  \cite{DLY,DLY2,Wen,LY,Yun2}. From the perspective of practical application in engineering, it is important to characterize the singular behavior, see \cite{KLeY,KLiY,KLiY2,KangYu,Li}. The corresponding results for nonlinear problems, say,  for $p$-Laplacian and Finsler Laplacian, have been investigated in \cite{CS1,CS2}. 

Inspired by the above works for the stress and electric field in the solid-solid composite materials, it is also an interesting and important problem to investigate the stress concentration in the steady Stokes systems when two suspending rigid inclusions are close to each other. There have been several important works on the study of the Stokes flow in the presence of two circular cylinders, see Jeffrey \cite{Jef1,Jef2}, Wannier \cite{Wann}, Hillairet \cite{H}, Hillairet-Kela\"i \cite{HK} and the references therein. However, when the cylinders are close-to-touching, it is difficult to deal with the singular behavior of the solution. Recently, Ammari et al. \cite{AKKY} employed the method of bipolar coordinates to investigate the stress concentration in the two-dimensional Stokes flow when inclusions are the two-dimensional cross sections of circular cylinders of the same radii and the background velocity field is linear.  They derived an asymptotic representation formula for the stress and completely captured the singular behavior of the stress by using the bipolar coordinates and showed the blow-up rate is $\varepsilon^{-1/2}$ in dimension two, the same as the linear elasticity case. However, as mentioned in \cite{AKKY} and in Kang's ICM talk \cite{K}, it is
quite interesting and challenging to extend them to the more general case when the cross sections of the cylinders are strictly convex. In this paper, we address this problem by assuming that inclusions are convex in dimension two. The method developed here may be applied to deal with the three-dimensional case, we will investigate this problem in  forthcoming papers \cite{LX2,LXZ,LXZ2}. 

Before we state our main results precisely, we fix our domain and notation. Let $D$ be a bounded open set in $\mathbb{R}^{2}$, and let $D_{1}^{0}$ and $D_{2}^{0}$ be a pair of (touching at the origin) $C^3$-convex subdomains of $D$, far away from $\partial D$, and satisfy
\begin{equation}\label{def_D0}
D_{1}^{0}\subset\{(x_1, x_{2})\in\mathbb R^{2}~:~ x_{2}>0\},\quad D_{2}^{0}\subset\{(x_1, x_{2})\in\mathbb R^{2}~:~ x_{2}<0\},
\end{equation}
with $\{x_2=0\}$ being their common tangent plane, after a rotation of coordinates if necessary. This means that the axis $x_{1}$ are in the tangent plane. 
Translate $D_{i}^{0}$ ($i=1,\, 2$) by $\pm\frac{\varepsilon}{2}$ along $x_{2}$-axis in the following way
\begin{equation*}
D_{1}^{\varepsilon}:=D_{1}^{0}+(0,\frac{\varepsilon}{2})\quad \text{and}\quad D_{2}^{\varepsilon}:=D_{2}^{0}+(0,-\frac{\varepsilon}{2}).
\end{equation*}
So the distance between $D_1$ and $D_2$ is $\varepsilon$. For simplicity of notation, we drop the superscript $\varepsilon$ and denote
\begin{equation*}
D_{i}:=D_{i}^{\varepsilon}\, (i=1, \, 2), \quad  \Omega:=D\setminus\overline{D_1\cup D_2}.
\end{equation*}
Denote by $P_1:= (0,\frac{\varepsilon}{2})$, $P_2:=(0,-\frac{\varepsilon}{2}) $ the two nearest points between $\partial D_1$ and $\partial D_2$ such that
$\varepsilon=\text{dist}(P_1, P_2)=\text{dist}(\partial D_1, \partial D_2)$. Assume 
$$\mbox{dist}(D_{1}\cup D_{2},\partial D)>\kappa_{0}>0,$$
where $\kappa_{0}$ is a constant independent of $\varepsilon$. Also, suppose that the $C^{3}$-norms of $D_1$ and $D_2$ are bounded by another positive constant $\kappa_{1}$, independent of $\varepsilon$. Denote the linear space of rigid displacements in $\mathbb{R}^{2}$:
$$\Psi:=\Big\{{\boldsymbol\psi}\in C^{1}(\mathbb{R}^{2};\mathbb{R}^{2})~|~e({\boldsymbol\psi}):=\frac{1}{2}(\nabla{\boldsymbol\psi}+(\nabla{\boldsymbol\psi})^{\mathrm{T}})=0\Big\},$$
with a basis $\{\boldsymbol{e}_{i},~x_{k}\boldsymbol{e}_{j}-x_{j}\boldsymbol{e}_{k}~|~1\leq\,i\leq\,2,~1\leq\,j<k\leq\,2\}$, where $e_{1},e_{2}$ are standard bases of $\mathbb{R}^{2}$.

Let us consider the following Stokes flow containing two adjacent rigid particles:
\begin{align}\label{sto}
\begin{cases}
\mu \Delta {\bf u}=\nabla p,\quad\quad~~\nabla\cdot {\bf u}=0,&\hbox{in}~~\Omega:=D\setminus\overline{D_{1}\cup D_{2}},\\
{\bf u}|_{+}={\bf u}|_{-},&\hbox{on}~~\partial{D_{i}},~i=1,2,\\
e({\bf u})=0, &\hbox{in}~~D_{i},~i=1,2,\\
\int_{\partial{D}_{i}}\frac{\partial {\bf u}}{\partial \nu}\Big|_{+}\cdot{\boldsymbol\psi}_{\alpha}-\int_{\partial{D}_{i}}p\,{\boldsymbol\psi}_{\alpha}\cdot
\nu=0
,&i=1,2,\\
{\bf u}={\boldsymbol\varphi}, &\hbox{on}~~\partial{D},
\end{cases}
\end{align}
where $\mu>0$, ${\boldsymbol\psi}_{\alpha}\in\Psi$, $\alpha=1,2,3$, 
$\frac{\partial {\bf u}}{\partial \nu}\big|_{+}:=\mu(\nabla {\bf u}+(\nabla {\bf u})^{\mathrm{T}})\nu,$
and $\nu$ is the  unit outer normal vector of $D_{i},~i=1,2$. Here and throughout this paper the subscript $\pm$ indicates the limit from outside and inside the domain, respectively. Since $D$ is bounded, from $\nabla\cdot{\bf u}=0$ and Gauss theorem, it follows that the
prescribed velocity field ${\boldsymbol\varphi}$ must satisfy the compatibility condition:
\begin{equation}\label{compatibility}
\int_{\partial{D}}{\boldsymbol\varphi}\cdot \nu\,=0,
\end{equation}
to achieve the existence and uniqueness of the solutions. For the Stokes flow in a bounded domain, following the work of Ladyzhenskaya \cite{Lady1959}, it is simple to show the existence and uniqueness of a generalized solution by an integral variational formulation and Riesz representation theorem. It is well known that since the Stokes flow is elliptic in the sense of Douglis-Nirenberg, see \cite{T1984}, the regularity along with appropriate estimates can be obtained directly from the general theory of Agmon, Douglis and Nirenberg \cite{ADN1964} and Solonnikov \cite{Solonni1966}.  

For the regularity of ${\bf u}$, we would like to point out that near $\partial{D}$, the coefficent of \eqref{sto} is constant, and the boundary, as well as the boundary data ${\boldsymbol\varphi}$, is appropriately smooth, so standard elliptic boundary regularity results immediately imply that ${\bf u}$ is $C^{1}$ there. Indeed, since the Stokes system is elliptic in the sense of Douglis-Nirenberg, away from the origin (where the two particles may touch), standard elliptic regularity results (for operators with constant coefficients) can be obtained, as a particular case, from the work of Agmon, Douglis and Nirenberg \cite{ADN1964} and Solonnikov \cite{Solonni1966}, and immediately imply that $|\nabla{\bf u}|$ is bounded. The origin (especially when $\varepsilon=0$), however, presents a serious problem. As mentioned before, a question of particular interest is  whether the stress remains uniformly bounded, even when inclusions touch or nearly touch. Therefore, the main purpose of this paper is to investigate the dependence of $|\nabla{\bf u}|$ on the distance $\varepsilon$ when $\varepsilon$ tends to zero. To this aim, we give more specific information of our domain. Since we assume that $\partial D_1$ and $\partial D_2$ are of $C^{3}$, then near the origin there exists a constant $R$, independent of $\varepsilon$, such that the portions of $\partial D_1$ and $\partial D_2$ are represented, respectively, 
by graphs 
\begin{equation*}
x_2=\frac{\varepsilon}{2}+h_1(x_1)\quad\text{and}\quad x_2=-\frac{\varepsilon}{2}-h_2(x_1),\quad \text{for}~ |x_1|\leq 2R,
\end{equation*}
where $h_1$, $h_2\in C^{3}(B'_{2R}(0))$ and satisfy 
\begin{align}
&-\frac{\varepsilon}{2}-h_{2}(x_1) <\frac{\varepsilon}{2}+h_{1}(x_1),\quad\mbox{for}~~ |x_1|\leq 2R,\label{h1-h2}\\
&h_{1}(0)=h_2(0)=0,\quad \partial_{x_1} h_{1}(0)=\partial_{x_1}h_2(0)=0,\label{h1h1}\\
&h_{1}(x_1)=h_2(x_1)=\frac{\kappa_{2}}{2}x_1^{2}+O(|x_1|^{3}),\quad\mbox{for}~~|x_1|<2R,\label{h1h14}
\end{align}
where the constant $\kappa_{2}>0$. For $0\leq r\leq 2R$, let us define the neck region between $D_{1}$ and $D_{1}$ by
\begin{equation*}
\Omega_r:=\left\{(x_1,x_{2})\in \Omega~:~ -\frac{\varepsilon}{2}-h_2(x_1)<x_{2}<\frac{\varepsilon}{2}+h_1(x_1),~|x_1|<r\right\}.
\end{equation*}

Throughout this paper, we say a constant is {\em{universal}} if it depends only on $\mu,\kappa_{0},\kappa_{1},\kappa_{2}$, and the upper bounds of the $C^{3}$ norm of $\partial{D}$, $\partial{D}_{1}$ and $\partial{D}_{2}$, but independent of $\varepsilon$.

\subsection{Upper Bounds of $|\nabla{\bf u}|$ and $|p|$}
Let  $({\bf u},p)$ be a pair of solution to \eqref{sto}. We introduce the Cauchy stress tensor
\begin{equation}\label{defsigma}
\sigma[{\bf u},p]=2\mu e({\bf u})-p\mathbb{I},
\end{equation}
where $\mathbb{I}$ is the identity matrix. Then we reformulate \eqref{sto} as
\begin{align}\label{Stokessys}
\begin{cases}
\nabla\cdot\sigma[{\bf u},p]=0,~~\nabla\cdot {\bf u}=0,&\hbox{in}~~\Omega,\\
{\bf u}|_{+}={\bf u}|_{-},&\hbox{on}~~\partial{D_{i}},~i=1,2,\\
e({\bf u})=0, &\hbox{in}~~D_{i},~i=1,2,\\
\int_{\partial{D}_{i}}\sigma[{\bf u},p]\cdot{\boldsymbol\psi}_{\alpha}
\nu=0
,&i=1,2,\alpha=1,2,3,\\
{\bf u}={\boldsymbol\varphi}, &\hbox{on}~~\partial{D}.
\end{cases}
\end{align}
Even though there are some technical connections between this paper and previous work on the Lam\'e system, most of the material here is essentially new, due to the presence of the pressure terms. Part (i) of Theorem \ref{mainthm2D} below extends the analogous result in \cite{AKKY} for the circular inclusion case to more general inclusions with arbitrary shape, part (ii) of Theorem \ref{mainthm2D} gives the etimates for the second-order derivatives of ${\bf u}$. The methodology used in the present paper is completely different from those employed in \cite{AKKY}.

\begin{theorem}\label{mainthm2D}(Upper Bounds)
Assume that $D_1,D_2,D,\Omega$, and $\varepsilon$ are defined as in Subsection \ref{subsec1.1}, and ${\boldsymbol\varphi}\in C^{2,\alpha}(\partial D;\mathbb R^2)$ for some $0<\alpha<1$. Let ${\bf u}\in H^1(D;\mathbb R^2)\cap C^1(\bar{\Omega};\mathbb R^2)$ and $p\in L^2(D)\cap C^0(\bar{\Omega})$ be the solution to \eqref{Stokessys} and \eqref{compatibility}. Then for sufficiently small $0<\varepsilon<1$, the following assertions hold.

(i) For any $x\in \Omega_{R}$, 
\begin{align*}
&|\nabla{{\bf u}}(x)|\leq \frac{C}{(\varepsilon+x_1^2)^{1/2}}\|{\boldsymbol\varphi}\|_{C^{2,\alpha}(\partial D;\mathbb R^2)},\\
&\inf_{c\in\mathbb{R}}\|p+c\|_{C^0(\bar{\Omega}_{R})}\leq \frac{C}{\varepsilon}\|{\boldsymbol\varphi}\|_{C^{2,\alpha}(\partial D;\mathbb R^2)},
\end{align*}
and
$$\|\nabla{\bf u}\|_{L^{\infty}(\Omega\setminus\Omega_{R})}+\|p\|_{L^{\infty}(\Omega\setminus\Omega_{R})}\leq\,C\|{\boldsymbol\varphi}\|_{C^{2,\alpha}(\partial D;\mathbb R^2)},$$
where $C$ is a universal constant independent of $\varepsilon$. In particular, 
\begin{equation*}
\|\nabla{{\bf u}}\|_{L^{\infty}(\Omega)}\leq \frac{C}{\sqrt\varepsilon}\|{\boldsymbol\varphi}\|_{C^{2,\alpha}(\partial D;\mathbb R^2)};
\end{equation*}

(ii) For any $x\in \Omega_{R}$, 
\begin{align*}
|\nabla^2{{\bf u}}(x)|+|\nabla p(x)|\leq\frac{C}{(\varepsilon+x_1^2)^{3/2}}\|{\boldsymbol\varphi}\|_{C^{2,\alpha}(\partial D;\mathbb R^2)},
\end{align*}
and
$$\|\nabla^2{\bf u}\|_{L^{\infty}(\Omega\setminus\Omega_{R})}+\|\nabla p\|_{L^{\infty}(\Omega\setminus\Omega_{R})}\leq\,C\|{\boldsymbol\varphi}\|_{C^{2,\alpha}(\partial D;\mathbb R^2)}.$$
\end{theorem}

\begin{remark}
The proof of Theorem \ref{mainthm2D} in Subsection \ref{prfthm} actually gives a point-wise upper bound estimate of $p$ up to a constant; see \eqref{upper-p2D} for the detail. 
\end{remark}

\begin{remark}
The presence of $p$ makes the construction of auxiliary functions  in Subsection \ref{auxiliary} below more involved. This is an essential difficulty compared to the case in Lam\'{e} systems. For this, we have to establish the estimates of $\nabla^2{\bf u}_i$ and $\nabla p_i$, see Propositions \ref{propu11}--\ref{propu0} below. With these estimates in hand, we are ready to prove the upper bounds of $\nabla^2{\bf u}$ and $\nabla p$. However, it is not easy to capture the main singular terms of $\nabla^2{\bf u}$ and $\nabla p$, even to obtain a suitable lower bound  to show the optimality, a similar reason given in Remark \ref{rmk-p}.
\end{remark}

\begin{remark}
If ${\boldsymbol\varphi}=0$, then the solution to \eqref{sto} is ${\bf u}\equiv0$. Theorem \ref{mainthm2D} is trivial in this case. So we only need to prove them for $\|{\boldsymbol\varphi}\|_{C^{2,\alpha}(\partial D;\mathbb R^2)}=1$, by considering ${\bf u}/\|{\boldsymbol\varphi}\|_{C^{2,\alpha}(\partial D;\mathbb R^2)}$.
\end{remark}

\begin{remark}
We would like to point out that our method can be directly applied to deal with the $m$-convex inclusions case, say,  $h_{1}=\kappa_{m}|x_1|^{2m}+O(|x_1|^{2m+1})$ and $h_{2}(x_1)=\tilde\kappa_{m}|x_1|^{2m}+O(|x_1|^{2m+1})$, $m\geq1$, and establish a relationship between the blow-up rates of gradient and the order of the relative convexity of inclusions. The details are left to the interested readers.
\end{remark}

\subsection{Lower Bounds of $|\nabla{\bf u}|$}

To show that the above blow-up rates are optimal, we shall provide a lower bound of $|\nabla{\bf u}(x)|$ on the segment $\overline{P_1 P_2}$, with the same blow-up rate above, under some additional symmetric assumptions on the domain and the given boundary data, for simplicity. Suppose that 
\begin{align*}
({\rm S_{H}}): ~&~D_{1}\cup D_{2}~\mbox{ and~} D ~\mbox{are~ symmetric~ with~ repect~ to~ each~ coordinate~axis~}x_{1}, \\&\mbox{~and~}\mbox{the coordinate plane~} \{x_{2}=0\};
\end{align*}
and
\begin{align*}
({\rm S_{{\boldsymbol\varphi}}}):~~ {\boldsymbol\varphi}^{i}(x)=-{\boldsymbol\varphi}^{i}(-x),\quad\,i=1,2.\hspace{5.6cm}
\end{align*}

Let $\varepsilon=0$, and recall \eqref{def_D0}, 
$$D_{1}^{0}:=\{x\in\mathbb R^{2}~\big|~ x+P_{1}\in D_{1}\},\quad~ D_{2}^{0}:=\{x\in\mathbb R^{2}~\big|~ x-P_{2}\in D_{2}\},$$ 
and set $\Omega^{0}:=D\setminus\overline{D_{1}^{0}\cup D_{2}^{0}}$. Let us define a linear and continuous functional of ${\boldsymbol\varphi}$:
\begin{align}\label{blowupfactor}
\tilde b_{j}^{*\alpha}[{\boldsymbol\varphi}]:=
\int_{\partial D_j^0}{\boldsymbol\psi}_\alpha\cdot\sigma[{\bf u}^*,p^*]\nu,\quad\alpha=1,2,3,
\end{align}
where $({\bf u}^*,p^*)$ verify
\begin{align}\label{maineqn touch}
\begin{cases}
\nabla\cdot\sigma[{\bf u}^{*},p^{*}]=0,\quad\nabla\cdot {\bf u}^{*}=0,\quad&\hbox{in}\ \Omega^{0},\\
{\bf u}^{*}=\sum_{\alpha=1}^{3}C_{*}^{\alpha}{\boldsymbol\psi}_{\alpha},&\hbox{on}\ \partial D_{1}^{0}\cup\partial D_{2}^{0},\\
\int_{\partial{D}_{1}^{0}}{\boldsymbol\psi}_\alpha\cdot\sigma[{\bf u}^*,p^{*}]\nu+\int_{\partial{D}_{2}^{0}}{\boldsymbol\psi}_\alpha\cdot\sigma[{\bf u}^*,p^{*}]\nu=0,&\alpha=1,2,3,\\
{\bf u}^{*}={\boldsymbol\varphi},&\hbox{on}\ \partial{D},
\end{cases}
\end{align}
where the constants $C_{*}^{\alpha}$, $\alpha=1,2,3$, are uniquely determined by the solution $({\bf u}^*,p^*)$. We obtain a lower bound of $|\nabla{{\bf u}}|$ on the segment $\overline{P_{1}P_{2}}$.

\begin{theorem}\label{mainthm2}(Lower Bound)
Assume that $D_1,D_2,D,\Omega$, and $\varepsilon$ are defined as above, and ${\boldsymbol\varphi}\in C^{2,\alpha}(\partial D;\mathbb R^2)$ for some $0<\alpha<1$. Let ${\bf u}\in H^1(D;\mathbb R^2)\cap C^1(\bar{\Omega};\mathbb R^2)$ and $p\in L^2(D)\cap C^0(\bar{\Omega})$ be a solution to \eqref{Stokessys} and \eqref{compatibility}. Then if $\tilde b_{1}^{*1}[{\boldsymbol\varphi}]\neq0$, then there exists  a sufficiently small $\varepsilon_0>0$, such that for $0<\varepsilon<\varepsilon_0$, 
\begin{equation*}
|\nabla{{\bf u}}(0,x_{2})|\geq \frac{\Big|\tilde b_1^{*1}[{\boldsymbol\varphi}]\Big|}{C\sqrt\varepsilon}\|{\boldsymbol\varphi}\|_{C^{2,\alpha}(\partial D;\mathbb R^2)},\quad|x_{2}|\leq\varepsilon.
\end{equation*}
\end{theorem}

\begin{remark}
For example, suppose that $D_{1}=B_{1}(0,1+\frac{\varepsilon}{2})$, $D_{2}=B_{1}(0,-1-\frac{\varepsilon}{2})$ are two adjacent balls, contained in a bigger ball $D=B_{4}(0)$. Obviously, they satisfy the symmetry condition $({\rm S_{H}})$. If we take  ${\boldsymbol\varphi}=(0,x_{2})^{\mathrm T}$, clearly satisfying $({\rm S_{{\boldsymbol\varphi}}})$, then through numerical simulation it is not difficult to check  that $\Big|\tilde b_1^{*1}[{\boldsymbol\varphi}]\Big|\neq0$. To further consider the effect from boundary data, we study the boundary estimates when  particles are close to boundary in our forthcoming paper \cite{LXZ}, where new difficulties need to overcome because $h_{1}(x_1)\neq h_{2}(x_1)$ there. 
\end{remark}

\subsection{Main difficulties and Strategy of the proof} 
We follow the iteration approach first used in \cite{LLBY} and further developed in \cite{BLL,BLL2,CL}   to handle the Stokes problem \eqref{sto}. The advantage of this approach is that it does not rely on maximum principle and is therefore applicable to systems of equations.  However, there are several difficulties need to overcome in applying the iteration framework to the Stokes system, although it is known that \cite{AKKL} when the Lam\'{e} constant $\lambda$ goes to infinity while $\mu$ is fixed, the Lam\'{e} system approaches to a modified Stokes  problem $\mu\Delta u+\nabla p=0$ with $p=\lambda\nabla\cdot u$ bounded.  First, since the flow is incompressible, that is, the divergence of the velocity vector in Stokes flow is confined to be zero, it becomes more interesting and challenging to prove whether the stress blows up or not in the case of Stokes flow and how large it is if it actually occurs. 

By taking advantage of $e({\bf u})=0$ in $D_{1}$ and $D_{2}$ and the continuity of the transmission condition, we first decompose the problem at the cost of introducing six free constants and seven Dirichlet boundary problems. For these Dirichlet boundary problems, we focus only on the neck region between two particles and construct a family of auxiliary functions with divergence free by making use of the Keller-type function to fit the boundary conditions, although its divergence obviously does not vanish. To the authors' knowledge there are no papers that manage to tackle this difficulty caused by the divergence free condition. This is the first novelty and difference. Besides, the appearance of the pressure term $p$ turns out to be another essential difficulty. In order to apply an adapted version of energy iteration approach, it is probably the hardest issue to prove the boundedness of the global energy and to obtain appropriate estimates of the local energy for their differences of the solutions and auxiliary functions. Hence, the selection of auxiliary functions for $p$ is also a portion of our constructions.  We overcome this difficulty by choosing certain $\bar{p}$ to make the right-hand side of the equations as small as possible. This is quite different, and much more complicated than the standard outcome in the case of Lam\'e system. More detailed description is presented in Section \ref{sec2D} below. We would like to point out that our method works well for the inclusions with arbitrary shape and in all dimensions. 

The rest of this paper is organized as follows: In Section \ref{sec2D}, we present our main idea to establish the gradient estimates for the solution of Stokes system \eqref{Stokessys}. We explain the main difficulties needed to overcome in the course of adapting the iteration approach and list the ingredients to prove Theorem \ref{mainthm2D}. A sketched proof is given in the end of Section \ref{sec2D}. To prepare for proving the estimates listed in Section \ref{sec2D}, some basic results for the Stokes equations are contained in Section \ref{sec2}. We recall the $W^{m,p}$ estimates for the Stokes equations with partial zero boundary condition that allow us to combine with the bootstrap argument and  scaling argument to establish a general $W^{1,\infty}$ estimate in the narrow region.  Meanwhile, we give a Caccioppoli-type inequality here, which is a starting point to build our iteration formula. In Section \ref{sec_estimate2D}, for each auxiliary function ${\bf v}_{i}^{\alpha}$, we calculate the concrete estimates required for the iteration process, then prove the estimates for ${\bf u}_{i}^{\alpha}$ listed in Section \ref{sec2D}. More precisely, for $\alpha=1$, we use Lemma \ref{lemmaenergy} to prove that the global energy of ${\bf w}_{i}^{1}$ is bounded, then make use of the iteration technique and $W^{1,\infty}$ estimate to prove Proposition \ref{propu11}. However, for $\alpha=2,3$, although the constructed ${\bf v}_{i}^{\alpha}$ are divergence free, they do not satisfy the assumptions of Lemma \ref{lemmaenergy}. By observations, we rewrite the terms $\mu\Delta{\bf v}_{i}^{\alpha}-\nabla\bar{p}_{i}^{\alpha}$ as polynomials of $x_{2}$ and then use the integration by parts with respect to $x_{2}$ to prove the analogous conclusion as Lemma \ref{lemmaenergy}.  Finally, we prove Theorem \ref{mainthm2} for the lower bounds of $|\nabla{\bf u}|$ in Section \ref{sec5} and also give the point-wise estimates for the Cauchy stress tensor $\sigma[{\bf u},p]$.

\section{Main Ingredients of the Proof}\label{sec2D}

This section is devoted to presenting the main ingredients and a sketch of the proof of Theorem \ref{mainthm2D}.

\subsection{Decomposition of the Solution}
In dimension two,
$$\Psi=\mathrm{span}\Bigg\{{\boldsymbol\psi}_{1}=\begin{pmatrix}
1 \\
0
\end{pmatrix},
{\boldsymbol\psi}_{2}=\begin{pmatrix}
0\\
1
\end{pmatrix},
{\boldsymbol\psi}_{3}=\begin{pmatrix}
x_{2}\\
-x_{1}
\end{pmatrix}
\Bigg\}.$$
It follows from $e({\bf u})=0$ that
$${\bf u}=\sum_{\alpha=1}^{3}C_{i}^{\alpha}{\boldsymbol\psi}_{\alpha},\quad\mbox{in}~D_i,\quad i=1,2,$$
where $C_{i}^{\alpha}$ are six free constants to be determined by $({\bf u},p)$ later. We decompose the solution of \eqref{Stokessys}  as follows:
\begin{align}
{\bf u}(x)&=\sum_{i=1}^{2}\sum_{\alpha=1}^{3}C_i^{\alpha}{\bf u}_{i}^{\alpha}(x)+{\bf u}_{0}(x),\quad 
p(x)=\sum_{i=1}^{2}\sum_{\alpha=1}^{3}C_i^{\alpha}p_{i}^{\alpha}(x)+p_{0}(x),\quad x\in\,\Omega,\label{ud}
\end{align}
where ${\bf u}_{i}^{\alpha},{\bf u}_{0}\in{C}^{2}(\Omega;\mathbb R^2),~p_{i}^{\alpha}, p_0\in{C}^{1}(\Omega)$, respectively, satisfying
\begin{equation}\label{equ_v12D}
\begin{cases}
\nabla\cdot\sigma[{\bf u}_{i}^\alpha,p_{i}^{\alpha}]=0,\quad\nabla\cdot {\bf u}_{i}^{\alpha}=0,&\mathrm{in}~\Omega,\\
{\bf u}_{i}^{\alpha}={\boldsymbol\psi}_{\alpha},&\mathrm{on}~\partial{D}_{i},\\
{\bf u}_{i}^{\alpha}=0,&\mathrm{on}~\partial{D_{j}}\cup\partial{D},~j\neq i,
\end{cases}i=1,2,
\end{equation}
and
\begin{equation}\label{equ_v32D}
\begin{cases}
\nabla\cdot\sigma[{\bf u}_{0},p_0]=0,\quad\nabla\cdot {\bf u}_{0}=0,&\mathrm{in}~\Omega,\\
{\bf u}_{0}=0,&\mathrm{on}~\partial{D}_{1}\cup\partial{D_{2}},\\
{\bf u}_{0}={\boldsymbol\varphi},&\mathrm{on}~\partial{D}.
\end{cases}
\end{equation}
Therefore, the major difficulty in analyzing the stress concentration of this problem consists in establishing the gradient estimates of thess ${\bf u}_{i}^{\alpha}$ and determining these free constants $C_{i}^{\alpha}$. 

The philosophy of the above decomposition is as follows: Although these  $C_{i}^{\alpha}$ are free constants, to be  determined later, each ${\bf u}_{i}^{\alpha}$ is the unique solution to a Dirichlet boundary problem. After the study of ${\bf u}_{i}^{\alpha}$,  we in turn use the properties of ${\bf u}_{i}^{\alpha}$ to solve $C_{i}^{\alpha}$.  However, the appearance of the small distance $\varepsilon$ results in previous theories not providing a desired bound. Intuitively, for these ${\bf u}_{i}^{\alpha}$, because the boundary data on $\partial{D}_{1}$ is different with that on $\partial{D}_{2}$, $|\nabla{\bf u}|$ will be singular when $\varepsilon$ tends to zero. It turns out to be true. In fact, we want to know the asymptotic behavior of each $|\nabla{\bf u}|$ near the origin. To this end, we will construct a family of auxiliary functions ${\bf v}_{i}^{\alpha}$ with the same boundary condition as ${\bf u}_{i}^{\alpha}$. These constructions of the auxiliary functions will play a vital role in the establishment of gradient estimates of the solution to Stokes system. Here we need them to satisfy divergence free condition, at least in $\Omega_{R}$. This is one of the main difference with the linear elasticity case \cite{BLL2}, which also causes new difficulties. To identify these auxiliary functions capturing the main singular terms of $|\nabla{\bf u}_{i}^{\alpha}|$ and $p_{i}^{\alpha}$, let us study a general boundary value problem \eqref{w1} below, which the differences $({\bf u}_{i}^{\alpha}-{\bf v}_{i}^{\alpha},p_{i}^{\alpha}-\bar{p}_{i}^{\alpha})$ verify. As mentioned before, since the Stokes system is elliptic in the sense of Douglis-Nirenberg, the regularity and estimates can be obtained from \cite{ADN1964,Solonni1966} provided the global energy is bounded. So that in the sequel we only focus on the establishment of the estimates in the narrow region $\Omega_{R}$, except the boundedness of the global energy in $\Omega$, which is fundamental important.

Compared to Lam\'e system case, there are several differences on the constructions.  We explicate them one by one in the following process of constructions.

\subsection{Construction of Auxiliary Functions and Main Estimates}\label{auxiliary}
In what follows, let us denote
\begin{align}\label{55}
\delta(x_1):=\epsilon+h_{1}(x_1)+h_{2}(x_1),\quad\mbox{for}\, |x_1|\leq 2R.
\end{align}
In order to express our idea clearly and avoid unnecessary difficulties from computation, we assume for simplicity that $h_{1}$ and $h_{2}$ are quadratic and symmetric with respect to the plane $\{x_{2}=0\}$, say, $h_1(x_1)=h_2(x_1)=\frac{1}{2}x_1^2$ for $|x_1|\leq 2R$, (see discussions after the proof of Proposition \ref{propu11} in Subsection \ref{subsec1}). We introduce the Keller-type function $k(x)\in C^{3}(\mathbb{R}^{2})$,
\begin{equation*}
k(x)=\frac{x_{2}}{\delta(x_1)},\quad\hbox{in}\ \Omega_{2R},
\end{equation*}
and $\|k(x)\|_{C^{3}(\mathbb{R}^{2}\backslash\Omega_{R})}\leq C$, and satisfies $k(x)=\frac{1}{2}$ on $\partial D_{1}$,  $k(x)=-\frac{1}{2}$ on $\partial D_{2}$, $k(x)=0$ on $\partial D$. 
Clearly,
\begin{align*}
|k(x)|\leq\frac{1}{2},\quad\mbox{and}\quad\partial_{x_1}k(x)=-\frac{2x_{1}}{\delta(x_{1})}k(x),\quad
\partial_{x_2}k(x)
=\frac{1}{\delta(x_{1})},\quad\hbox{in}\ \Omega_{2R}.
\end{align*}
In the following, we only pay attention to the estimates in $\Omega_{2R}$.

To estimate ~$\nabla{\bf u}_{1}^{1}$, we construct ${\bf v}_{1}^{1}\in C^{2}(\Omega;\mathbb R^2)$, such that 
\begin{align}\label{v11}
{\bf v}_{1}^{1}=\Big(k(x)+\frac{1}{2}\Big)\boldsymbol\psi_{1}+x_{1}\Big(k^2(x)-\frac{1}{4}\Big)\boldsymbol\psi_{2},\quad\hbox{in}\ \Omega_{2R},
\end{align}
and $\|{\bf v}_{1}^{1}\|_{C^{2}(\Omega\setminus\Omega_{R})}\leq\,C$. Compared to \cite{BLL2}, here we use the additional terms $x_{1}\Big(k^2(x)-\frac{1}{4}\Big)\boldsymbol\psi_{2}$ to modify $\Big(k(x)+\frac{1}{2}\Big)\boldsymbol\psi_{1}$ so that the modified function ${\bf v}_{1}^{1}$ become divergence free; that is,
\begin{align}\label{freev112D}
\nabla\cdot {\bf v}_{1}^{1}=\partial_{x_1}({\bf v}^{1}_{1})^{(1)}+\partial_{x_2}({\bf v}^{1}_{1})^{(2)}=0,\quad\mathrm{in}~\Omega_{2R}.
\end{align}
This turns out to be a new difficulty here and in the  subsequent constructions. However, this is only the first step. By a calculation, the first derivatives of ${\bf v}_{1}^{1}$ show the singularity of order $\frac{1}{\delta(x_1)}$, that is,
\begin{align}\label{v11-1d}
\frac{1}{C\delta(x_1)}\leq|\nabla{\bf v}_{1}^{1}|(x)\leq\,\frac{C}{\delta(x_1)}.
\end{align}
In the process of employing the energy iteration approach developed in \cite{LLBY,BLL,BLL2}, to prove \eqref{v11-1d} can capture all the singular terms of $|\nabla{\bf u}_{i}^{\alpha}|$, we have to find an appropriate $\bar{p}_1^{1}$ to make $|\mu\Delta{\bf v}_{1}^{1}-\nabla\bar{p}_1^{1}|$ as small as possible in $\Omega_{2R}$. This is another crucial issue. Indeed, their difference
\begin{align*}
{\bf w}_{1}^{1}:={\bf u}_{1}^{1}-{\bf v}_{1}^{1}\quad\mbox{and}\quad q_{1}^{1}:=p_{1}^{1}-\bar{p}_{1}^{1}
\end{align*}
verify the following general boundary value problem
\begin{align}\label{w1}
\begin{cases}
-\mu\Delta{\bf w}+\nabla q={\bf f}:=\mu\Delta{\bf v}-\nabla\bar{p},\quad&\mathrm{in}\,\Omega,\\
\nabla\cdot {\bf w}=0\quad&\mathrm{in}\,\Omega_{2R},\\
\nabla\cdot {\bf w}=-\nabla\cdot {\bf v}\quad&\mathrm{in}\,\Omega\setminus\Omega_R,\\
{\bf w}=0,\quad&\mathrm{on}\,\partial\Omega.
\end{cases}
\end{align}
By an observation, we choose $\bar{p}_1^{1}\in C^{1}(\Omega)$ such that
\begin{equation}\label{defp11}
\bar{p}_1^1=\frac{2\mu x_1}{\delta(x_1)}k(x),\quad\mbox{in}~\Omega_{2R},
\end{equation}
and  $\|\bar{p}_1^{1}\|_{C^{1}(\Omega\setminus\Omega_{R})}\leq C$. It turns out that $|\mu\Delta{\bf v}_{1}^{1}-\nabla\bar{p}_1^{1}|$ is exactly smaller than $|\mu\Delta{\bf v}_{1}^{1}|$ itself. This choice of $\bar{p}_1^{1}$ enables us to adapt the iteration approach in \cite{BLL2} to work for the Stokes flow, and to prove  that such $\nabla {\bf v}_{1}^{1}$ and $\bar{p}_1^{1}$ are the main singular terms of $\nabla {\bf u}_{1}^{1}$ and $p_1^{1}$, respectively. This is an essential difference with that in \cite{BLL2}. We would like to remark that here the choice of $\bar{p}_1^{1}$ was inspired by the recent paper \cite{AKKY}, where the authors also established the estimates of $p$ in dimension two for circular inclusions $D_{1}$ and $D_{2}$. 

Similarly, for $i=2$, we only need to replace $k(x)=\frac{x_{2}}{\delta(x_1)}$ by $\tilde{k}(x)=\frac{-x_{2}}{\delta(x_1)}$ in $\Omega_{2R}$, which satisfies $\tilde{k}(x)+\frac{1}{2}=1$ on $\partial{D}_{2}$ and $\tilde{k}(x)+\frac{1}{2}=0$ on $\partial{D}_{1}$, then use $\tilde{k}$ to construct the corresponding ${\bf v}_{2}^{1}$ and $\bar{p}_2^{1}$. So, in the sequel, the assertions hold for both $i=1$ and $i=2$, but we only prove the case for $i=1$, and omit the case for $i=2$. 

Let us denote
$$\Omega_{\delta}(x_1):=\left\{(y_1,y_{2})\big| -\frac{\varepsilon}{2}-h_{2}(y_1)<y_{2}
<\frac{\varepsilon}{2}+h_{1}(y_1),\,|y_1-x_1|<\delta \right\},$$
for $|x_1|\leq\,R$, where $\delta:=\delta(x_1)$ is defined by \eqref{55}. Define a constant independent of $\varepsilon$ as follows:
\begin{equation}\label{defqialpha}
(q_{i}^\alpha)_{R}:=\frac{1}{|\Omega_{R}\setminus\Omega_{R/2}|}\int_{\Omega_{R}\setminus\Omega_{R/2}}q_{i}^\alpha(y)\mathrm{d}y.
\end{equation} 
The following estimates hold:

\begin{prop}\label{propu11}
Let ${\bf u}_{i}^{1}\in{C}^{2}(\Omega;\mathbb R^2),~p_{i}^{1}\in{C}^{1}(\Omega)$ be the solution to \eqref{equ_v12D}. Then we have
\begin{equation*}
\|\nabla({\bf u}_{i}^{1}-{\bf v}_{i}^{1})\|_{L^{\infty}(\Omega_{\delta/2}(x_1))}\leq C,~~ \,x\in\Omega_{R},
\end{equation*}
and
\begin{equation*}
\|\nabla^2({\bf u}_{i}^{1}-{\bf v}_{i}^{1})\|_{L^{\infty}(\Omega_{\delta/2}(x_1))}+\|\nabla q_i^1\|_{L^{\infty}(\Omega_{\delta/2}(x_1))}\leq \frac{C}{\delta(x_1)},~~ \,x\in\Omega_{R}.
\end{equation*}
Consequently, in view of \eqref{v11-1d} and \eqref{v11},
\begin{align*}
\frac{1}{C\delta(x_{1})}\leq|\nabla {\bf u}_{i}^1|\leq \frac{C}{\delta(x_{1})},~|\nabla^2{\bf u}_{i}^1|\leq C\left(\frac{1}{\delta(x_{1})}+\frac{|x_1|}{\delta^2(x_{1})}\right),\quad\,x\in\Omega_{R};
\end{align*}
and 
$$|p_{i}^{1}-(q_{i}^1)_{R}|\leq\frac{C}{\varepsilon},~|\nabla p_{i}^{1}|\leq C\left(\frac{1}{\delta(x_{1})}+\frac{|x_1|}{\delta^2(x_{1})}\right),\quad\,x\in\Omega_{R}.$$
\end{prop}

To estimate ${\bf u}_{1}^{2}$, we  construct  ${\bf v}_{1}^{2}\in C^{2}(\Omega;\mathbb R^2)$, such that 
\begin{align}\label{v12}
{\bf v}_{1}^{2}=\boldsymbol\psi_{2}\Big(k(x)+\frac{1}{2}\Big)
+\frac{6}{\delta(x_{1})}\left(\boldsymbol\psi_{1}x_1+\boldsymbol\psi_{2}x_{2}\Big(\frac{2x_1^2}{\delta(x_{1})}-\frac{1}{3}\Big)\right)
\big(k^2(x)-\frac{1}{4}\big),~\mbox{in}~\Omega_{2R},
\end{align}
and $\|{\bf v}_{1}^{2}\|_{C^{2}(\Omega\setminus\Omega_{R})}\leq\,C$. This construction looks a little more complicated than that for ${\bf v}_{1}^{1}$ in \eqref{v11}. It is trivial to check that $\nabla\cdot{\bf v}_{1}^{2}=0$ in $\Omega_{2R}$, and
\begin{equation}\label{v12upper1}
|\nabla{\bf v}_{1}^{2}|\leq\,C\left(\frac{1}{\delta(x_1)}+\frac{|x_1|}{\delta^{2}(x_1)}\right),\quad\mbox{in}~\Omega_{2R},
\end{equation}
and
\begin{equation}\label{v13lower}
|\nabla{\bf v}_{1}^{2}(0,x_{2})|\geq\,\frac{1}{C}\frac{1}{\delta(x_1)},\quad\mbox{if}~|x_{2}|\leq\frac{\varepsilon}{2}.
\end{equation}
However, by a direct calculation, one can find that there is a very large (or  ``bad") term in fact is from
$$|\partial_{x_2}({\bf v}_{1}^{2})^{(1)}|=\big|\frac{12x_{1}}{\delta^{2}(x_{1})}k(x)\big|\leq\,\frac{C|x_{1}|}{\delta^2(x_1)},\quad\hbox{in}\ \Omega_{2R}.$$ For this reason, we choose a $\bar{p}_1^2\in C^{1}(\Omega)$ such that
\begin{equation}\label{p12}
\bar{p}_1^2=-\frac{3\mu}{\delta^2(x_1)}+\frac{18\mu}{\delta(x_1)}\left(\frac{2x_{1}^{2}}{\delta(x_{1})}-\frac{1}{3}\right)k^{2}(x),\quad\mbox{in}~\Omega_{2R}.
\end{equation}
and  $\|\bar{p}_1^2\|_{C^{1}(\Omega\setminus\Omega_{R})}\leq C$, to make $|\mu\Delta{\bf v}_{1}^{2}-\nabla\bar{p}_1^{2}|$ be as small as possible, and prove Proposition \ref{propu12} below. It turns out that the estimate for $|\nabla{\bf u}_{1}^{2}|$ is the hardest term to control among all estimates of $|\nabla{\bf u}_{1}^{\alpha}|$, because the term $\frac{|x_1|}{\delta^2(x_1)}$ is of order $\varepsilon^{-3/2}$, which is obviously larger than other cases of $\alpha$. 

\begin{prop}\label{propu12}
Let ${\bf u}_{i}^{2}\in{C}^{2}(\Omega;\mathbb R^2),~p_{i}^{2}\in{C}^{1}(\Omega)$ be the solution to \eqref{equ_v12D}. Then we have
\begin{equation*}
\|\nabla({\bf u}_{i}^{2}-{\bf v}_{i}^{2})\|_{L^{\infty}(\Omega_{\delta/2}(x_1))}\leq \frac{C}{\sqrt{\delta(x_1)}},\quad\,x\in\Omega_{R},
\end{equation*}
and 
\begin{equation*}
\|\nabla^2({\bf u}_{i}^{2}-{\bf v}_{i}^{2})\|_{L^{\infty}(\Omega_{\delta/2}(x_1))}+\|\nabla q_i^2\|_{L^{\infty}(\Omega_{\delta/2}(x_1))}\leq C\left(\frac{1}{\delta(x_1)}+\frac{|x_1|}{\delta^2(x_{1})}\right),\quad\,x\in\Omega_{R}.
\end{equation*}
Consequently, in view of \eqref{v12upper1}--\eqref{p12},
\begin{align*}
\frac{1}{C\delta(x_1)}\leq|\nabla {\bf u}_{1}^2|\leq C\left(\frac{1}{\delta(x_1)}+\frac{|x_1|}{\delta^2(x_{1})}\right),~|\nabla^2 {\bf u}_{1}^2|\leq C\left(\frac{1}{\delta^2(x_1)}+\frac{|x_1|}{\delta^3(x_{1})}\right),~\,x\in\Omega_{R};
\end{align*}
and
$$|p_{i}^{2}-(q_{i}^2)_{R}|\leq\,\frac{C}{\varepsilon^2},~|\nabla p_{i}^{2}|\leq\,C\left(\frac{1}{\delta^2(x_1)}+\frac{|x_1|}{\delta^3(x_{1})}\right),\quad\,x\in\Omega_{R}.$$
\end{prop}

For ${\bf u}_{1}^{3}$, we construct  ${\bf v}_{1}^{3}\in C^{2}(\Omega;\mathbb R^2)$, such that 
\begin{align}\label{v13}
{\bf v}_{1}^{3}=\boldsymbol\psi_{3}\Big(k(x)+\frac{1}{2}\Big)
+\begin{pmatrix}
1-\frac{4x_{1}^2}{\delta(x_{1})}-5x_{2}k(x)\\\\
2x_1k(x)\left(2-\frac{4x_{1}^{2}}{\delta(x_{1})}-3x_{2}k(x)\right)
\end{pmatrix}
\big(k^2(x)-\frac{1}{4}\big),\quad \mbox{in}~\Omega_{2R},
\end{align}
and choose $\bar{p}_1^3\in C^{1}(\Omega)$ such that
$$\bar{p}_1^3=\frac{2\mu x_1}{\delta^2(x_1)}+\frac{12\mu x_1 }{\delta(x_1)}\Big(1-\frac{2 x_1^2}{\delta(x_1)}\Big)k^{2}(x),\quad\mbox{in}~\Omega_{2R}.$$

\begin{prop}\label{propu13}
Let ${\bf u}_{i}^{3}\in{C}^{2}(\Omega;\mathbb R^2),~p_{i}^{3}\in{C}^{1}(\Omega)$ be the solution to \eqref{equ_v12D}. Then 
\begin{equation*}
\|\nabla({\bf u}_{i}^{3}-{\bf v}_{i}^{3})\|_{L^{\infty}(\Omega_{\delta/2}(x_1))}\leq C,\quad x\in\Omega_{R},
\end{equation*}
and 
\begin{equation*}
\|\nabla^2({\bf u}_{i}^{3}-{\bf v}_{i}^{3})\|_{L^{\infty}(\Omega_{\delta/2}(x_1))}+\|\nabla q_{i}^{3}\|_{L^{\infty}(\Omega_{\delta/2}(x_1))}\leq \frac{C}{\delta(x_{1})},\quad x\in\Omega_{R}.
\end{equation*}
Consequently,
\begin{align*}
|\nabla {\bf u}_{i}^3|\leq \frac{C}{\delta(x_1)},~|\nabla^2 {\bf u}_{i}^3|\leq\frac{C}{\delta^2(x_1)},\quad\,x\in\Omega_{R};
\end{align*}
and
\begin{align*}
|p_{i}^{3}-(q_{i}^3)_{R}|\leq\,\frac{C}{\varepsilon^{3/2}},~|\nabla p_{i}^{3}|\leq\frac{C}{\delta^{2}(x_{1})},\quad\,x\in\Omega_{R}.
\end{align*}
\end{prop}

Since ${\bf u}_0=0$ on $\partial D_1\cup\partial D_2$ and ${\bf u}_{1}^{\alpha}+{\bf u}_{2}^{\alpha}={\boldsymbol\psi}_{\alpha}$ on $\partial D_1\cup\partial D_2$, the result in \cite{LLBY} for elliptic systems assures the boundedness of their gradients  (indeed, theorem 1.1 there applies directly to ${\bf u}_0$ and ${\bf u}_{1}^{\alpha}+{\bf u}_{2}^{\alpha}$). The proof is immediate and so is omitted. We list the assertion here for convenience.

\begin{prop}\label{propu0}
Let ${\bf u}_0,{\bf u}_{i}^{\alpha}\in{C}^{2,\gamma}(\Omega;\mathbb R^2),~p_0,p_{i}^{\alpha}\in{C}^{1,\gamma}(\Omega)$ be the solution to \eqref{equ_v12D} and \eqref{equ_v32D}. Then
\begin{equation*}
\|\nabla^{k_1}{\bf u}_0\|_{L^{\infty}(\Omega)}+\|\nabla^{k_1}({\bf u}_{1}^{\alpha}+{\bf u}_{2}^{\alpha})\|_{L^{\infty}(\Omega)}\leq C,\quad k_1=1,2;
\end{equation*}
and
\begin{equation*}
\|\nabla^{k_2} p_0\|_{L^{\infty}(\Omega)}+\|\nabla^{k_2}(p_{1}^{\alpha}+p_{2}^{\alpha})\|_{L^{\infty}(\Omega)}\leq C,\quad k_2=0,1.
\end{equation*}
\end{prop}

\begin{remark}
Consequently, in view of \eqref{defsigma}, we have 
\begin{equation*}
\|\sigma[{\bf u}_{1}^{\alpha}+{\bf u}_{2}^{\alpha},p_{1}^{\alpha}+p_{2}^{\alpha}]\|_{L^{\infty}(\Omega)}\leq C,\quad \|\sigma[{\bf u}_0,p_0]\|_{L^{\infty}(\Omega)}\leq C,
\end{equation*}
which implies that the stress tensor does not blow up when the boundary data takes the same value on $\partial D_1$ and $\partial D_2$. This extends the results in \cite[Theorems 3.4, 4.1, and 5.1]{AKKY} about  circular cylinders to  more general strictly convex case.
\end{remark}

Based on the above estimates of $\nabla{\bf u}_{i}^{\alpha}$, we try to make use of them to solve the six free constants $C_{i}^{\alpha}$ introduced in \eqref{ud}. From the forth line of \eqref{Stokessys} and the decompositions \eqref{ud}, let us study the following linear system of $C_{i}^{\alpha}$:
\begin{equation}\label{equ-decompositon}
\sum_{i=1}^2\sum\limits_{\alpha=1}^{3} C_{i}^{\alpha}
\int_{\partial D_j}{\boldsymbol\psi}_\beta\cdot\sigma[{\bf u}_{i}^\alpha,p_{i}^{\alpha}]\nu
+\int_{\partial D_j}{\boldsymbol\psi}_\beta\cdot\sigma[{\bf u}_{0},p_{0}]\nu=0,~~\beta= 1,2,3,
\end{equation}
where $j=1,2$. The difficulty is mainly from that the coefficients are completely determined by the boundary integrations involving ${\bf u}_{i}^{\alpha}$ and $p_{i}^{\alpha}$. First of all, the trace theorem can ensure the boundedness of $C_i^{\alpha}$. However, the differences of $|C_1^{\alpha}-C_2^{\alpha}|$ may be some infinitesimal quantities, see Proposition \ref{lemCialpha} below. As a matter of fact, whether they can be solved or estimated depends entirely on how good estimates of $|\nabla{\bf u}_{1}^{\alpha}|$ can be obtained above, including both the upper and the lower bound estimates. Therefore, the establishment of the estimates of ${\bf u}_{i}^{\alpha}$, $p_{i}^{\alpha}$ and $C_1^{\alpha}$ in this paper is actually a complementary unity.  

\begin{prop}\label{lemCialpha}
Let $C_{i}^{\alpha}$ be defined in \eqref{ud}. Then
$$|C_i^{\alpha}|\leq\,C,\quad\,i=1,2,~\alpha=1,2,3,$$
and
\begin{equation*}
|C_1^1-C_2^1|\leq C\sqrt{\varepsilon},\quad |C_1^2-C_2^2|\leq C\varepsilon^{3/2},\quad|C_1^3-C_2^3|\leq C\sqrt{\varepsilon}.
\end{equation*}
\end{prop}

\subsection{The Completion of the Proof of Theorem \ref{mainthm2D}}\label{prfthm}
Now let us complete the proof of Theorem \ref{mainthm2D}.
\begin{proof}[Proof of Theorem \ref{mainthm2D}]
We  rewrite \eqref{ud} as
\begin{equation*}
\nabla{{\bf u}}=\sum_{\alpha=1}^{3}\left(C_{1}^{\alpha}-C_{2}^{\alpha}\right)\nabla{\bf u}_{1}^{\alpha}
+\nabla {\bf u}_{b},~\mbox{and}~ 
p=\sum_{\alpha=1}^{3}\left(C_{1}^{\alpha}-C_{2}^{\alpha}\right)p_{1}^{\alpha}+p_{b},\quad\mbox{in}~\Omega,
\end{equation*}
where
\begin{equation*}
{\bf u}_{b}:=\sum_{\alpha=1}^{3}C_{2}^{\alpha}({\bf u}_{1}^{\alpha}+{\bf u}_{2}^{\alpha})+{\bf u}_{0},\quad p_{b}:=\sum_{\alpha=1}^{3}C_{2}^{\alpha}(p_{1}^{\alpha}+p_{2}^{\alpha})+p_{0}.
\end{equation*}

(i) By virtue of  Propositions  \ref{propu11}, and \ref{propu12}--\ref{lemCialpha}, we have
\begin{align}\label{upper-u2D}
|\nabla{{\bf u}}(x)|&\leq\sum_{\alpha=1}^{3}\left|(C_{1}^{\alpha}-C_{2}^{\alpha}\right)\nabla{\bf u}_{1}^{\alpha}(x)|+C\nonumber\\
&\leq\,\frac{C\sqrt{\varepsilon}}{\delta(x_{1})}+C\varepsilon^{3/2}\left(\frac{1}{\delta(x_1)}+\frac{|x_1|}{\delta^2(x_{1})}\right)+\frac{C\sqrt{\varepsilon}}{\delta(x_{1})}\leq C\frac{\sqrt{\varepsilon}+|x_1|}{\varepsilon+x_1^2}.
\end{align}
For $|p|$, we first take $q_{R}=\sum_{\alpha=1}^{3}(C_1^\alpha-C_2^\alpha)(q_{1}^\alpha)_{R}$, where $(q_{1}^\alpha)_{R}$ is defined in \eqref{defqialpha}. Then for any $x\in\Omega_R$, by using Propositions \ref{propu11},  \ref{propu12}--\ref{lemCialpha}, we derive 
\begin{align}\label{upper-p2D}
|p(x)-q_{R}|&\leq\sum_{\alpha=1}^{3}\left|(C_{1}^{\alpha}-C_{2}^{\alpha}\right)(p_{1}^{\alpha}(x)-(q_{1}^\alpha)_{R})|+C\nonumber\\
&\leq\,\frac{C\sqrt{\varepsilon}}{\varepsilon}+\frac{C\varepsilon^{3/2}}{\varepsilon^2}+\frac{C\sqrt{\varepsilon}}{\varepsilon^{3/2}}\leq \frac{C}{\varepsilon}.
\end{align}

(ii) In view of Propositions  \ref{propu11}, and \ref{propu12}--\ref{lemCialpha} again, we obtain 
\begin{align*}
|\nabla^2{{\bf u}}(x)|&\leq\sum_{\alpha=1}^{3}\left|(C_{1}^{\alpha}-C_{2}^{\alpha}\right)\nabla^2{\bf u}_{1}^{\alpha}(x)|+C\nonumber\\
&\leq\,\frac{C\sqrt{\varepsilon}|x_1|}{\delta^2(x_{1})}+C\varepsilon^{3/2}\left(\frac{1}{\delta^2(x_1)}+\frac{|x_1|}{\delta^3(x_{1})}\right)+\frac{C\sqrt{\varepsilon}}{\delta^2(x_{1})}\leq \frac{C}{(\varepsilon+x_1^2)^{3/2}}.
\end{align*}
Similarly, 
\begin{align*}
|\nabla p(x)|&\leq\sum_{\alpha=1}^{3}\left|(C_{1}^{\alpha}-C_{2}^{\alpha}\right)\nabla p_{1}^{\alpha}(x)|+C\leq \frac{C}{(\varepsilon+x_1^2)^{3/2}}.
\end{align*}
Thus, Theorem \ref{mainthm2D} is proved.
\end{proof}

\section{Preliminaries}\label{sec2}

In this section we recall and prove some basic results about Stokes systems. The results are concerning the energy estimates, $W^{2,\infty}$ estimates and a Caccioppoli-type inequality, which are needed to apply the iteration technique in the subsequent Sections.

\subsection{Some Basic Results}

A basic result for the theoretical and numerical analysis of the Stokes systems in a bounded domain $\Omega\subset\mathbb{R}^{d}$ is the existence of a continuous right inverse of the divergence as an operator from the Sobolev space $H_0^{1}(\Omega;\mathbb R^d)$ into the space $L_{0}^2(\Omega)$ of functions in $L^2(\Omega)$ with vanishing mean value, where $d\geq2$. For a Lipschitz domain $\Omega$, this result is proved by employing compactness arguments, see, for example, \cite{T1984}.  
\begin{lemma}\label{lemmaf0}
Let $\Omega$ be a bounded Lipschitz domain in $\mathbb R^d$. Given $f\in L^2(\Omega)$ with $\int_{\Omega}f=0$. Then there exists a function ${\boldsymbol\phi}\in H_0^{1}(\Omega;\mathbb R^d)$ such that
$$\Div{\boldsymbol\phi}=f\quad\mbox{in}~\Omega,$$
and 
\begin{equation}\label{estphif}
\|{\boldsymbol\phi}\|_{H_{0}^{1}(\Omega)}\leq C\|f\|_{L^2(\Omega)},
\end{equation}
where $C>0$ is a constant depending only on $d$ and $\mathrm{diam}(\Omega)$.
\end{lemma}

By virtue of Lemma \ref{lemmaf0}, we have the following estimate for the pressure term $q$.

\begin{lemma}\label{corowq}
Let $({\bf w},q)$ be the solution to 
\begin{equation}\label{eqbfwq}
-\mu\Delta{\bf w}+\nabla q={\bf g},\quad\mbox{in}~\Omega,
\end{equation}
where ${\bf g}\in L^2(\Omega;\mathbb R^d)$. Then
\begin{align}\label{qL2}
\int_{\Omega}|q-q_{_\Omega}|^2\mathrm{d}x\leq C\int_{\Omega}|\nabla{\bf w}|^2\mathrm{d}x+C\int_{\Omega}|{\bf g}|^2\mathrm{d}x,
\end{align}
where $q_{_\Omega}:=\frac{1}{|\Omega|}\int_{\Omega}q\mathrm{d}x$, and $C>0$ is a constant depending only on $d,\mu$ and $\mathrm{diam}(\Omega)$.
\end{lemma}

\begin{proof}
By Lemma \ref{lemmaf0}, there exists a function ${\boldsymbol\phi}\in H_0^{1}(\Omega;\mathbb R^d)$ such that
$$\Div{\boldsymbol\phi}=q-q_{_\Omega},\quad\mbox{in}~\Omega,$$
with estimate
\begin{equation}\label{estphiq}
\|{\boldsymbol\phi}\|_{H_{0}^{1}(\Omega)}^2\leq C_0\|q-q_{_\Omega}\|_{L^2(\Omega)}^2,
\end{equation}
where the constant $C_{0}$ is fixed now. Now we use this ${\boldsymbol\phi}$ as a test function to \eqref{eqbfwq}, and apply the integration by parts, 
\begin{equation*}
\mu\int_{\Omega}\nabla{\bf w}\cdot\nabla{\boldsymbol\phi}\mathrm{d}x-\int_{\Omega}|q-q_{_\Omega}|^2\mathrm{d}x=\int_{\Omega}{\bf g}{\boldsymbol\phi}\mathrm{d}x.
\end{equation*}
Employing Young's inequality and taking into account \eqref{estphiq}, we have
\begin{align*}
\int_{\Omega}|q-q_{_\Omega}|^2\mathrm{d}x&\leq \frac{1}{2C_0}\int_{\Omega}|\nabla{\boldsymbol\phi}|^2\mathrm{d}x+C\int_{\Omega}|\nabla{\bf w}|^2\mathrm{d}x+\frac{1}{2C_0}\int_{\Omega}|{\boldsymbol\phi}|^2\mathrm{d}x+C\int_{\Omega}|{\bf g}|^2\mathrm{d}x\\
&\leq \frac{1}{2}\int_{\Omega}|q-q_{_\Omega}|^2\mathrm{d}x+C\int_{\Omega}|\nabla{\bf w}|^2\mathrm{d}x+C\int_{\Omega}|{\bf g}|^2\mathrm{d}x,
\end{align*}
which implies \eqref{qL2}. The proof is completed.
\end{proof}

An immediate consequence of Lemma \ref{corowq} is obtained by a rescaling argument.
\begin{corollary}\label{rmkmean}
If $\Omega$ is a ball $B_r$, then estimate \eqref{estphif} becomes
\begin{equation*}
\|{\boldsymbol\phi}\|_{L^{2}(B_r)}+r\|\nabla{\boldsymbol\phi}\|_{L^{2}(B_r)}\leq Cr\|f\|_{L^2(B_r)},
\end{equation*}
and \eqref{qL2} becomes
\begin{align*}
\int_{B_r}|q-q_{_{B_r}}|^2\mathrm{d}x\leq C\int_{B_r}|\nabla{\bf w}|^2\mathrm{d}x+Cr^2\int_{B_r}|{\bf g}|^2\mathrm{d}x,
\end{align*}
where $C$ is independent of $r$.
\end{corollary}
\begin{proof}
Indeed, if $({\bf w},q)$ satisfies
$$-\mu\Delta{\bf w}+\nabla q={\bf g},\quad\mbox{in}~B_{r},
$$
then by rescaling, $(\tilde{\bf w},\tilde{q})$ satisfies
\begin{equation}\label{tilde_w}
-\mu\Delta\tilde{\bf w}+\nabla \tilde{q}=\tilde{\bf g},\quad\mbox{in}~B_{1},
\end{equation}
where
$$\tilde{\bf w}(y)={\bf w}(ry),\quad\tilde{q}(y)=rq(ry),\quad\,\tilde{\bf g}(y)=r^{2}{\bf g}(ry).$$
Employing \eqref{qL2} for \eqref{tilde_w} completes the proof.
\end{proof}

Before proceeding to our result, we must recall another well-known result for Stokes system. The following $L^{q}$-estimate for Stokes flow in a bounded domain, with partially vanishing boundary data, can be found in \cite[Theorem IV.5.1]{G2011}.

\begin{theorem}\label{thmWmq}
Let $\Omega$ be an arbitrary domain in $\mathbb R^d$, $d\geq2$, with a boundary portion $\sigma$ of class $C^{m+2}$, $m\geq0$. Let $\Omega_0$ be any bounded subdomain of $\Omega$ with $\partial\Omega_0\cap\partial\Omega=\sigma$. Further, let
\begin{align*}
{\bf u}\in W^{1,q}(\Omega_0), \quad p\in L^q(\Omega_0),\quad 1<q<\infty,
\end{align*}
be such that
\begin{align*}
(\nabla{\bf u},\nabla{\boldsymbol\psi})&=-\langle{\bf f},{\boldsymbol\psi}\rangle+(p,\nabla\cdot{\boldsymbol\psi}),\quad\mbox{for~all}~{\boldsymbol\psi}\in C_0^\infty(\Omega_0),\\
({\bf u},\nabla\varphi)&=0,\quad \mbox{for~all}~\varphi\in C_0^\infty(\Omega_0),\\
{\bf u}&=0,\quad \mbox{at}~\sigma.
\end{align*}
Then, if ${\bf f}\in W^{m,q}(\Omega_0)$, we have
$${\bf u}\in W^{m+2,q}(\Omega'),\quad p\in W^{m+1,q}(\Omega'),$$
for any $\Omega'$ satisfying 
\begin{enumerate}
\item $\Omega'\subset\Omega$,
\item $\partial\Omega'\cap\partial\Omega$ is a strictly interior subregion of $\sigma$.
\end{enumerate}
Finally, the following estimate holds
\begin{align*}
\|{\bf u}\|_{W^{m+2,q}(\Omega')}+\|p\|_{W^{m+1,q}(\Omega')}\leq C\left(\|{\bf f}\|_{W^{m,q}(\Omega_0)}+\|{\bf u}\|_{W^{1,q}(\Omega_0)}+\|p\|_{L^{q}(\Omega_0)}\right),
\end{align*}
where $C=C(d,m,q,\Omega',\Omega_0)$.
\end{theorem}

Now we consider the domain 
\begin{equation}\label{defQ}
Q_{r}=\left\{y\in\mathbb{R}^{d}~\Big|~-\frac{\varepsilon}{2\delta}-\frac{1}{\delta}h_{2}(\delta\,y'+z')<y_{d}
<\frac{\varepsilon}{2\delta}+\frac{1}{\delta}h_{1}(\delta\,y'+z'),~|y'|<r\right\},
\end{equation}
where $\delta=\delta(z')=\varepsilon+|z'|^{2}$ and $h_{1}(x')=h_{2}(x')=\frac{1}{2}|x'|^{2}$. Then $Q_{1}$ has a nearly unit size.  Denote its top and bottom boundaries by $\hat\Gamma_{1}^{+}$ and $\hat\Gamma_{1}^{-}$. Applying a bootstrap argument yields the following Proposition from Theorem \ref{thmWmq}.

\begin{prop}\label{lemWG}
Let $q\in(1,\infty)$ and let ${\mathcal W}\in W^{1,q}(Q_1)$ and $\mathcal{G}\in L^q(Q_1)$ be a weak solution of
\begin{align}\label{eqWG}
\begin{cases}
-\nabla\cdot\sigma[{\mathcal W},\mathcal{G}]=\mathcal{F},\quad&\mathrm{in}\,Q_1\\
\nabla\cdot {{\mathcal W}}=0,\quad&\mathrm{in}\,Q_1,\\
{{\mathcal W}}=0,\quad&\mathrm{on}\,\sigma:=\hat\Gamma_{1}^{+}\cup\hat\Gamma_{1}^{-}.
\end{cases}
\end{align}
Then if $\mathcal{F}\in L^\infty(Q_1)$, we have 
\begin{align}\label{energy0Q1/2}
\|\nabla\mathcal{W}\|_{L^{\infty}(Q_{1/2})}+\|\mathcal{G}-\mathcal{G}_{1}\|_{L^{\infty}(Q_{1/2})}\leq C\big(\|\mathcal{F}\|_{L^{\infty}(Q_1)}+\|\nabla\mathcal{W}\|_{L^{2}(Q_1)}\big),
\end{align}
where $\mathcal{G}_{1}=\frac{1}{|Q_1|}\int_{Q_1}\mathcal{G}$; and for $m\geq1$, 
\begin{align}\label{energyQ1/2}
\|\nabla^{m+1}\mathcal{W}\|_{L^{\infty}(Q_{1/2})}+\|\nabla^{m}\mathcal{G}\|_{L^{\infty}(Q_{1/2})}\leq C\left(\sum_{j=0}^{m}\|\nabla^{j}\mathcal{F}\|_{L^{\infty}(Q_1)}+\|\nabla\mathcal{W}\|_{L^{2}(Q_1)}\right).
\end{align}
\end{prop}

\begin{proof}
We apply Theorem \ref{thmWmq} to \eqref{eqWG}, starting with $m=0$ (since $\partial{D}_{i}$ is of class $C^{3}$) and $q=2$, then the following holds:
\begin{align*}
\|\mathcal{W}\|_{W^{2,2}(Q_{2/3})}+&\|\mathcal{G}-\mathcal{G}_{1}\|_{W^{1,2}(Q_{2/3})}\\
&\leq C\left(\|\mathcal{F}\|_{L^{\infty}(Q_1)}+\|\mathcal{W}\|_{W^{1,2}(Q_1)}+\|\mathcal{G}-\mathcal{G}_{1}\|_{L^{2}(Q_1)}\right),
\end{align*}
where $\mathcal{G}_{1}=\frac{1}{|Q_1|}\int_{Q_1}\mathcal{G}$. By virtue of Lemma \ref{corowq} and Poincar\'e inequality (since $\mathcal{W}=0$ on $\sigma$), we get
\begin{align*}
\|\mathcal{W}\|_{W^{2,2}(Q_{2/3})}+&\|\mathcal{G}-\mathcal{G}_{1}\|_{W^{1,2}(Q_{2/3})}\leq C\left(\|\mathcal{F}\|_{L^{\infty}(Q_1)}+\|\nabla\mathcal{W}\|_{L^{2}(Q_1)}\right).
\end{align*}
Employing Sobolev embedding theorem, we have $W^{1,2}(Q_1)\subset L^{2^*}(Q_1)$, where $2^*=\frac{2d}{d-2}>d$ if $d>2$, and $2^*>2$ if $d=2$. Moreover, $W_{0}^{1,q}(Q_1)$ is compactly embedded in $C^{0}(\bar{Q}_1)$ if $q>d$. Hence, using Theorem \ref{thmWmq} again for $m=0$ and applying bootstrap arguments for some $q>d$, we obtain \eqref{energy0Q1/2}. 
For $m\geq 1$, by using Theorem \ref{thmWmq} to \eqref{eqWG} with a
slightly smaller domain, we obtain 
\begin{align*}
\|\mathcal{W}\|_{W^{m+2,q}(Q_{1/2})}+&\|\mathcal{G}-\mathcal{G}_{1}\|_{W^{m+1,q}(Q_{1/2})}\\
&\leq C\left(\|\mathcal{F}\|_{W^{m,q}(Q_{2/3})}+\|\mathcal{W}\|_{W^{1,q}(Q_{2/3})}+\|\mathcal{G}-\mathcal{G}_{1}\|_{L^{q}(Q_{2/3})}\right).
\end{align*}
Then by using Poincar\'e inequality and \eqref{energy0Q1/2}, we have 
\begin{align*}
\|\nabla^{m+1}\mathcal{W}\|_{W^{1,q}(Q_{1/2})}+\|\nabla^{m}\mathcal{G}\|_{W^{1,q}(Q_{1/2})}\leq C\left(\|\mathcal{F}\|_{W^{m,q}(Q_{1})}+\|\nabla\mathcal{W}\|_{L^{2}(Q_{1})}\right).
\end{align*}
Together with Sobolev embedding theorem for $q>d$, we obtain \eqref{energyQ1/2}. 
\end{proof}

\subsection{Local $W^{1,\infty}$ Estimates}
Thanks to the following change of variables
\begin{equation}\label{changeofvariant}
\left\{
\begin{aligned}
&x'-z'=\delta(z') y',\\
&x_d=\delta(z') y_{d},
\end{aligned}
\right.
\end{equation}
we transform $\Omega_{\delta(z')}(z')$ into a nearly unit size domain 
$Q_{1}$ defined by \eqref{defQ}. Recall that each ${\bf u}_{i}^{\alpha}$ is incompressible, that is,
$$\nabla\cdot {\bf u}_{i}^{\alpha}=0,\quad\mathrm{in}~\Omega,$$
and meanwhile each ${\bf v}_{i}^{\alpha}$ that we constructed in Section \ref{auxiliary} is also incompressible in $\Omega_{2R}$ with $\|{\bf v}_{i}^{\alpha}\|_{C^{2}(\Omega\setminus\Omega_{R})}\leq\,C$.
Then their differences
\begin{align*}
{\bf w}_{i}^{\alpha}:={\bf u}_{i}^{\alpha}-{\bf v}_{i}^{\alpha},\quad\mbox{and}\quad q_{i}^{\alpha}:=p_{i}^{\alpha}-\bar{p}_{i}^{\alpha},\quad\,i=1,2,~\alpha=1,2,3,
\end{align*}
verify the boundary value problem \eqref{w1}. By applying Proposition \ref{lemWG}, the following Proposition holds in any dimension:
\begin{prop}\label{lemWG2}
Let $({\bf w},q)$ be the solution to \eqref{w1}. Then the following estimates hold
\begin{align}\label{W2pstokes}
\|\nabla {\bf w}\|_{L^{\infty}(\Omega_{\delta/2}(z'))}&+\|q-q_{\delta;z'}\|_{L^{\infty}(\Omega_{\delta/2}(z'))}\nonumber\\
&\leq
\,C\left(\frac{1}{\delta^{d/2}(z')}\|\nabla {\bf w}\|_{L^{2}(\Omega_{\delta}(z'))}+\delta(z')\|{\bf f}\|_{L^{\infty}(\Omega_{\delta}(z'))} \right),
\end{align}
where $\delta=\delta(z')$ and $q_{\delta;z'}=\frac{1}{|\Omega_{\delta}(z')|}\int_{\Omega_{\delta}(z')}q$; and for $m\geq1$,
\begin{align}\label{Wmpstokes}
&\|\nabla^{m+1}{\bf w}\|_{L^{\infty}(\Omega_{\delta/2}(z'))}+\|\nabla^{m}q\|_{L^{\infty}(\Omega_{\delta/2}(z'))}\nonumber\\
&\leq
\frac{C}{\delta^{m+1}(z')}\left(\delta^{1-\frac{d}{2}}(z')\|\nabla {\bf w}\|_{L^{2}(\Omega_{\delta}(z'))}+\sum_{j=0}^{m}\delta^{2+j}(z')\|\nabla^{j}{\bf f}\|_{L^{\infty}(\Omega_{\delta}(z'))} \right).
\end{align}
\end{prop}

\begin{proof}
Using change of variables
 \eqref{changeofvariant}, let us define
\begin{equation*}
\mathcal{W}(y', y_{d}):={\bf w}(\delta\,y'+z',\delta\,y_{d}),\quad \mathcal{V}(y', y_{d}):={\bf v}(\delta\,y'+z',\delta\,y_{d}),\quad y\in{Q}_{1},
\end{equation*}
and
\begin{equation*}
\mathcal{G}(y', y_{d}):=\delta q(\delta\,y'+z',\delta\,y_{d}),\quad \bar {\mathcal{P}}(y', y_{d}):=\delta\bar{p}(\delta\,y'+z',\delta\,y_{d}),\quad y\in{Q}_{1}.
\end{equation*}
It follows from \eqref{w1} that
\begin{align*}
\begin{cases}
-\mu\Delta\mathcal{W}+\nabla(\mathcal{G}-\mathcal{G}_{1})=\mu\Delta\mathcal{V}-\nabla\bar{\mathcal{P}}=:\mathcal{F},\quad&\mathrm{in}\,Q_1\\
\nabla\cdot \mathcal{W}=0,\quad&\mathrm{in}\,Q_1,\\
\mathcal{W}=0,\quad&\mathrm{on}\,\hat\Gamma_{1}^{+}\cup\hat\Gamma_{1}^{-}.
\end{cases}
\end{align*}
Then by Lemma \ref{lemWG} and rescaling back, there holds \eqref{W2pstokes} and \eqref{Wmpstokes}.
\end{proof}

\subsection{Global Energy Estimates}
To apply Proposition \ref{lemWG2}, we need to first establish the energy estimates $\|\nabla {\bf w}\|_{L^{2}(\Omega_{\delta}(z'))}$. To this aim, it requires to prove the global energy is bounded in the whole domain $\Omega$, under certain assumptions on the right hand side of the equation in \eqref{w1}. The following lemma holds in any dimension $d\geq2$ and will be used to directly prove the boundedness of the energy of ${\bf w}_{i}^{\alpha}$ for the case $\alpha=1$. While, for other cases, more technique is presented in Subsection \ref{subsec2} and Subsection \ref{subsec3}.

\begin{lemma}\label{lemmaenergy}
Let $({\bf w},q)$ be the solution to \eqref{w1}. Then if ${\bf v}\in C^{2}(\Omega;\mathbb R^d)$ and $\bar{p}\in C^{1}(\Omega)$ satisfy
\begin{equation}\label{estv113D3}
\|{\bf v}\|_{C^{2}(\Omega\setminus\Omega_R)}\leq\,C,\quad\|\bar{p}\|_{C^1(\Omega\setminus\Omega_{R})}\leq C,
\end{equation}
and 
\begin{align}\label{int-fw}
\Big| \int_{\Omega_{R}}\sum_{j=1}^{d}{\bf f}^{(j)}{\bf w}^{(j)}\mathrm{d}x\Big|\leq\,C\left(\int_{\Omega}|\nabla {\bf w}|^2\mathrm{d}x\right)^{1/2},
\end{align}
then
\begin{align}\label{estw11case2}
\int_{\Omega}|\nabla {\bf w}|^{2}\mathrm{d}x\leq C.
\end{align}
\end{lemma}

\begin{proof}
Suppose $({\bf w},q)$ is the solution to \eqref{w1}, then it also verifies \begin{align}\label{w12}
-\mu\Delta{\bf w}+\nabla(q-q_{_{R,out}})=\mu\Delta{\bf v}-\nabla\bar{p},
\end{align}
where $q_{_{R,out}}:=\frac{1}{|\Omega\backslash\Omega_{R}|}\int_{\Omega\backslash\Omega_{R}}q\mathrm{d}x$. Multiplying equation \eqref{w12} by ${\bf w}$, and integrating by parts, yields
\begin{align*}
\mu\int_{\Omega}|\nabla{\bf w}|^2\mathrm{d}x-\int_{\Omega}(q-q_{_{R,out}})\nabla\cdot{\bf w}\mathrm{d}x=\int_{\Omega}{\bf f}\cdot{\bf w}\mathrm{d}x.
\end{align*}
In view of the second and third lines in \eqref{w1}, we have
\begin{align}\label{energy}
\mu\int_{\Omega}|\nabla{\bf w}|^2\mathrm{d}x=-\int_{\Omega\setminus\Omega_R}(q-q_{_{R,out}})\nabla\cdot{\bf v}\mathrm{d}x+\int_{\Omega}{\bf f}\cdot{\bf w}\mathrm{d}x.
\end{align}

By H\"{o}lder's inequality and \eqref{estv113D3}, 
\begin{align}\label{energy2}
\left|\int_{\Omega\setminus\Omega_R}(q-q_{_{R,out}})\nabla\cdot{\bf v}\mathrm{d}x\right|\leq\,C\left(\int_{\Omega\setminus\Omega_R}|q-q_{_{R,out}}|^{2}\mathrm{d}x\right)^{1/2}.
\end{align}
Applying Lemma \ref{corowq} to \eqref{w12} in $\Omega\backslash\Omega_{R}$ and using \eqref{estv113D3}, \begin{equation}\label{estaverageq1}
\int_{\Omega\backslash\Omega_{R}}
\left|q-q_{_{R,out}}\right|^{2}\mathrm{d}x\leq C\int_{\Omega\backslash\Omega_{R}}|\nabla{\bf w}|^2\mathrm{d}x+C.
\end{equation}
As far as the second term on the right hand side of \eqref{energy} is concerned, by assumptions \eqref{estv113D3} and \eqref{int-fw}, and using Poincar\'e inequality, we deduce
\begin{align}\label{bddnblaw1}
\Big|\int_{\Omega}{\bf f}\cdot{\bf w}\mathrm{d}x\Big|
&\leq C\Big|\int_{\Omega_{R}}{\bf f}\cdot{\bf w}\mathrm{d}x\Big|
+C\left(\int_{\Omega\backslash\Omega_{R}}|\nabla {\bf w}|^2\mathrm{d}x\right)^{1/2}\leq\,C\left(\int_{\Omega}|\nabla {\bf w}|^2\mathrm{d}x\right)^{1/2}.
\end{align}
Thus, combining \eqref{energy}--\eqref{bddnblaw1}, we conclude that \eqref{estw11case2} holds. 
\end{proof}

\begin{remark}
We remark that the boundedness of the global energy is an important step to employ our iteration approach to obtain the boundedness of the local energy in a small subregion of the narrow region. We begin by the following Caccioppoli-type inequality to establish the local energy estimates. More details, see Section \ref{sec_estimate2D} below.
\end{remark}

\subsection{A Caccioppoli-type Inequality}
We need the following simple lemma, see, for example, \cite{HanLin}.
\begin{lemma}\label{simpleLemma}
Let $f(t)\geq 0$ be bounded in $[\tau_{0},\tau_{1}]$ with $\tau_{0}\geq 0$. Suppose for $\tau_{0}\leq t<s\leq\tau_{1}$, we have
$$f(t)\leq\theta\,f(s)+\frac{A}{(s-t)^{\alpha}}+B,$$
for some $\theta\in[0,1)$. Then for any $\tau_{0}\leq t<s\leq\tau_{1}$ there holds
$$f(t)\leq\,C(\alpha,\theta)\Big(\frac{A}{(s-t)^{\alpha}}+B\Big).$$
\end{lemma}

For Stokes system, the following Caccioppoli-type inequality is a starting point to build an adapted version of the iteration formula used in \cite{BLL,BLL2} to deal with Lam\'e system. For $|z'|\leq R/2$ and $s<R/2$, let
$$\Omega_{s}(z'):=\left\{(x',x_{2})\big| -\frac{\varepsilon}{2}+h_{2}(x')<x_{2}
<\frac{\varepsilon}{2}+h_{1}(x'),\,|x'-z'|<s \right\}.$$

\begin{lemma} (Caccioppoli-type Inequality) Let $({\bf w},q)$ be the solution to \eqref{w1}. For $0<t<s\leq R$, there holds
\begin{align}\label{iterating1}
\int_{\Omega_{t}(z')}|\nabla {\bf w}|^{2}\mathrm{d}x\leq\,&
\frac{C\delta^2(z')}{(s-t)^{2}}\int_{\Omega_{s}(z')}|\nabla {\bf w}|^2\mathrm{d}x+C\left((s-t)^{2}+\delta^{2}(z')\right)\int_{\Omega_{s}(z')}|{\bf f}|^{2}\mathrm{d}x.
\end{align}
\end{lemma}

\begin{proof}
Since ${\bf w}=0$ on $\partial D_1\cup\partial D_2$, employing Poincar\'e inequality, it is not difficult to deduce, see \cite[(3.36),(3.39)]{BLL2},
\begin{align}\label{estwDwl2}
\int_{\Omega_{s}(z')}|{\bf w}|^2\mathrm{d}x\leq
C\delta^2(z')\int_{\Omega_{s}(z')}|\nabla {\bf w}|^2\mathrm{d}x.
\end{align}
For $0<t<s\leq R$, let $\eta$ be a smooth function satisfying $\eta(x')=1$ if $|x'-z'|<t$, $\eta(x')=0$ if $|x'-z'|>s$, $0\leqslant\eta(x')\leqslant1$ if $t\leqslant|x'-z'|<s$,
 and
$|\eta'(x_1)\leq \frac{2}{s-t}$. Let $q_{s;z'}=\frac{1}{|\Omega_{s}(z')|}\int_{\Omega_{s}(z')}q\mathrm{d}x$. Multiplying the equation 
$$-\mu\Delta{\bf w}+\nabla(q-q_{s;z'})={\bf f}:=\mu\Delta{\bf v}-\nabla\bar{p},\quad\mbox{in}~\Omega_{s}(z'),$$
by $\eta^{2}{\bf w}$ and integrating by parts, we easily obtain
\begin{align}\label{energynarr}
\mu\int_{\Omega_{s}(z')}\eta^{2}|\nabla {\bf w}|^{2}\mathrm{d}x=&
-\mu\int_{\Omega_{s}(z')}({\bf w}\nabla {\bf w})\cdot\nabla\eta^{2}\mathrm{d}x
+\int_{\Omega_{s}(z')} \eta^{2}{\bf f}\cdot{\bf w}\mathrm{d}x\nonumber\\
&+\int_{\Omega_{s}(z')}\nabla\cdot(\eta^{2}{\bf w})\left(q-q_{s;z'}\right)\mathrm{d}x.
\end{align}
Next, we shall estimate the terms on the right-hand side of \eqref{energynarr} one by one. 

For the first term in \eqref{energynarr}, employing the Cauchy's inequality, 
\begin{align}\label{w1Dw1eta}
\int_{\Omega_{s}(z')}|{\bf w}\nabla {\bf w}||\nabla\eta^{2}|\mathrm{d}x
\leq \frac{1}{4}\int_{\Omega_{s}(z')}|\nabla {\bf w}|^{2}\mathrm{d}x
+\frac{C}{(s-t)^2}\int_{\Omega_{s}(z')}|{\bf w}|^{2}\mathrm{d}x,
\end{align}
and for the second term, 
\begin{align}\label{etaw1v1p1}
\frac{1}{\mu}\left|\int_{\Omega_{s}(z')} \eta^{2}{\bf f}\cdot{\bf w}\mathrm{d}x\right|&\leq
\frac{C}{(s-t)^{2}}\int_{\Omega_{s}(z')}|{\bf w}|^{2}\mathrm{d}x+C(s-t)^{2}\int_{\Omega_{s}(z')}|{\bf f}|^2\mathrm{d}x.
\end{align}
For the last term, in view of \eqref{changeofvariant}, and by means of  the rescaling argument as in Corollary \ref{rmkmean}, we have
\begin{align}\label{narrowq1}
\int_{\Omega_{s}(z')}\left|q-q_{s;z'}\right|^{2}\mathrm{d}x&\leq C_{1}
\int_{\Omega_{s}(z')}|\nabla {\bf w}|^2\mathrm{d}x
+C\delta^{2}(z')\int_{\Omega_{s}(z')}|{\bf f}|^2\mathrm{d}x,
\end{align}
where $C_{1}$ is fixed now. Making use of the second line of \eqref{w1}, $\nabla\cdot{\bf w}=\nabla\cdot{\bf v}=0$ in $\Omega_R$, and applying Young's inequality, we have
\begin{align}\label{etaw1q1}
&\frac{1}{\mu}\left|\int_{\Omega_{s}(z')}\left(q-q_{s;z'}\right)\nabla\cdot(\eta^{2}{\bf w})\mathrm{d}x \right|\nonumber\\
\leq&\, \frac{1}{4C_{1}}\int_{\Omega_{s}(z')}|q-q_{s;z'}|^2\eta^{2}\mathrm{d}x+C\int_{\Omega_{s}(z')}|{\bf w}|^{2}|\nabla\eta|^{2}\mathrm{d}x.
\end{align}
Combining \eqref{estwDwl2} and \eqref{w1Dw1eta}--\eqref{etaw1q1} with  \eqref{energynarr} yields
\begin{align*}
\int_{\Omega_{t}(z')}|\nabla {\bf w}|^{2}\mathrm{d}x\leq\,& \frac{1}{2}\int_{\Omega_{s}(z')}|\nabla {\bf w}|^{2}\mathrm{d}x
+\frac{C\delta^2(z')}{(s-t)^{2}}\int_{\Omega_{s}(z')}|\nabla {\bf w}|^2\mathrm{d}x\nonumber\\
&+C\left((s-t)^{2}+\delta^{2}(z')\right)\int_{\Omega_{s}(z')}|{\bf f}|^2\mathrm{d}x.
\end{align*}
By virtue of Lemma \ref{simpleLemma}, we deduce that the estimate \eqref{iterating1} holds.
\end{proof}

In order to apply Proposition \ref{lemWG2} to establish the $L^{\infty}$ estimates of $|\nabla{\bf w}_1^\alpha|$ in $\Omega_{2R}$, from the right hand side of \eqref{W2pstokes}, we will start from the Caccioppoli-type inequality \eqref{iterating1} and build an iteration formula, based on the method developed in \cite{LLBY,BLL} to obtain the local energy in the small region $\Omega_{\delta(z')}(z')$. This depends on more information of every ${\bf v}_1^\alpha$ and $p_1^\alpha$, constructed in Section \ref{auxiliary}. More calculations are given in the next Section.

\section{Proofs for the main Estimates}\label{sec_estimate2D}
This section is devoted to proving the main estimates, listed in Section \ref{auxiliary}, to prove our main result Theorem \ref{mainthm2D}. To express our idea clear, first we consider the case $\alpha=1$ in Subsection \ref{subsec1}, where the derivation is not too complex, compared to the rest cases. Since the treatment is subtle for the cases $\alpha=2$, Subsection \ref{subsec2} is technical, especially when we deal with the boundedness of the corresponding global energy $E_i^{\alpha}$ defined in \eqref{wialpha} below.

\subsection{Estimates of $|\nabla{\bf u}_{i}^{1}|$ and $|p_{i}^{1}|$}\label{subsec1}

Recalling \eqref{v11}, by direct calculations,
\begin{equation}\label{estv112}
\partial_{x_1}({\bf v}^{1}_{1})^{(1)}=\partial_{x_1}k(x)=-\frac{2x_{1}}{\delta(x_{1})}k(x),\quad \partial_{x_2}({\bf v}^{1}_{1})^{(1)}=\partial_{x_2}k(x)=\frac{1}{\delta(x_{1})};
\end{equation}
and
\begin{align}\label{estv1122}
|\partial_{x_1}({\bf v}^{1}_{1})^{(2)}|\leq C,\quad
\partial_{x_2}({\bf v}^{1}_{1})^{(2)}= \frac{2x_{1}}{\delta(x_{1})}k(x).
\end{align}
Clearly, using \eqref{defp11}, we have 
\begin{equation}\label{v11p11}
\mu\partial_{x_2}({\bf v}^{1}_{1})^{(2)}-\bar{p}_1^1=0, \quad\mathrm{in}~\Omega_{2R}.
\end{equation}

By further calculations,
\begin{align}
\partial_{x_1x_1}({\bf v}^{1}_{1})^{(1)}&=-\frac{2}{\delta(x_{1})}k(x)+\frac{8x_{1}^{2}}{\delta^{2}(x_{1})}k(x),\quad\partial_{x_2x_2}({\bf v}^{1}_{1})^{(1)}=0;\label{v11-22-2D}\\
|\partial_{x_1x_1}({\bf v}^{1}_{1})^{(2)}|&\leq\frac{C|x_{1}|}{\delta(x_{1})},\quad\mbox{and}~
|\partial_{x_1}\bar{p}_1^1|\leq\frac{C}{\delta(x_1)}\nonumber.
\end{align}
Therefore,
\begin{align}\label{estdivv11p1}
|{\bf f}_{1}^{1}|=\left|\mu\Delta{\bf v}_{1}^1-\nabla\bar{p}_1^1\right|\leq \frac{C}{\delta(x_{1})}.
\end{align}

By virtue of divergence free condition, \eqref{freev112D}, using estimates  \eqref{v11-22-2D}, \eqref{estv112}, and thanks to Lemma \ref{lemmaenergy}, we can prove
\begin{align*}
\int_{\Omega}|\nabla {\bf w}_1^1|^{2}\mathrm{d}x\leq C.
\end{align*}
Then, combining with \eqref{estsigvp} and applying the iteration process in Lemma \ref{lem3.1} allows us to show that
\begin{align}\label{estw11narrow}
\int_{\Omega_{\delta}(z_{1})}|\nabla {\bf w}_1^1|^{2}\mathrm{d}x\leq C\delta^2(z_1).
\end{align}
Substituting \eqref{estw11narrow} and \eqref{estdivv11p1} into \eqref{W2pstokes},
we conclude that the above $\nabla {\bf v}_{1}^1$ and $\bar{p}_1^1$ are the main singular terms of $\nabla {\bf u}_{1}^1$ and $p_1^1$, respectively.   

Let us denote the total energy in $\Omega$ by
\begin{align}\label{wialpha}
E_{i}^{\alpha}:=\int_{\Omega}|\nabla{\bf w}_{i}^{\alpha}|^{2}\mathrm{d}x,\quad\alpha=1,2,3.
\end{align}
The above properties of ${\bf v}_{1}^{1}$ and $\bar{p}_1^{1}$ enables us to obtain the boundedness of $E_{i}^{1}$,  which is a very important step to employ our iteration approach.
\begin{lemma}\label{lem3.0}
Let $({\bf w}_i^1,q_i^1)$ be the solution to \eqref{w1}. Then
\begin{align}\label{w1alpha}
E_{i}^{1}=\int_{\Omega}|\nabla{\bf w}_{i}^{1}|^{2}\mathrm{d}x\leq\,C,\quad i=1,2.
\end{align}
\end{lemma}

\begin{proof}
Here and throughout this paper, we only prove the case of $i=1$  for instance, since the case $i=2$ is the same. From Lemma  \ref{lemmaenergy},  the proof of inequality \eqref{w1alpha} is reduced to proving condition \eqref{int-fw}, i.e.
\begin{align*}
\Big| \int_{\Omega_{R}}\sum_{j=1}^{2}({\bf f}_1^1)^{(j)}({\bf w}_1^1)^{(j)}\mathrm{d}x\Big|\leq\,C\left(\int_{\Omega}|\nabla {\bf w}_1^1|^2\mathrm{d}x\right)^{1/2}.
\end{align*}
For this, first as in \cite{BLL}, by mean value theorem and Poincar\'e inequality, there exists $r_{0}\in(R,\frac{3}{2}R)$ such that
\begin{align}\label{w1Dw1}
\int_{\substack{|x_1|=r_{0},\\-\epsilon/2-h_{2}(x_1)<x_{2}<\epsilon/2+h_{1}(x_1)}}|{\bf w}_1^1|\mathrm{d}x_{2}
&\leq C\left(\int_{\Omega}|\nabla {\bf w}_1^1|^2\mathrm{d}x\right)^{1/2}.
\end{align}

Applying the integration by parts with respect to $x_1$, for $j=1$, using $\partial_{x_2x_2}({\bf v}^{1}_{1})^{(1)}=0$ in \eqref{v11-22-2D}, we have 
\begin{align*}
\left|\int_{\Omega_{r_{0}}} ({\bf f}_1^1)^{(1)}({\bf w}_1^1)^{(1)}\mathrm{d}x\right|
&=\left|\int_{\Omega_{r_{0}}} ({\bf w}_1^1)^{(1)}(\mu\partial_{x_1x_1}({\bf v}_1^1)^{(1)}-\partial_{x_1}\bar{p}_1^1)\mathrm{d}x\right|\nonumber\\
&\leq\,C\int_{\Omega_{r_{0}}}\big(|\partial_{x_1}({\bf w}_1^1)^{(1)}||\partial_{x_1}({\bf v}_1^1)^{(1)}|+|\partial_{x_1}({\bf w}_1^1)^{(1)}||\bar{p}_1^1|\big)\mathrm{d}x\nonumber\\
&\quad+C\,\int_{\substack{|x_1|=r_{0},\\-\epsilon/2-h_{2}(x_1)<x_{2}<\epsilon/2+h_{1}(x_1)}}|({\bf w}_1^1)^{(1)}|\mathrm{d}x_{2}=:C\mathrm{I}_1+C\,\mathrm{I}_2^{(1)};
\end{align*}
and for $j=2$, using \eqref{v11p11}, we obtain
\begin{align*}
&\left|\int_{\Omega_{r_{0}}} ({\bf f}_1^1)^{(2)}({\bf w}_1^1)^{(2)}\mathrm{d}x\right|=\left|\mu\int_{\Omega_{r_{0}}} ({\bf w}_1^1)^{(2)}\partial_{x_1x_1}({\bf v}_1^1)^{(2)}\mathrm{d}x\right|\nonumber\\
\leq&\,C\int_{\Omega_{r_{0}}}|\partial_{x_1}({\bf w}_1^1)^{(2)}||\partial_{x_1}({\bf v}_1^1)^{(2)}|\mathrm{d}x+C\int_{\substack{|x_1|=r_{0},\\-\epsilon/2-h_{2}(x_1)<x_{2}<\epsilon/2+h_{1}(x_1)}}|({\bf w}_1^1)^{(2)}|\mathrm{d}x_{2}\nonumber\\
=&\,:C\mathrm{I}_2+C\,\mathrm{I}_2^{(2)}.
\end{align*}
Applying H\"{o}lder's inequality and using \eqref{estv112}, \eqref{estv1122}, and \eqref{defp11},
\begin{align*}
|\mathrm{I}_1|&\leq C\left(\int_{\Omega_{r_{0}}}|\partial_{x_1}({\bf v}_1^1)^{(1)}|^{2}+|\bar{p}_1^1|^2\mathrm{d}x\right)^{1/2}
\left(\int_{\Omega} |\partial_{x_1}({\bf w}_1^1)^{(2)}|^2\mathrm{d}x\right)^{1/2}\nonumber\\
&\leq C \left(\int_{\Omega} |\nabla {\bf w}_1^1|^2\mathrm{d}x\right)^{1/2},
\end{align*}
and
\begin{align*}
|\mathrm{I}_2|&\leq C\left(\int_{\Omega_{r_{0}}}|\partial_{x_1}({\bf v}_1^1)^{(2)}|^{2}\mathrm{d}x\right)^{1/2}
\left(\int_{\Omega} |\partial_{x_1}({\bf w}_1^1)^{(2)}|^2\mathrm{d}x\right)^{1/2}\leq C \left(\int_{\Omega} |\nabla {\bf w}_1^1|^2\mathrm{d}x\right)^{1/2}.
\end{align*}
Together with \eqref{w1Dw1}, we derive \eqref{int-fw}. Then, thanks to Lemma  \ref{lemmaenergy}, the proof of Lemma \ref{lem3.0} is finished.
\end{proof}

\begin{remark}
With \eqref{w1alpha}, as mentioned before, we can directly apply classical elliptic estimates, see \cite{ADN1959,ADN1964,Solonni1966}, to obtain $\|\nabla{\bf w}_{i}^{1}\|_{L^{\infty}(\Omega\setminus\Omega_{R})}\leq\,C$. So we only focus on the estimates in $\Omega_{R}$.
\end{remark}

Integrating \eqref{estdivv11p1}, we have
\begin{align}\label{estsigvp}
\int_{\Omega_{s}(z_1)}|{\bf f}_{1}^{1}|^{2}\mathrm{d}x\leq
\frac{Cs}{\delta(z_1)}.
\end{align}
This and \eqref{w1alpha} are good enough to employ the adapted iteration technique to obtain the following local energy estimates.

\begin{lemma}\label{lem3.1}
Let $({\bf w}_i^1,q_i^1)$ be the solution to \eqref{w1}. Then
\begin{align}\label{estw11narrow3D}
\int_{\Omega_{\delta}(z_1)}|\nabla {\bf w}_i^1|^{2}\mathrm{d}x\leq C\delta^2(z_1),\quad i=1,2.
\end{align}
\end{lemma}

\begin{proof} 
The iteration scheme we use in the proof is in
spirit similar  to that used in \cite{LLBY,BLL2}. Let us set
$$E(t):=\int_{\Omega_{t}(z_1)}|\nabla {\bf w}_1^1|^{2}\mathrm{d}x.$$
Combining \eqref{estsigvp} with \eqref{iterating1} yields
\begin{align}\label{iteration3D}
E(t)\leq \left(\frac{c_{0}\delta(z_1)}{s-t}\right)^2 E(s)+C\left((s-t)^{2}+\delta^2(z_1)\right)\frac{s}{\delta(z_1)},
\end{align}
where $c_{0}$ is a constant and we fix it now. Let $k_{0}=\left[\frac{1}{4c_{0}\sqrt{\delta(z_1)}}\right]$ and $t_{i}=\delta(z_1)+2c_{0}i\delta(z_1), i=0,1,2,\dots,k_{0}$. 
So, applying \eqref{iteration3D} with $s=t_{i+1}$ and $t=t_{i}$, we have the following iteration formula:
\begin{align*}
E(t_{i})\leq \frac{1}{4}E(t_{i+1})+C(i+1)\delta^2(z_1).
\end{align*}
After $k_{0}$ iterations, and by virtue of \eqref{w1alpha}, we obtain
\begin{align*}
E(t_0)&\leq \left(\frac{1}{4}\right)^{k}E(t_{k})
+C\delta^2(z_1)\sum\limits_{l=0}^{k-1}\left(\frac{1}{4}\right)^{l}(l+1)\nonumber\\
&\leq \left(\frac{1}{4}\right)^{k}E_{1}^{\alpha}+C\delta^2(z_1)\sum\limits_{l=0}^{k-1}\left(\frac{1}{4}\right)^{l}(l+1)\leq C\delta^2(z_1),
\end{align*}
for sufficiently small $\varepsilon$ and $|z_1|$.  As a consequence,  \eqref{estw11narrow3D} is proved.
\end{proof}

We are now in a position to prove Proposition \ref{propu11}.
\begin{proof}[Proof of Proposition \ref{propu11}.]
By using \eqref{estw11narrow3D} and \eqref{estdivv11p1}, and applying \eqref{W2pstokes}, we have 
\begin{align*}
\|\nabla {\bf w}_1^1\|_{L^{\infty}(\Omega_{\delta/2}(z_1))}
\leq
\,C\left(\delta^{-1}\|\nabla {\bf w}_1^1\|_{L^{2}(\Omega_{\delta}(z_1))}+\delta\|{\bf f}_1^1\|_{L^{\infty}(\Omega_{\delta}(z_1))} \right)\leq\,C.
\end{align*} 
Since ${\bf w}_1^1={\bf u}_1^1-{\bf v}_1^1$, it follows from \eqref{v11-1d} that
\begin{align*}
|\nabla {\bf u}_1^1(z)|&\leq|\nabla {\bf v}_1^1(z)|+|\nabla {\bf w}_1^1(z)|\leq\frac{C}{\delta(z_1)},\quad \quad~z\in\Omega_{R},\\
|\nabla {\bf u}_1^1(0',z_{2})|&\geq|\nabla {\bf v}_1^1(0',z_{2})|-C\geq\frac{1}{C\delta(z_1)},\quad |z_{2}|\leq\frac{\varepsilon}{2}.
\end{align*}
Similarly, we obtain from \eqref{Wmpstokes} that  
\begin{align*}
&\|\nabla^2 {\bf w}_1^1\|_{L^{\infty}(\Omega_{\delta/2}(z_1))}+\|\nabla q_1^1\|_{L^{\infty}(\Omega_{\delta/2}(z_1))}\\
&\leq
C\left(\delta^{-2}\|\nabla {\bf w}_1^1\|_{L^{2}(\Omega_{\delta}(z_1))}+\|{\bf f}_1^1\|_{L^{\infty}(\Omega_{\delta}(z_1))}+\delta\|\nabla{\bf f}_1^1\|_{L^{\infty}(\Omega_{\delta}(z_1))} \right).
\end{align*}
By \eqref{v11}, \eqref{defp11}, \eqref{v11-22-2D}, and using a direct computation, we derive 
\begin{align*}
|{\bf f}_1^1|+\delta|\nabla{\bf f}_1^1|\leq \frac{C}{\delta(z_1)},\quad z\in\Omega_R.
\end{align*}
Then together with \eqref{estw11narrow3D}, we obtain
\begin{align*}
\|\nabla^2 {\bf w}_1^1\|_{L^{\infty}(\Omega_{\delta/2}(z_1))}+\|\nabla q_1^1\|_{L^{\infty}(\Omega_{\delta/2}(z_1))}\leq\frac{C}{\delta(z_1)}.
\end{align*}
This in combination with ${\bf w}_1^1={\bf u}_1^1-{\bf v}_1^1$ and \eqref{v11} yields
$$|\nabla^2 {\bf u}_1^1|\leq C\left(\frac{1}{\delta(z_{1})}+\frac{|z_1|}{\delta^2(z_{1})}\right),\quad z\in\Omega_R.$$
Combining $q_1^1=p_1^1-\bar{p}_1^1$,  the mean value theorem, and \eqref{defp11}, we have 
\begin{align*}
|p_{1}^{1}(z)-(q_1^1)_{R}|\leq |q_1^{1}(z)-(q_1^1)_{R}|+|\bar{p}_1^{1}(z)|\leq\frac{C}{\varepsilon},\quad z\in\Omega_R,
\end{align*}
where $(q_1^1)_{R}$ is a constant independent of $\varepsilon$ defined in \eqref{defqialpha}. Similarly, 
$$|\nabla p_{1}^{1}(z)|\leq |\nabla q_1^{1}(z)|+|\nabla\bar{p}_1^{1}(z)|\leq C\left(\frac{1}{\delta(z_1)}+\frac{|z_1|}{\delta^2(z_1)}\right),\quad z\in\Omega_R.$$
Therefore, Proposition \ref{propu11} is proved. 
\end{proof}

\begin{remark}
From the above calculations, our method works well for general $h_{1}(x_1)$ and $h_{2}(x_1)$. For example, if $h_{1}=h_{2}=h(x_1)$ satisfies \eqref{h1-h2}--\eqref{h1h14}, then we take
$${\bf v}_{1}^1=\boldsymbol\psi_{1}\Big(k(x)+\frac{1}{2}\Big)+\frac{\partial_{x_{1}}(h_{1}+h_{2})(x_1)}{2}\boldsymbol\psi_{2}\Big(k^2(x)-\frac{1}{4}\Big),\quad\hbox{in}\ \Omega_{2R}.
$$
It is easy to check that $\nabla\cdot{\bf v}_{1}^1=0$ in $\Omega_{2R}$. As far as the estimates of ${\bf v}_{1}^1$ and ${\bf f}_{1}^1$ are concerned, only slight modification  is needed. 
\end{remark}

\begin{remark}
For $m\geq2$, by \eqref{Wmpstokes} and replicating the same argument as above, we obtain
\begin{align}\label{generalm}
\|\nabla^{m+1}{\bf w}_1^1\|_{L^{\infty}(\Omega_{\delta/2}(z_1))}+\|\nabla^{m} q_1^1\|_{L^{\infty}(\Omega_{\delta/2}(z_1))}\leq\frac{C}{\delta^{m}(z_1)}.
\end{align}
However, recalling the definition of ${\bf v}_1^1$ in \eqref{v11} and using direct calculations, we have for any $z\in\Omega_R$,
$$|\nabla^{m+1}{\bf v}_1^1|\leq
\begin{cases}
C\left(\frac{1}{\delta^{2}(z_1)}+\frac{|z_1|}{\delta^{3}(z_1)}\right),&\quad m=2,\\
\frac{C}{\delta^{3}(z_1)},&\quad m=3,\\
C\left(\frac{1}{\delta^{m-1}(z_1)}+\frac{|z_1|}{\delta^{m}(z_1)}\right),&\quad m\geq4.
\end{cases}$$
Compared to \eqref{generalm}, one can see that $\nabla^{m+1}{\bf v}_1^1$ cannot capture the main singular term of $\nabla^{m+1}{\bf u}_1^1$ any more when $m\geq3$. Therefore, to prove higher-order partial derivatives,  new auxiliary functions are needed.
\end{remark}

\subsection{Estimates of $|\nabla{\bf u}_{i}^{2}|$ and $|p_{i}^{2}|$}\label{subsec2}
Using \eqref{v12} and \eqref{p12}, a direct calculation yields
\begin{align*}
\mu\partial_{x_2} ({\bf v}_{1}^2)^{(2)}-\bar{p}_1^2=\frac{3\mu}{2\delta(x_1)}+\frac{3\mu}{\delta^{2}(x_{1})}(1-x_{1}^{2}),
\end{align*}
which directly implies that
\begin{align}\label{estv122}
\mu\partial_{x_2x_2} ({\bf v}_{1}^2)^{(2)}-\partial_{x_2}\bar{p}_1^2=0.
\end{align}
Moreover,
\begin{equation*}
\partial_{x_1}\bar{p}_1^2=\frac{12\mu x_1}{\delta^3(x_1)}-\frac{36\mu x_1}{\delta^2(x_1)}\left(\frac{8x_{1}^{2}}{\delta(x_{1})}-3\right)k^{2}(x).
\end{equation*}
It also makes 
\begin{align}\label{estp121}
\left|\mu\partial_{x_2x_2} ({\bf v}_{1}^2)^{(1)}-\partial_{x_1}\bar{p}_1^2\right|=\left|\frac{36\mu x_1}{\delta^2(x_1)}\left(\frac{8x_{1}^{2}}{\delta(x_{1})}-3\right)k^{2}(x)\right|\leq \frac{C|x_1|}{\delta^2(x_1)},
\end{align}
which is smaller than $|\partial_{x_2x_2} ({\bf v}_{1}^2)^{(1)}|\sim\frac{1}{\delta^{2}(x_{1})}$.
The following estimates are also needed:
\begin{equation}\label{estv1211}
\partial_{x_1}({\bf v}_{1}^{2})^{(1)}=\frac{6}{\delta(x_1)}\Big(1-\frac{2x_1^2}{\delta(x_1)}\Big)\big(k^2(x)-\frac{1}{4}\big)-\frac{24x_1^2x_2}{\delta^3(x_1)}k(x),
\end{equation}
\begin{align}
\frac{1}{C\delta(x_1)}&\leq|\partial_{x_2}({\bf v}_{1}^{2})^{(2)}|\leq\frac{C}{\delta(x_1)},\label{estv1210}\\
\partial_{x_2}({\bf v}_{1}^{2})^{(1)}&=\frac{12x_{1}}{\delta^{2}(x_{1})}k(x),\quad\quad
\left|\partial_{x_1}({\bf v}_{1}^{2})^{(2)}\right|\leq\frac{C|x_{1}|}{\delta(x_{1})};\label{estv121}
\end{align}
and
\begin{align}\label{estv121p122}
\left|\partial_{x_1x_1}({\bf v}_{1}^{2})^{(1)}\right|\leq\frac{C|x_1|}{\delta^2(x_1)},\quad
\left|\partial_{x_1x_1}({\bf v}_{1}^{2})^{(2)}\right|\leq\frac{C}{\delta(x_{1})}.
\end{align}
Therefore, we have
\begin{align}\label{estDivv12}
|{\bf f}_1^2|=\left|\mu\Delta{\bf v}_{1}^2-\nabla\bar{p}_1^2\right|\leq\frac{C|x_1|}{\delta^{2}(x_{1})}.
\end{align}

Using \eqref{estv1210} and \eqref{estv121}, we have 
\begin{equation}\label{v13-0}
\int_{\Omega}|\nabla {\bf v}_{1}^{2}|^{2}\mathrm{d}x\leq C\int_{\Omega}\left(\frac{1}{\delta(x_1)}+\frac{|x_1|}{\delta^{2}(x_1)}\right)^{2}\mathrm{d}x\leq\frac{C}{\varepsilon^{3/2}},
\end{equation}
it is impossible prove the boundedness of the global energy of ${\bf w}_{1}^{2}$ by directly applying Lemma \ref{lem3.0} (see the proof of Lemma \ref{lem3.0}). So we must improve our technique in the proof of Lemma \ref{lem_energyw12} below. 

\begin{lemma}\label{lem_energyw12}
Let $({\bf w}_i^2,q_i^2)$ be the solution to \eqref{w1}. Then
\begin{align*}
\int_{\Omega}|\nabla {\bf w}_{i}^{2}|^{2}\mathrm{d}x\leq C,\quad\,i=1,2.
\end{align*}
\end{lemma}

\begin{proof}
Similarly as in Lemma \ref{lem3.0}, by virtue of Lemma \ref{lemmaenergy}, it suffices to prove
\begin{align}\label{energyDw13}
\Big| \int_{\Omega_{r_{0}}}\sum_{j=1}^{2}({\bf f}_{1}^{2})^{(j)}({\bf w}_{1}^{2})^{(j)}\mathrm{d}x\Big|\leq\,C\left(\int_{\Omega}|\nabla {\bf w}_{1}^{2}|^2\mathrm{d}x\right)^{1/2},
\end{align}
where $r_{0}\in(R,\frac{3}{2}R)$ is fixed by \eqref{w1Dw1}.

First, for $j=1$, 
\begin{align}\label{fw12}
\int_{\Omega_{r_{0}}} ({\bf f}_{1}^{2})^{(1)}({\bf w}_{1}^{2})^{(1)}\mathrm{d}x
=\int_{\Omega_{r_{0}}} ({\bf w}_1^2)^{(1)}(\mu\Delta({\bf v}^{2}_{1})^{(1)}-\partial_{x_{1}}\bar{p}_1^2)\mathrm{d}x.
\end{align}
Note that here we can not use the integration by parts with the respect to $x_1$ any more as in Lemma \ref{lem3.0}, because even if we did, after that the terms $|\partial_{x_1}({\bf v}^{2}_{1})^{(1)}|\leq\frac{C}{\delta(x_1)}$ are still too large to prove \eqref{energyDw13}, see for example \eqref{v13-0}. Observing from \eqref{estp121}, \eqref{estv1211}, and \eqref{estv121}, we can write $({\bf f}_{1}^{2})^{(1)}$ in the polynomial form
$$\mu\Delta({\bf v}^{2}_{1})^{(1)}-\partial_{x_{1}}\bar{p}_1^1:=A^{21}_{0}(x_1)+A^{21}_{2}(x_1)x_{2}^{2},$$
where $A^{21}_{l}(x_1)$ are rational functions of $x_{1}$. In view of \eqref{estp121} and \eqref{estv121p122},  
$$\Big|A^{21}_{0}(x_1)+A^{21}_{2}(x_1)x_{2}^{2}\Big|\leq\frac{C|x_1|}{\delta^{2}(x_1)}.$$
However, by a simple differentiation, we know
$$\partial_{x_{2}}\Big(A^{21}_{0}(x_1)x_{2}+\frac{1}{3}A^{21}_{2}(x_1)x_{2}^{3}\Big)=A^{21}_{0}(x_1)+A^{21}_{2}(x_1)x_{2}^{2},$$
and thus,  
$$\int_{\Omega_{r_{0}}}\Big|A^{21}_{0}(x_1)x_{2}+\frac{1}{3}A^{21}_{2}(x_1)x_{2}^{3}\Big|^{2}\leq\,C\int_{\Omega_{r_{0}}}\frac{|x_1|^{2}}{\delta^{2}(x_1)}\leq\,C.$$
So we rewrite \eqref{fw12} as follows:
\begin{align*}
\int_{\Omega_{r_{0}}} ({\bf f}_{1}^{2})^{(1)}({\bf w}_{1}^{2})^{(1)}\mathrm{d}x
=\int_{\Omega_{r_{0}}} ({\bf w}_{1}^{2})^{(1)}\partial_{x_{2}}\Big(A^{21}_{0}(x_1)x_{2}+\frac{1}{3}A^{21}_{2}(x_1)x_{2}^{3}\Big)\mathrm{d}x.
\end{align*}
Since ${\bf w}_1^2=0$ on $\partial D_{1}\cup \partial D_{2}$, we apply the integration by parts with respect to $x_{2}$, instead of $x_1$ used in Lemma \ref{lem3.0}. By H\"older's inequality and \eqref{w1Dw1}, we deduce 
\begin{align*}
\left|\int_{\Omega_{r_{0}}} ({\bf f}_{1}^{2})^{(1)}({\bf w}_{1}^{2})^{(1)}\mathrm{d}x\right|&\leq\int_{\Omega_{r_{0}}}|\partial_{x_{2}}({\bf w}_{1}^{2})^{(1)}|\Big|A^{21}_{0}(x_1)x_{2}+\frac{1}{3}A^{21}_{2}(x_1)x_{2}^{3}\Big|\mathrm{d}x\nonumber\\
&\quad+C\left(\int_{\Omega}|\nabla {\bf w}_{1}^{2}|^2\mathrm{d}x\right)^{1/2}\leq\,C\left(\int_{\Omega}|\nabla {\bf w}_{1}^{2}|^2\mathrm{d}x\right)^{1/2}.
\end{align*}
Hence,
\begin{equation}\label{estf13w13}
\left|\int_{\Omega_{r_{0}}} ({\bf f}_{1}^{2})^{(1)}({\bf w}_{1}^{2})^{(1)}\mathrm{d}x\right|\leq C\left(\int_{\Omega}|\nabla {\bf w}_{1}^{2}|^2\mathrm{d}x\right)^{1/2}.
\end{equation}

For $j=2$, by means of \eqref{estv122}, we have 
$$({\bf f}_{1}^{2})^{(2)}=\mu\partial_{x_1x_1}({\bf v}^{2}_{1})^{(2)}.$$
In view of \eqref{estv121}, we obtain
$$\int_{\Omega_{r_{0}}}|\partial_{x_1}({\bf v}^{2}_{1})^{(2)}|^{2}\mathrm{d}x\leq\,C\int_{\Omega_{r_{0}}}\frac{|x_1|^{2}}{\delta^{2}(x_1)}\leq\,C,$$
so we can still use the integration by parts with respect to $x_1$, similarly as in Lemma \ref{lemmaenergy}, and  use \eqref{w1Dw1} again to obtain 
\begin{align}\label{estf1333}
\left|\int_{\Omega_{r_{0}}} ({\bf f}_{1}^{2})^{(2)}({\bf w}_{1}^{2})^{(2)}\mathrm{d}x\right|
&=\mu\left|\int_{\Omega_{r_{0}}} ({\bf w}_{1}^{2})^{(2)}\partial_{x_1x_1}({\bf v}^{2}_{1})^{(2)}\mathrm{d}x\right|\nonumber\\
&\leq\mu\int_{\Omega_{r_{0}}}|\partial_{x_1}({\bf w}_{1}^{2})^{(2)}||\partial_{x_1}({\bf v}^{2}_{1})^{(2)}|\mathrm{d}x\nonumber\\
&\quad+\int_{\substack{|x_1|=r_{0},\\-\epsilon/2-h_{2}(x_1)<x_{2}<\epsilon/2+h_{1}(x_1)}}|({\bf w}_{1}^{2})^{(2)}|\mathrm{d}s\nonumber\\
&\leq C\left(\int_{\Omega}|\nabla {\bf w}_{1}^{2}|^2\mathrm{d}x\right)^{1/2}.
\end{align}
Taking into account \eqref{estf13w13} and \eqref{estf1333}, we obtain \eqref{energyDw13}. Thanks to Lemma \ref{lemmaenergy}, the proof of Lemma \ref{lem_energyw12} is finished. 
\end{proof}

Using \eqref{estDivv12}, we obtain
\begin{align}\label{L2_f12}
\int_{\Omega_{s}(z_1)}|{\bf f}_{1}^{2}|^{2}\mathrm{d}x\leq\frac{Cs}{\delta^{3}(z_1)}(s^2+|z_1|^2).
\end{align}
With \eqref{L2_f12} and Lemma \ref{lem_energyw12} at hand, we may argue as before to complete the proof of Proposition \ref{propu12}.

\begin{proof}[Proof of Proposition \ref{propu12}.]
Substituting \eqref{L2_f12} into  \eqref{iterating1}, instead of \eqref{iteration3D}, 
\begin{align*}
\int_{\Omega_{t}(z_1)}|\nabla {\bf w}_{1}^{2}|^{2}\mathrm{d}x&\leq\,
\frac{C\delta^2(z_1)}{(s-t)^{2}}\int_{\Omega_{s}(z_1)}|\nabla {\bf w}_{1}^{2}|^2\mathrm{d}x\\
&\quad+C\left((s-t)^{2}+\delta^{2}(z_1)\right)\frac{s}{\delta^3(z_1)}(s^2+|z_1|^2).
\end{align*}
By the same iteration process used in Lemma \ref{lem3.1}, we deduce 
\begin{align*}
\int_{\Omega_{\delta}(z_1)}|\nabla {\bf w}_1^2|^{2}\mathrm{d}x\leq C\delta^2(z_1).
\end{align*}
Substituting this, together with \eqref{estDivv12}, into \eqref{W2pstokes} yields
\begin{align*}
\|\nabla {\bf w}_1^{2}\|_{L^{\infty}(\Omega_{\delta/2}(z_1))}
\leq\,C\left(1+\frac{|z_1|}{\delta(z_1)}\right)
\leq\,\frac{C}{\sqrt{\delta(z_1)}}.
\end{align*} 
Using \eqref{Wmpstokes} and 
\begin{equation*}
|{\bf f}_1^2|+\delta|\nabla{\bf f}_1^2|\leq\frac{C}{\delta(z_1)}+\frac{C|z_1|}{\delta^{2}(z_1)},
\end{equation*}
we have 
\begin{align*}
\|\nabla^2 {\bf w}_1^{2}\|_{L^{\infty}(\Omega_{\delta/2}(z_1))}+\|\nabla q_1^2\|_{L^{\infty}(\Omega_{\delta/2}(z_1))}\leq\frac{C}{\delta(z_1)}+\frac{C|z_1|}{\delta^{2}(z_1)}.
\end{align*}
Recalling  ${\bf w}_1^{2}={\bf u}_1^{2}-{\bf v}_1^{2}$ and $q_1^2=p_1^2-\bar p_1^2$, and using \eqref{v12upper1}--\eqref{p12}, Proposition \ref{propu12} is proved. 
\end{proof}

\subsection{Estimate of $|\nabla{\bf u}_{i}^{3}|$ and $|p_{i}^{3}|$}\label{subsec3}

Recalling \eqref{v13}, by careful calculations,
\begin{align}\label{v1311}
\partial_{x_1}({\bf v}_{1}^{3})^{(1)}&=-\frac{9x_1}{2}k^{2}(x)+12\left(\frac{-x_1}{\delta(x_1)}+\frac{2x_1^{3}}{\delta^{2}(x_1)}\right)k^{2}(x)+30x_{1}k^{4}(x)\nonumber\\
&\quad+2\left(\frac{x_1}{\delta(x_1)}-\frac{x_1^{3}}{\delta^{2}(x_1)}\right),
\end{align}
\begin{equation}\label{v1321}
\partial_{x_2}({\bf v}_{1}^{3})^{(1)}=\frac{9}{2}k(x)+2\left(\frac{1}{\delta(x_1)}-\frac{4x_1^2}{\delta^2(x_1)}\right)k(x)-20k^{3}(x)+\frac{1}{2},
\end{equation}
\begin{align}\label{v1312}
\partial_{x_1}({\bf v}_{1}^{3})^{(2)}=&-2k(x)-\frac{1}{2}+2(\frac{5x_1^2}{\delta(x_1)}-\frac{4x_1^4}{\delta^2(x_1)})k(x)+\frac{3x_2}{2}k^{2}(x)\nonumber\\
&+(4-6x_1^2-\frac{52x_1^2}{\delta(x_1)}+\frac{64x_1^4}{\delta^2(x_1)})k^{3}(x)+48x_1^2k^{5}(x),
\end{align}
and
\begin{equation}\label{v1322}
\partial_{x_2}({\bf v}_{1}^{3})^{(2)}=-\frac{2x_1}{\delta(x_1)}+\frac{9x_1}{2}k^{2}(x)+\frac{2x_1^3}{\delta^2(x_1)}+12\left(\frac{x_1}{\delta(x_1)}-\frac{2x_1^3}{\delta^2(x_1)}\right)k^{2}(x)-30x_1k^{4}(x).
\end{equation}
Thus,
\begin{equation*}
\nabla\cdot{\bf v}_{1}^{3}=\partial_{x_1}({\bf v}_{1}^{3})^{(1)}+\partial_{x_2}({\bf v}_{1}^{3})^{(2)}=0,\quad \mbox{in}~\Omega_{2R}.
\end{equation*}
But, at this moment, we not only have no $\mu\partial_{x_2}({\bf v}_{1}^{3})^{(2)}-\bar{p}_1^3=0$, even $\mu\partial_{x_2x_2}({\bf v}_{1}^{3})^{(2)}-\partial_{x_2}\bar{p}_1^3=0$. Indeed, we  only have the following upper bounds:
\begin{equation}\label{estv131-p131}
\Big|\mu\partial_{x_2x_2}({\bf v}_{1}^{3})^{(1)}-\partial_{x_1}\bar{p}_1^3\Big|=\mu\Big|\frac{9}{2\delta(x_1)}-\frac{72x_2^2}{\delta^3(x_1)}+\frac{144x_1^2 x_2^2}{\delta^4(x_1)}-\frac{192 x_1^4x_2^2}{\delta^5(x_1)}\Big|\leq\frac{C}{\delta(x_{1})},
\end{equation} 
and 
\begin{equation}\label{estv132-p132}
\Big|\mu\partial_{x_2x_2}({\bf v}_{1}^{3})^{(2)}-\partial_{x_2}\bar{p}_1^3\Big|=\mu\Big|\frac{9x_1x_2}{\delta^2(x_1)}-\frac{120x_1x_2^3}{\delta^4(x_1)}\Big|\leq\frac{C|x_{1}|}{\delta(x_{1})}.
\end{equation} 
Furthermore,
\begin{equation*}
\left|\partial_{x_1x_1}({\bf v}_{1}^{3})^{(1)}\right|\leq\frac{C}{\delta(x_1)},\quad \left|\partial_{x_1x_1}({\bf v}_{1}^{3})^{(2)}\right|\leq\frac{C|x_1|}{\delta(x_1)}.
\end{equation*} 
Therefore,
\begin{align*}
|{\bf f}_1^3|=\left|\mu\Delta {\bf v}_{1}^3-\nabla\bar{p}_1^3\right|\leq \frac{C}{\delta(x_1)},\quad\mbox{in}~\Omega_{2R}.
\end{align*}
This is the same as in \eqref{estdivv11p1} for ${\bf f}_{1}^{1}$. In oder to apply \eqref{iteration3D} and the iteration process as in Lemma \ref{lem3.1}, it remains to prove the boundedness of $E_{i}^{3}$.

\begin{lemma}\label{lem3.4}
Let $({\bf w}_i^3,q_i^3)$ be the solution to \eqref{w1}. Then
\begin{align*}
E_{i}^{3}=\int_{\Omega}|\nabla{\bf w}_{i}^{3}|^{2}\mathrm{d}x\leq\,C.
\end{align*}
\end{lemma}

\begin{proof}
Similarly as in the proof of Lemma \ref{lem_energyw12}, we rewrite \eqref{estv131-p131} and \eqref{estv132-p132} into the form of polynomials as follows:
\begin{align*}
\mu\partial_{x_2x_2} ({\bf v}_{1}^3)^{(1)}-\partial_{x_1}\bar{p}_1^3&:=A^{31}_{0}(x_1)+A^{31}_{2}(x_1)x_2^{2}=\partial_{x_2}\Big(A^{31}_{0}(x_1)x_2+\frac{1}{3}A^{31}_{2}(x_1)x_2^{3}\Big),\\
\mu\partial_{x_2x_2} ({\bf v}_{1}^3)^{(2)}-\partial_{x_2}\bar{p}_1^3&:=A^{32}_{1}(x_1)x_2+A^{32}_{3}(x_1)x_2^{3}=\partial_{x_2}\Big(\frac{1}{2}A^{32}_{1}(x_1)x_2^{2}+\frac{1}{4}A^{32}_{3}(x_1)x_2^{4}\Big),
\end{align*}
where $A^{3j}_{k}(x_1)$ are rational functions of $x_1$. Moreover, 
\begin{align*}
\Big|A^{31}_{0}(x_1)x_2+\frac{1}{3}A^{31}_{2}(x_1)x_2^{3}\Big|,\,\Big|\frac{1}{2}A^{32}_{1}x_2^{2}+\frac{1}{4}A^{32}_{3}x_2^{4}\Big|\leq\,C.
\end{align*}
Applying the integration by parts with respect to $x_2$, and using ${\bf w}_1^3=0$ on $\partial D_{1}\cup \partial D_{2}$, H\"older's inequality, and \eqref{w1Dw1}, gives
\begin{align*}
&\left|\int_{\Omega_{r_{0}}} (\mu\partial_{x_2x_2} ({\bf v}_{1}^3)^{(j)}-\partial_{x_j}\bar{p}_1^3)({\bf w}_{1}^{3})^{(j)}\mathrm{d}x\right|\\
\leq&\,C\int_{\Omega_{r_{0}}}|\partial_{x_2}({\bf w}_{1}^{3})^{(j)}|\mathrm{d}x+C\left(\int_{\Omega}|\nabla {\bf w}_{1}^{3}|^2\mathrm{d}x\right)^{1/2}
\leq\,C\left(\int_{\Omega}|\nabla {\bf w}_{1}^{3}|^2\mathrm{d}x\right)^{1/2},\quad\,j=1,2.
\end{align*}
For other terms $\partial_{x_1x_1}({\bf v}_{1}^3)^{(j)}$, since \eqref{v1311}--\eqref{v1322}, we still use the integration by parts with respect to $x_1$ as  \eqref{estf1333} in Lemma \ref{lem_energyw12}.
\begin{align*}
\left|\mu\int_{\Omega_{r_{0}}} ({\bf w}_{1}^{3})^{(j)}\partial_{x_1x_1}({\bf v}^{3}_{1})^{(j)}\mathrm{d}x\right|\leq C\left(\int_{\Omega}|\nabla {\bf w}_{1}^{3}|^2\mathrm{d}x\right)^{1/2}.
\end{align*}
Thus,
$$\Big| \int_{\Omega_{r_{0}}}\sum_{j=1}^{2}({\bf f}_{1}^{3})^{(j)}({\bf w}_{1}^{3})^{(j)}\mathrm{d}x\Big|\leq\,C\left(\int_{\Omega}|\nabla {\bf w}_{1}^{3}|^2\mathrm{d}x\right)^{1/2}.
$$
Thanks to Lemma \ref{lemmaenergy}, we finish the proof of the lemma.
\end{proof}

\begin{proof}[Proof of Proposition \ref{propu13}]
Notice that the above bounds for ${\bf v}_{1}^{3}$ and $f_1^{3}$ are the same as those of the case $\alpha=1$, following exactly the same proof as for Proposition \ref{propu11}, we have Proposition \ref{propu13} holds.
\end{proof}

\subsection{Estimates of $C_i^\alpha$}

Next, we will make use of the estimates of ${\bf v}_{1}^{\alpha}$ and $\bar{p}_1^{\alpha}$ to solve those free constants $C_i^\alpha$.
Let us define, similarly as in \cite{BLL},
\begin{align}\label{aijbj}
a_{ij}^{\alpha\beta}:=-
\int_{\partial D_j}{\boldsymbol\psi}_\beta\cdot\sigma[{\bf u}_{i}^\alpha,p_{i}^{\alpha}]\nu,\quad
b_{j}^{\beta}:=
\int_{\partial D_j}{\boldsymbol\psi}_\beta\cdot\sigma[{\bf u}_{0},p_{0}]\nu.
\end{align}
Employing the integration by parts, \eqref{aijbj} are equivalent to the volume integrals:
\begin{align}\label{defaij}
a_{ij}^{\alpha\beta}=\int_{\Omega} \left(2\mu e({\bf u}_{i}^{\alpha}), e({\bf u}_{j}^{\beta})\right)\mathrm{d}x,\quad
b_{j}^{\beta}=-\int_{\Omega} \left(2\mu e({\bf u}_{0}),e({\bf u}_{j}^\beta)\right)\mathrm{d}x.
\end{align}
Then, \eqref{equ-decompositon} can be rewritten as
\begin{align}\label{systemC2D}
\begin{cases}
\sum\limits_{\alpha=1}^{3}C_{1}^{\alpha}a_{11}^{\alpha\beta}
+\sum\limits_{\alpha=1}^{3}C_{2}^{\alpha}a_{21}^{\alpha\beta}
-b_{1}^{\beta}=0,&\\\\
\sum\limits_{\alpha=1}^{3}C_{1}^{\alpha}a_{12}^{\alpha\beta}
+\sum\limits_{\alpha=1}^{3}C_{2}^{\alpha}a_{22}^{\alpha\beta}
-b_{2}^{\beta}=0.
\end{cases}\quad \beta=1,2,3.
\end{align}
To prove Proposition \ref{lemCialpha}, we need to get the estimates of  $a_{11}^{\alpha\beta}$ and $b_1^\beta$, $\alpha,\beta=1,2,3$.

First, for $a_{11}^{\alpha\beta}$,
\begin{lemma}\label{lema11}
We have
\begin{align}
\frac{1}{C\sqrt{\varepsilon}}&\leq a_{11}^{11}\leq \frac{C}{\sqrt{\varepsilon}},\quad~~
\frac{1}{C\varepsilon^{3/2}}\leq a_{11}^{22}\leq \frac{C}{\varepsilon^{3/2}},\quad
~~\frac{1}{C\sqrt{\varepsilon}}\leq a_{11}^{33}\leq \frac{C}{\sqrt{\varepsilon}};\label{esta1111}\\
|a_{11}^{12}|,~|a_{11}^{23}|&\leq C|\log\varepsilon|,~|a_{11}^{13}|\leq C,~|a_{11}^{\alpha\beta}+a_{21}^{\alpha\beta}|\leq C,~\alpha,\beta=1,2,3,~\alpha\neq\beta;\label{esta1112}
\end{align}
and
\begin{equation}\label{estb1}
|b_1^\beta|\leq C,~\beta=1,2,3.
\end{equation}
\end{lemma}

\begin{proof}

{\bf (i) Proof of \eqref{esta1111}.} For $a_{11}^{11}$, it follows from  \eqref{defaij} and Proposition \ref{propu11} that
\begin{align*}
a_{11}^{11}\leq C\int_{\Omega}|\nabla {\bf u}_1^1|^2\mathrm{d}x\leq C\int_{\Omega_R}\frac{1}{\delta^2(x_1)}\mathrm{d}x+C\leq\frac{C}{\sqrt{\varepsilon}}.
\end{align*}
On the other hand, on account of \eqref{estv112} and Proposition \ref{propu11}, we derive
\begin{align*}
a_{11}^{11}\geq \frac{1}{C}\int_{\Omega_R}|e({\bf u}_1^1)|^2\mathrm{d}x&\geq\frac{1}{C}\int_{\Omega_R}|\partial_{x_2}({\bf v}_1^1)^{(1)}|^2\mathrm{d}x-C\geq \frac{1}{C}\int_{\Omega_R}\frac{1}{\delta^2(x_1)}\mathrm{d}x-C\geq\frac{1}{C\sqrt{\varepsilon}}.
\end{align*}

For $a_{11}^{22}$, from Proposition \ref{propu12}, we have
\begin{align*}
a_{11}^{22}\leq C\int_{\Omega}|\nabla {\bf u}_1^2|^2\mathrm{d}x\leq C\int_{\Omega_R}\left(\frac{1}{\delta(x_1)}+\frac{|x_1|}{\delta^2(x_{1})}\right)^2\mathrm{d}x+C\leq \frac{C}{\varepsilon^{3/2}}.
\end{align*}
For the lower bound, first employing the triangle inequality,
\begin{align*}
a_{11}^{22}&\geq\frac{1}{C}\int_{\Omega_R}|e({\bf u}_1^2)|^2\mathrm{d}x\geq\frac{1}{C}\int_{\Omega_R}|\partial_{x_2}({\bf v}_1^2)^{(1)}+\partial_{x_{2}}({\bf w}_1^2)^{(1)}|^2\mathrm{d}x\\
&\geq\frac{1}{C}\int_{\Omega_R}|\partial_{x_2}({\bf v}_1^2)^{(1)}|^2\mathrm{d}x-\frac{2}{C}\int_{\Omega_R}|\partial_{x_2}({\bf v}_1^2)^{(1)}||\partial_{x_2}({\bf w}_1^2)^{(1)}|\mathrm{d}x.
\end{align*}
Using \eqref{estv121}, 
\begin{align*}
\int_{\Omega_R}|\partial_{x_2}({\bf v}_1^2)^{(1)}|^2\mathrm{d}x&\geq\frac{1}{C}\int_{\Omega_R}\frac{x_{1}^{2}x_2^2}{\delta^6(x_{1})}\mathrm{d}x\geq \frac{1}{C}\int_{|x_1|<R}\frac{x_{1}^{2}}{\delta^3(x_{1})}\mathrm{d}x_1\\
&\geq\frac{1}{C}\int_{0}^R\frac{r^2}{(\varepsilon+r^2)^3}\ dr\geq\frac{1}{C\varepsilon^{3/2}}\int_0^{\frac{R}{\sqrt{\varepsilon}}}\frac{r^2}{(1+r^2)^3}\ dr\geq\frac{1}{C\varepsilon^{3/2}},
\end{align*}
while, by virtue of Proposition \ref{propu12},
\begin{equation*}
\int_{\Omega_R}|\partial_{x_2}({\bf v}_1^2)^{(1)}||\partial_{x_2}({\bf w}_1^2)^{(1)}|\mathrm{d}x\leq\int_{\Omega_R}\frac{|x_1|}{\delta^{5/2}(x_{1})}\mathrm{d}x\leq\frac{C}{\sqrt\varepsilon}.
\end{equation*}
Hence, $a_{11}^{22}\geq\frac{1}{C\varepsilon^{3/2}}$.

For $a_{11}^{33}$, using Proposition \ref{propu13}, 
\begin{equation*}
a_{11}^{33}\leq C\int_{\Omega}|\nabla {\bf u}_1^3|^2\mathrm{d}x\leq C\int_{\Omega_R}\frac{1}{\delta^2(x_1)}\mathrm{d}x+C\leq\frac{C}{\sqrt{\varepsilon}}.
\end{equation*}
For the lower bound,  in view of \eqref{v1321} and Proposition \ref{propu13}, 
\begin{align*}
a_{11}^{33}&\geq \frac{1}{C}\int_{\Omega_R}|\partial_{x_2}({\bf v}_1^3)^{(1)}|^2\mathrm{d}x-C\geq\frac{1}{C}\int_{\Omega_R}\frac{4x_2^2}{\delta^4(x_1)}\Big(1-\frac{4x_{1}^2}{\delta(x_{1})}\Big)^2\mathrm{d}x-C\\
&\geq\frac{1}{C}\int_{|x_1|<\frac{\sqrt{\varepsilon}}{6}}\frac{1}{\delta(x_1)}\mathrm{d}x_1-C\geq\frac{1}{C\sqrt{\varepsilon}}.
\end{align*}
Thus, the estimate \eqref{esta1111} is proved.

{\bf (ii) Proof of \eqref{esta1112}--\eqref{estb1}.} 
By using \eqref{defaij}, we have 
\begin{align*}
a_{11}^{12}&=\int_{\Omega_R} \left(2\mu e({\bf u}_{1}^{1}), e({\bf u}_{1}^2)\right)\mathrm{d}x+C\\
&=\int_{\Omega_R} \left(2\mu e({\bf v}_{1}^1), e({\bf v}_{1}^2)\right)\mathrm{d}x+\int_{\Omega_R} \left(2\mu e({\bf v}_{1}^1), e({\bf w}_{1}^2)\right)\mathrm{d}x\\
&\quad+\int_{\Omega_R} \left(2\mu e({\bf w}_{1}^1), e({\bf v}_{1}^2)\right)\mathrm{d}x+\int_{\Omega_R} \left(2\mu e({\bf w}_{1}^1), e({\bf w}_{1}^2)\right)\mathrm{d}x+C.
\end{align*}
It follows from Propositions \ref{propu11} and \ref{propu12} that
\begin{align*}
\left|\int_{\Omega_R} \left(2\mu e({\bf v}_{1}^1), e({\bf w}_{1}^2)\right)\mathrm{d}x\right|\leq\int_{|x_1|\leq R}\frac{C}{\sqrt{\delta(x_1)}}\ dx_1\leq C|\log\varepsilon|,
\end{align*}
\begin{align*}
\left|\int_{\Omega_R} \left(2\mu e({\bf w}_{1}^1), e({\bf v}_{1}^2)\right)\mathrm{d}x\right|\leq\int_{|x_1|\leq R}\frac{C|x_1|}{\delta(x_1)}\ dx_1\leq C|\log\varepsilon|,
\end{align*}
and 
\begin{align*}
\left|\int_{\Omega_R} \left(2\mu e({\bf w}_{1}^1), e({\bf w}_{1}^2)\right)\mathrm{d}x\right|\leq \int_{|x_1|\leq R}C\sqrt{\delta(x_1)}\ dx_1\leq C.
\end{align*}
In view of  \eqref{estv112}, \eqref{estv1122} and \eqref{estv1211}--\eqref{estv121}, one can see that the the biggest term in $(2\mu e({\bf v}_{1}^1), e({\bf v}_{1}^2))$ is
\begin{align*}
\partial_{x_2}({\bf v}_1^1)^{(1)}\cdot\partial_{x_2}({\bf v}_1^2)^{(1)}=\frac{12x_1x_2}{\delta^3(x_1)},
\end{align*} 
which is an odd function with respect to $x_1$ and thus the integral is $0$. The integral of the rest terms is bounded by $C|\log\varepsilon|$. Hence, we derive 
$$|a_{11}^{12}|\leq C|\log\varepsilon|.$$
Similarly, the biggest term in $(2\mu e({\bf v}_{1}^2), e({\bf v}_{1}^3))$ is
\begin{align*}
&\partial_{x_2}({\bf v}_1^2)^{(1)}\cdot\partial_{x_2}({\bf v}_1^3)^{(1)}\\
&=\frac{12x_{1}}{\delta^{2}(x_{1})}k(x)\left(\frac{9}{2}k(x)+2\left(\frac{1}{\delta(x_1)}-\frac{4x_1^2}{\delta^2(x_1)}\right)k(x)-20k^{3}(x)+\frac{1}{2}\right),
\end{align*}
which is an odd function with respect to $x_1$. The integral of the rest terms is bounded by $C|\log\varepsilon|$ and thus
$$|a_{11}^{23}|\leq C|\log\varepsilon|.$$

For $a_{11}^{13}$, we obtain from Propositions \ref{propu11} and \ref{propu13} that
\begin{align*}
&\left|\int_{\Omega_R} \left(2\mu e({\bf v}_{1}^1), e({\bf w}_{1}^3)\right)\mathrm{d}x\right|+\left|\int_{\Omega_R} \left(2\mu e({\bf w}_{1}^1), e({\bf v}_{1}^3)\right)\mathrm{d}x\right|\\
&\quad+\left|\int_{\Omega_R} \left(2\mu e({\bf w}_{1}^1), e({\bf w}_{1}^3)\right)\mathrm{d}x\right|\leq C.
\end{align*}
From \eqref{estv112}, \eqref{estv1122} and \eqref{v1311}--\eqref{v1322}, it follows that the biggest term in $(2\mu e({\bf v}_{1}^1), e({\bf v}_{1}^3))$ is
\begin{align*}
\partial_{x_2}({\bf v}_1^1)^{(1)}\cdot\partial_{x_2}({\bf v}_1^3)^{(1)}=\frac{1}{\delta(x_1)}\left(\frac{9}{2}k(x)+2\left(\frac{1}{\delta(x_1)}-\frac{4x_1^2}{\delta^2(x_1)}\right)k(x)-20k^{3}(x)+\frac{1}{2}\right).
\end{align*} 
A direct calcalution gives 
\begin{align*}
\left|\int_{\Omega_R}\partial_{x_2}({\bf v}_1^1)^{(1)}\cdot\partial_{x_2}({\bf v}_1^3)^{(1)}\right|\leq C.
\end{align*}
The integral of the rest terms is bounded by $C$. Therefore, we obtain
$$|a_{11}^{13}|\leq C.$$
The rest terms in \eqref{esta1112} and \eqref{estb1} are easy. We omit the details here. 
\end{proof}

\begin{proof}[Proof of Proposition \ref{lemCialpha}.]
First, by using the trace theorem, we can prove the boundedness of $C_i^\alpha$, similarly as in \cite{BLL}.
To solve $|C_1^\alpha-C_2^\alpha|$, we use the first three equations:
\begin{align*}
a_{11}(C_{1}-C_{2})=f:=b_1-(a_{11}+a_{21})C_{2},
\end{align*}
where $a_{ij}:=(a_{ij}^{\alpha\beta})_{\alpha,\beta=1}^3$, $C_{i}:=(C_{i}^{1}, C_i^2, C_i^3)^{\mathrm{T}}$, and  $b_1:=(b_1^1, b_1^2, b_1^3)^{\mathrm{T}}$. It is known from \cite{BLL} that the matrix $a_{11}=(a_{11}^{\alpha\beta})_{\alpha,\beta=1}^3$  is positive definite. By virtue of Cramer's rule and Lemma \ref{lema11}, we obtain
\begin{align*}
|C_{1}^1-C_{2}^{1}|&=\left|\frac{1}{\det a_{11}}\left(
f^{1}\begin{vmatrix}
a_{11}^{22} & a_{11}^{23} \\\\
a_{11}^{32} & a_{11}^{33}
\end{vmatrix}-
f^{2}\begin{vmatrix}
a_{11}^{12} & a_{11}^{13} \\\\
a_{11}^{32} & a_{11}^{33}
\end{vmatrix}+
f^{3}\begin{vmatrix}
a_{11}^{12} & a_{11}^{13} \\\\
a_{11}^{22} & a_{11}^{23}
\end{vmatrix}
\right)\right|\\
&\leq C\varepsilon^{5/2}\left(|f^1|\Big|a_{11}^{22}a_{11}^{33}\Big|+|f^2|\Big|a_{11}^{12}a_{11}^{33}\Big|+|f^3|\Big|a_{11}^{22}a_{11}^{13}\Big|\right)\leq C\sqrt{\varepsilon}.
\end{align*}
Similarly,  we have 
\begin{align*}
|C_{1}^2-C_{2}^{2}|\leq C\varepsilon^{3/2},\quad
|C_{1}^3-C_{2}^{3}|\leq  C\sqrt{\varepsilon}.
\end{align*}
\end{proof}

\section{The Proof of Theorem \ref{mainthm2}: Lower Bounds}\label{sec5}

In this section, we prove the lower bounds of $\nabla{\bf u}$ on the segment $\overline{P_{1}P_{2}}$ in dimension two, following the method in \cite{Li2021}. Recalling that $({\bf u}^{*}, p^{*})$ verifies \eqref{maineqn touch}, we decompose
$${\bf u}^{*}=\sum_{\alpha=1}^{3}C_{*}^{\alpha}{\bf u}^{*\alpha}+{\bf u}_{0}^{*},\quad\mbox{and}~ p^{*}=\sum_{\alpha=1}^{3}C_{*}^{\alpha}p^{*\alpha}+p_{0}^{*},$$
where $({\bf u}^{*\alpha},p^{*\alpha})$ satisfies
\begin{equation}\label{def valpha*}
\begin{cases}
\nabla\cdot\sigma[{\bf u}^{*\alpha},p^{*\alpha}]=0,\quad\nabla\cdot {\bf u}^{*\alpha}=0,&\mathrm{in}~\Omega^{0},\\
{\bf u}^{*\alpha}={\boldsymbol\psi}_{\alpha},&\mathrm{on}~\partial{D}_{1}^{0}\cup\partial{D}_{2}^{0},\\
{\bf u}^{*\alpha}=0,&\mathrm{on}~\partial{D},
\end{cases}
\end{equation}
and $({\bf u}_{0}^{*},p_{0}^{*})$ satisfies
\begin{equation}\label{equ_v3*}
\begin{cases}
\nabla\cdot\sigma[{\bf u}_0^{*},p_0^{*}]=0,\quad\nabla\cdot {\bf u}_0^{*}=0,&\mathrm{in}~\Omega^{0},\\
{\bf u}_{0}^{*}=0,&\mathrm{on}~\partial{D}_{1}^{0}\cup\partial{D_{2}^{0}},\\
{\bf u}_{0}^{*}={\boldsymbol\varphi},&\mathrm{on}~\partial{D}.
\end{cases}
\end{equation}

In view of the definition of the blow-up factors \eqref{blowupfactor}, we define, similarly,  
\begin{align}\label{blowupfactor1}
\tilde b_{j}^{\beta}:=
\int_{\partial D_j}{\boldsymbol\psi}_\beta\cdot\sigma[{\bf u}_{b},p_{b}]\nu,
\end{align}
where 
\begin{equation*}
{\bf u}_{b}:=\sum_{\alpha=1}^{3}C_{2}^{\alpha}{\bf u}^{\alpha}+{\bf u}_{0},\quad p_{b}:=\sum_{\alpha=1}^{3}C_{2}^{\alpha}p^{\alpha}+p_{0},
\end{equation*}
and
$${\bf u}^{\alpha}:={\bf u}_{1}^{\alpha}+{\bf u}_{2}^{\alpha},\quad p^{\alpha}:=p_{1}^{\alpha}+p_{2}^{\alpha}$$ satisfy
\begin{equation}\label{def valpha}
\begin{cases}
\nabla\cdot\sigma[{\bf u}^{\alpha},p^{\alpha}]=0,\quad\nabla\cdot {\bf u}^{\alpha}=0,&\mathrm{in}~\Omega,\\
{\bf u}^{\alpha}={\boldsymbol\psi}_{\alpha},&\mathrm{on}~\partial{D}_{1}\cup\partial D_{2},\\
{\bf u}^{\alpha}=0,&\mathrm{on}~\partial{D},
\end{cases}
\end{equation}
and $({\bf u}_{0},p_{0})$ satisfies
\begin{equation*}
\begin{cases}
\nabla\cdot\sigma[{\bf u}_0,p_0]=0,\quad\nabla\cdot {\bf u}_0=0,&\mathrm{in}~\Omega,\\
{\bf u}_{0}=0,&\mathrm{on}~\partial{D}_{1}\cup\partial{D}_{2},\\
{\bf u}_{0}={\boldsymbol\varphi},&\mathrm{on}~\partial{D}.
\end{cases}
\end{equation*}
Then, comparing \eqref{blowupfactor} with \eqref{blowupfactor1}, we have
\begin{prop}\label{protildeb}
For $\beta=1,2,3$, we have 
\begin{equation}\label{convtilbe2D}
|\tilde b_{1}^{\beta}[{\boldsymbol\varphi}]-\tilde b_{1}^{*\beta}[{\boldsymbol\varphi}]|\leq C\sqrt{\varepsilon}.
\end{equation}
\end{prop}

To prove Proposition \ref{protildeb}, we need the convergence of ${\bf u}^\alpha$, $p^\alpha$, ${\bf u}_0$, $p_0$ and $C_{2}^\alpha$. Set
$$\rho_{2}^{\alpha}(\varepsilon):=\begin{cases}
	\varepsilon^{1/2},\quad\alpha=1,3,\\
	\varepsilon^{1/4},\quad\alpha=2.
	\end{cases}$$

\subsection{Convergence of ${\bf u}_1^\alpha$ and $p_1^\alpha$} 

For $\alpha=1,2,3$, let us define $({\bf u}_{1}^{*\alpha},p_1^{*\alpha})$ satisfying
\begin{equation}\label{equ_ui*alpha}
\begin{cases}
\nabla\cdot\sigma[{\bf u}_{1}^{*\alpha},p_{1}^{*\alpha}]=0,\quad\nabla\cdot {\bf u}_{1}^{*\alpha}=0,&\mathrm{in}~\Omega^0,\\
{\bf u}_{1}^{*\alpha}={\boldsymbol\psi}_{\alpha},&\mathrm{on}~\partial{D}_{1}^0\setminus\{0\},\\
{\bf u}_{1}^{*\alpha}=0,&\mathrm{on}~\partial{{D}_{2}^0}\cup\partial{D}.
\end{cases}
\end{equation}	
For each $\alpha$, we will prove that ${\bf u}_{1}^{\alpha}\rightarrow {\bf u}_{1}^{*\alpha}$ when $\varepsilon\rightarrow 0$, with proper convergence rates.
Define  the auxiliary function $k^{*}(x)$ as the limit of $k(x)$ as $\varepsilon\rightarrow0$. Namely, $k^{*}(x)=\frac{1}{2}$ on $\partial{D}_{1}^{0}$, $k^{*}(x)=-\frac{1}{2}$ on $\partial{D}_{2}^{0}\cup\partial{D}$ and
\begin{equation*}
k^{*}(x)=\frac{x_{2}}{|x_1|^2},\quad\hbox{in}\ \Omega_{2R}^{0},\quad \|k^{*}(x)\|_{C^{2}(\Omega^{0}\setminus\Omega_{R}^{0})}\leq\,C,
\end{equation*}
where $\Omega_r^{0}:=\left\{(x_1,x_{2})\in \mathbb{R}^{2}~\big|~-\frac{|x_1|^2}{2}<x_{2}<\frac{|x_1|^2}{2},~|x_1|<r\right\}$, $r<R$. Then we have the following lemma.

\begin{lemma}\label{lem difference v11}
Let $({\bf u}_{1}^{\alpha},p_1^{\alpha})$ and $({\bf u}_{1}^{*\alpha},p_1^{*\alpha})$ satisfy  \eqref{equ_v12D} and \eqref{equ_ui*alpha}, respectively. Then the following boundary estimates hold
\begin{align}\label{boundaryup}
|\nabla({\bf u}_{1}^{\alpha}-{\bf u}_{1}^{*\alpha})(x)|_{x\in\partial D}+|(p_{1}^{\alpha}-p_{1}^{*\alpha})(x)|_{x\in\partial D}&\leq
\,C\rho_{2}^{\alpha}(\varepsilon). 
\end{align}
\end{lemma}

\begin{proof}
Notice that $({\bf u}_{1}^{\alpha}-{\bf u}_{1}^{*\alpha},p_1^{\alpha}-p_1^{*\alpha})$ satisfies
\begin{equation*}
\begin{cases}
\nabla\cdot\sigma[{\bf u}_{1}^{\alpha}-{\bf u}_{1}^{*\alpha},p_{1}^{\alpha}-p_{1}^{*\alpha}]=0,&\mathrm{in}~V:=D\setminus\overline{D_{1}\cup D_{2}\cup D_{1}^{0}\cup D_{2}^{0}},\\
\nabla\cdot ({\bf u}_{1}^{\alpha}-{\bf u}_{1}^{*\alpha})=0,&\mathrm{in}~V,\\
{\bf u}_{1}^{\alpha}-{\bf u}_{1}^{*\alpha}={\boldsymbol\psi}_{\alpha}-{\bf u}_{1}^{*\alpha},&\mathrm{on}~\partial{D}_{1}\setminus D_{1}^{0},\\
{\bf u}_{1}^{\alpha}-{\bf u}_{1}^{*\alpha}=-{\bf u}_{1}^{*\alpha},&\mathrm{on}~\partial{D}_{2}\setminus D_{2}^{0},\\
{\bf u}_{1}^{\alpha}-{\bf u}_{1}^{*\alpha}={\bf u}_{1}^{\alpha}-{\boldsymbol\psi}_{\alpha},&\mathrm{on}~\partial{D}_{1}^{0}\setminus(D_{1}\cup\{0\}),\\
{\bf u}_{1}^{\alpha}-{\bf u}_{1}^{*\alpha}={\bf u}_{1}^{\alpha},&\mathrm{on}~\partial{D}_{2}^{0}\setminus D_{2},\\
{\bf u}_{1}^{\alpha}-{\bf u}_{1}^{*\alpha}=0,&\mathrm{on}~\partial{D}.
\end{cases}
\end{equation*}
We will prove the case that $\alpha=1$  for instance, since the other cases are similar. 

If $x\in\partial{D}_{1}\setminus D_{1}^{0}\subset\Omega^{0}\setminus\Omega_{R}^{0}$, then we notice that the point $(x_1,x_{2}-\varepsilon/2)\in\Omega^{0}\setminus\Omega_{R}^{0}$. By using mean value theorem, 
\begin{align}\label{partial D11}
|({\bf u}_{1}^{1}-{\bf u}_{1}^{*1})(x_1,x_{2})|&=|(\boldsymbol\psi_{1}-{\bf u}_{1}^{*1})(x_1,x_{2})|=|{\bf u}_{1}^{*1}(x_1,x_{2}-\varepsilon/2)-{\bf u}_{1}^{*1}(x_1,x_{2})|\leq C\varepsilon.
\end{align}
Let
$$\mathcal{C}_{r}:=\left\{x\in\mathbb R^{2}\big| |x_1|<r,~-\frac{\varepsilon}{2}-r^2\leq x_{2}\leq\frac{\varepsilon}{2}+r^2\right\},\quad r<R.$$
If $x\in\partial{D}_{1}^{0}\setminus(D_{1}\cup\mathcal{C}_{\varepsilon^{\theta}})$, where $0<\theta<1$ is some constant to be determined later, then by mean value theorem  and Proposition \ref{propu11}, we have
\begin{align}\label{partial D11*}
|({\bf u}_{1}^{1}-{\bf u}_{1}^{*1})(x_1,x_{2})|&=|({\bf u}_{1}^{1}-\boldsymbol\psi_{1})(x_1,x_{2})|=|{\bf u}_{1}^{1}(x_1,x_{2})-{\bf u}_{1}^{1}(x_1,x_{2}+\varepsilon/2)|\nonumber\\
&\leq\frac{C\varepsilon}{\varepsilon+|x_1|^{2}}\leq C\varepsilon^{1-2\theta}.
\end{align}
Similarly, for $x\in\partial{D}_{2}\setminus D_{2}^{0}$, 
\begin{align}\label{partial D21}
|({\bf u}_{1}^{1}-{\bf u}_{1}^{*1})(x_1,x_{2})|\leq C\varepsilon,
\end{align}
and for $x\in\partial{D}_{2}^{0}\setminus(D_{2}\cup\mathcal{C}_{\varepsilon^{\theta}})$, we have
\begin{align}\label{partial D21*}
|({\bf u}_{1}^{1}-{\bf u}_{1}^{*1})(x_1,x_{2})|\leq\,C\varepsilon^{1-2\theta}.
\end{align}

Define another auxillary function
\begin{align*}
{\bf v}_{1}^{*1}=\boldsymbol\psi_{1}\Big(k^*(x)+\frac{1}{2}\Big)+\boldsymbol\psi_{2}x_{1}\Big((k^*(x))^2-\frac{1}{4}\Big),\quad\hbox{in}\ \Omega_{2R}^*,
\end{align*}
and $\|{\bf v}_{1}^{*1}\|_{C^{2}(\Omega^0\setminus\Omega_{R}^{0})}\leq\,C$. Recalling \eqref{v11}, we have 
\begin{equation*}
\left|\partial_{x_{2}}({\bf v}_{1}^{1}-{\bf v}_{1}^{*1})\right|\leq \frac{C\varepsilon}{|x_1|^{2}(\varepsilon+|x_1|^{2})}.
\end{equation*}
Then for $x\in\Omega_{R}^{0}$ with $|x_1|=\varepsilon^{\theta}$, combining with Proposition \ref{propu11}, we obtain
\begin{align}\label{partial x2 u11}
\left|\partial_{x_{2}}({\bf u}_{1}^{1}-{\bf u}_{1}^{*1})(x_1,x_{2})\right|
&=\left|\partial_{x_{2}}({\bf u}_{1}^{1}-{\bf v}_{1}^{1})+\partial_{x_{2}}({\bf v}_{1}^{1}-{\bf v}_{1}^{*1})
+\partial_{x_{2}}({\bf v}_{1}^{*1}-{\bf u}_{1}^{*1})\right|(x_1,x_{2})\nonumber\\
&\leq  C\left(1+\frac{1}{\varepsilon^{4\theta-1}}\right).
\end{align}
Thus, for $x\in\Omega_{R}^{0}$ with $|x_1|=\varepsilon^{\theta}$, by using the triangle inequality, \eqref{partial D21*}, the mean value theorem, and \eqref{partial x2 u11}, 
\begin{align}\label{es v11*2D}
\left|({\bf u}_{1}^{1}-{\bf u}_{1}^{*1})(x_1,x_{2})\right|&\leq\left|({\bf u}_{1}^{1}-{\bf u}_{1}^{*1})(x_1,x_{2})-({\bf u}_{1}^{1}-{\bf u}_{1}^{*1})(x_1,-h_{2}(x_1)\right|
+C\varepsilon^{1-2\theta}\nonumber\\
&\leq\left|\partial_{x_{2}}({\bf u}_{1}^{1}-{\bf u}_{1}^{*1})\right|\cdot(h_{1}(x_1)+h_{2}(x_1))\Big|_{|x_1|=\varepsilon^{\theta}}+C\varepsilon^{1-2\theta}\nonumber\\
&\leq C\left(1+\frac{1}{\varepsilon^{4\theta-1}}\right)\varepsilon^{2\theta}+C\varepsilon^{1-2\theta}=C\left(\varepsilon^{2\theta}+\varepsilon^{1-2\theta}\right).
\end{align}	
By taking $2\theta=1-2\theta$, we get $\theta=\frac{1}{4}$. Substituting it into \eqref{partial D11*}, \eqref{partial D21*} and \eqref{es v11*2D}, and using \eqref{partial D11}, \eqref{partial D21}, and ${\bf u}_{1}^{1}-{\bf u}_{1}^{*1}=0$ on $\partial D$, we obtain
\begin{align*}
|{\bf u}_{1}^{1}-{\bf u}_{1}^{*1}|\leq C\varepsilon^{1/2},\quad\mbox{on}~\partial{(V\setminus\mathcal{C}_{\varepsilon^{1/4}})}.
\end{align*}
Applying the maximum modulus for Stokes systems in $V\setminus\mathcal{C}_{\varepsilon^{1/4}}$ (see, for example, \cite{Lady,Mazya}), we obtain
\begin{align*}
|({\bf u}_{1}^{1}-{\bf u}_{1}^{*1})(x)|\leq C\varepsilon^{1/2},\quad\mbox{in}~ V\setminus \mathcal{C}_{\varepsilon^{1/4}}.
\end{align*}
In view of the  standard boundary  estimates for Stokes systems with ${\bf u}_{1}^{\alpha}-{\bf u}_{1}^{*\alpha}=0$ on $\partial{D}$ (see, for example \cite{Kratz,Mazya}), we obtain \eqref{boundaryup} with $\alpha=1$. The proof is finished.
\end{proof}

\subsection{Convergence of $\frac{C_{1}^{\alpha}+C_{2}^{\alpha}}{2}-C_{*}^{\alpha}$}
Next, we prove the convergence of $\frac{C_{1}^{\alpha}+C_{2}^{\alpha}}{2}-C_{*}^{\alpha}$ under the symmetric conditions $({\rm S_{H}})$ and $({\rm S_{\boldsymbol\varphi}})$ for simplicity. It follows from  \eqref{systemC2D} that
\begin{align}\label{C1C2_d}
\left\{
\begin{aligned}
\sum_{\alpha=1}^{3}C_{1}^{\alpha}a_{11}^{\alpha\beta}+\sum_{\alpha=1}^{3}C_{2}^{\alpha}a_{21}^{\alpha\beta}-b_{1}^{\beta}&=0,\\
\sum_{\alpha=1}^{3}C_{1}^{\alpha}a_{12}^{\alpha\beta}+\sum_{\alpha=1}^{3}C_{2}^{\alpha}a_{22}^{\alpha\beta}-b_{2}^{\beta}&=0,
\end{aligned}
\right.\quad\quad~~\beta=1,2,3.
\end{align}
By a rearrangement, from the first line of \eqref{C1C2_d}, we have
\begin{equation*}
\sum_{\alpha=1}^{3}(C_{1}^{\alpha}+C_{2}^{\alpha})a_{11}^{\alpha\beta}+\sum_{\alpha=1}^{3}C_{2}^{\alpha}(a_{21}^{\alpha\beta}-a_{11}^{\alpha\beta})-b_{1}^{\beta}=0,
\end{equation*}
and
\begin{equation*}
\sum_{\alpha=1}^{3}(C_{1}^{\alpha}+C_{2}^{\alpha})a_{21}^{\alpha\beta}+\sum_{\alpha=1}^{3}C_{1}^{\alpha}(a_{11}^{\alpha\beta}-a_{21}^{\alpha\beta})-b_{1}^{\beta}=0.
\end{equation*}
Adding these two equations together leads to
\begin{equation}\label{C1C2_1}
\sum_{\alpha=1}^{3}(C_{1}^{\alpha}+C_{2}^{\alpha})(a_{11}^{\alpha\beta}+a_{21}^{\alpha\beta})+\sum_{\alpha=1}^{3}(C_{1}^{\alpha}-C_{2}^{\alpha})(a_{11}^{\alpha\beta}-a_{21}^{\alpha\beta})-2b_{1}^{\beta}=0.
\end{equation}
Similarly, for the second equation of \eqref{C1C2_d}, we have 
\begin{equation}\label{C1C2_22}
\sum_{\alpha=1}^{3}(C_{1}^{\alpha}+C_{2}^{\alpha})(a_{12}^{\alpha\beta}+a_{22}^{\alpha\beta})+\sum_{\alpha=1}^{3}(C_{1}^{\alpha}-C_{2}^{\alpha})(a_{12}^{\alpha\beta}-a_{22}^{\alpha\beta})-2b_{2}^{\beta}=0.
\end{equation}
Then, adding \eqref{C1C2_1} and \eqref{C1C2_22} together, and dividing by two, we obtain
\begin{equation}\label{C1+C2_1}
\sum_{\alpha=1}^{3}\frac{C_{1}^{\alpha}+C_{2}^{\alpha}}{2}a^{\alpha\beta}
+\sum_{\alpha=1}^{3}\frac{C_{1}^{\alpha}-C_{2}^{\alpha}}{2}(a_{11}^{\alpha\beta}-a_{22}^{\alpha\beta}+a_{12}^{\alpha\beta}-a_{21}^{\alpha\beta})-(b_{1}^{\beta}+b_{2}^{\beta})=0,
\end{equation}
where 
\begin{equation}\label{def_a}
a^{\alpha\beta}=\sum_{i,j=1}^{2}a_{ij}^{\alpha\beta}=-\left(\int_{\partial{D}_{1}}{\boldsymbol\psi}_\beta\cdot\sigma[{\bf u}^\alpha,p^{\alpha}]\nu+\int_{\partial{D}_{2}}{\boldsymbol\psi}_\beta\cdot\sigma[{\bf u}^\alpha,p^{\alpha}]\nu\right),
\end{equation}
and ${\bf u}^{\alpha}:={\bf u}_{1}^{\alpha}+{\bf u}_{2}^{\alpha}$, $p^{\alpha}:=p_{1}^{\alpha}+p_{2}^{\alpha}$ satisfy
\eqref{def valpha}.

On the other hand, from the third line of \eqref{maineqn touch}, we deduce
\begin{align}\label{equ_C*alpha}
&\sum_{\alpha=1}^{3}C_{*}^{\alpha}a_{*}^{\alpha\beta}-(b_{1}^{*\beta}+b_{2}^{*\beta})=0,\quad\beta=1,2,3,
\end{align}
where, corresponding to \eqref{def_a},
\begin{align}\label{b*jbeta}
a_{*}^{\alpha\beta}=-\int_{\partial{D}_{1}^{0}\cup\partial{D}_{2}^{0}}{\boldsymbol\psi}_\beta\cdot\sigma[{\bf u}^{*\alpha},p^{*\alpha}]\nu,\quad b_{j}^{*\beta}=\int_{\partial{D}_{j}^{0}}{\boldsymbol\psi}_\beta\cdot\sigma[{\bf u}_0^*,p_0^{*}]\nu,\quad j=1,2.
\end{align}
Here we remark that
$${\bf u}^{*\alpha}:={\bf u}_{1}^{*\alpha}+{\bf u}_{2}^{*\alpha},\quad\mbox{and}~ p^{*\alpha}:=p_{1}^{*\alpha}+p_{2}^{*\alpha}.$$

\begin{lemma}\label{es b1 b1* beta=1}
Let $b_{1}^{\beta}$ and $b_{1}^{*\beta}$ be defined in \eqref{aijbj} and \eqref{b*jbeta}, respectively. Then we have 	
\begin{equation}\label{est v0-v0*}
\left|b_{1}^{\beta}[{\boldsymbol\varphi}]-b_{1}^{*\beta}[{\boldsymbol\varphi}]\right|\leq C\,\rho_{2}^{\beta}(\varepsilon)\|{\boldsymbol\varphi}\|_{L^{\infty}(\partial D)},
\end{equation}
where $b_{1}^{\beta}[{\boldsymbol\varphi}]$ is defined by \eqref{aijbj}, and $b_{1}^{*\beta}[{\boldsymbol\varphi}]$  by \eqref{b*jbeta}.
\end{lemma}

\begin{proof}
Recalling the definitions \eqref{defaij}, \eqref{equ_v12D} and \eqref{equ_v32D}, using the integration by parts, we have 
\begin{align*}
b_{1}^{\beta}[{\boldsymbol\varphi}]=\int_{\partial D_{1}}{\boldsymbol\psi}_\beta\cdot\sigma[{\bf u}_0,p_0]\nu&=-\int_{\partial D}{\boldsymbol\varphi}\cdot\sigma[{\bf u}_1^\beta,p_1^{\beta}]\nu.
\end{align*}
Similarly, in view of \eqref{def valpha*}, \eqref{equ_v3*}, and \eqref{b*jbeta}, we deduce 
\begin{align*}
b_{1}^{*\beta}[{\boldsymbol\varphi}]=-\int_{\partial D}{\boldsymbol\varphi}\cdot\sigma[{\bf u}_1^{*\beta},p_1^{*\beta}]\nu.
\end{align*}
Thus, 
\begin{equation*}
b_{1}^{\beta}[{\boldsymbol\varphi}]-b_{1}^{*\beta}[{\boldsymbol\varphi}]=-\int_{\partial D}{\boldsymbol\varphi}\cdot\sigma[{\bf u}_1^\beta-{\bf u}_1^{*\beta},p_1^{\beta}-p_1^{*\beta}]\nu.
\end{equation*}
Applying Lemma \ref{lem difference v11} implies \eqref{est v0-v0*}.
\end{proof}

\begin{lemma}\label{lem valpha1}
Let ${\bf u}^{\alpha}$ and ${\bf u}^{*\alpha}$ be defined by \eqref{def valpha} and \eqref{def valpha*}, respectively, $\alpha=1,2,3$. Then 
\begin{align}\label{est v11 v11*}
\left|\int_{\partial D_{1}}{\boldsymbol\psi}_\beta\cdot\sigma[{\bf u}^{\alpha},p^{\alpha}]\nu-\int_{\partial D_{1}^{0}}{\boldsymbol\psi}_\beta\cdot\sigma[{\bf u}^{*\alpha},p^{*\alpha}]\nu\right|\leq
C\,\rho_{2}^{\beta}(\varepsilon)\|{\boldsymbol\psi}_{\alpha}\|_{L^{\infty}(\partial D)}.
\end{align}
\end{lemma}

\begin{proof}
For $\alpha=1,2,3$, it follows from \eqref{def valpha} that
\begin{equation*}
\begin{cases}
\nabla\cdot\sigma[{\bf u}^{\alpha}-{\boldsymbol\psi}_{\alpha},p^{\alpha}]=0,\quad\nabla\cdot ({\bf u}^{\alpha}-{\boldsymbol\psi}_{\alpha})=0,&\mbox{in}~\Omega,\\
{\bf u}^{\alpha}-{\boldsymbol\psi}_{\alpha}=0,&\mbox{on}~\partial{D}_{1}\cup\partial{D}_{2},\\
{\bf u}^{\alpha}-{\boldsymbol\psi}_{\alpha}=-{\boldsymbol\psi}_{\alpha},&\mbox{on}~\partial D.
\end{cases}	
\end{equation*}
Using the integration by parts, for $\alpha=1,2,3$,
\begin{align*}
\int_{\partial D_{1}}{\boldsymbol\psi}_\beta\cdot\sigma[{\bf u}^{\alpha},p^{\alpha}]\nu=\int_{\partial D_{1}}{\boldsymbol\psi}_\beta\cdot\sigma[{\bf u}^{\alpha}-{\boldsymbol\psi}_{\alpha},p^{\alpha}]\nu=\int_{\partial D}{\boldsymbol\psi}_\alpha\cdot\sigma[{\bf u}_1^{\beta},p_1^{\beta}]\nu.
\end{align*}
Similarly,
\begin{align*}
\int_{\partial D_{1}^{0}}{\boldsymbol\psi}_\beta\cdot\sigma[{\bf u}^{*\alpha},p^{*\alpha}]\nu=\int_{\partial D}{\boldsymbol\psi}_\alpha\cdot\sigma[{\bf u}_1^{*\beta},p_1^{*\beta}]\nu.
\end{align*}
Hence,
\begin{equation*}
\int_{\partial D_{1}}{\boldsymbol\psi}_\beta\cdot\sigma[{\bf u}^{\alpha},p^{\alpha}]\nu-\int_{\partial D_{1}^{0}}{\boldsymbol\psi}_\beta\cdot\sigma[{\bf u}^{*\alpha},p^{*\alpha}]\nu=\int_{\partial D}{\boldsymbol\psi}_\alpha\cdot\sigma[{\bf u}_1^{\beta}-{\bf u}_1^{*\beta},p_1^{\beta}-p_1^{*\beta}]\nu.
\end{equation*}
Thus, by virtue of Lemma \ref{lem difference v11}, we obtain  \eqref{est v11 v11*}. 
\end{proof}

Next, we prove the convergence of $\frac{C_1^\alpha+C_2^\alpha}{2}-C_*^\alpha$.

\begin{prop}\label{propC2D}
Assume that $D_{1}\cup D_{2}$ and $D$ satisfies $({\rm S_{H}})$, and ${\boldsymbol\varphi}$ satisfies the symmetric condition $({\rm S_{{\boldsymbol\varphi}}})$. Let $C_{1}^{\alpha}, C_{2}^{\alpha}$, and $C_{*}^{\alpha}$ be defined in \eqref{ud} and \eqref{maineqn touch}, respectively. Then
\begin{equation*}	C_1^\alpha+C_2^\alpha=C_*^\alpha\equiv0,~\alpha=1,2,\quad\mbox{and}~\left|\frac{C_1^3+C_2^3}{2}-C_*^3\right|\leq C\sqrt{\varepsilon}.
\end{equation*}
\end{prop}

\begin{proof} 
In view of our hypotheses $({\rm S_{H}})$, we consider the symmetry of the domain with respect to the origin first, and observe that, for $\alpha=1,2$,
\begin{align*}
{\bf u}_{2}^{\alpha}(x)\big|_{\partial D_{1}}={\bf u}_{1}^{\alpha}(-x)\big|_{\partial D_{2}}=0,\quad{\bf u}_{2}^{\alpha}(x)\big|_{\partial D_{2}}={\bf u}_{1}^{\alpha}(-x)\big|_{\partial D_{1}}={\boldsymbol\psi}_{\alpha},
\end{align*}
and
\begin{align*}
{\bf u}_{2}^{\alpha}(x)\big|_{\partial D}={\bf u}_{1}^{\alpha}(-x)\big|_{\partial D}=0.
\end{align*}
Since $\nabla\cdot\sigma[{\bf u}_i^{\alpha},p_i^{\alpha}]=0$ in $\Omega$, we obtain $${\bf u}_{2}^{\alpha}(x)={\bf u}_{1}^{\alpha}(-x),\quad\mbox{in}~\Omega,\quad~\alpha=1,2.$$ 
Therefore, combining the definition of $a_{ij}^{\alpha\beta}$ defined in \eqref{defaij}, we have
\begin{equation}\label{origin_sym02D}
a_{11}^{\alpha\beta}=a_{22}^{\alpha\beta},\quad a_{12}^{\alpha\beta}=a_{21}^{\alpha\beta},\quad~\alpha,\beta=1,2.
\end{equation}	
For $\alpha=3$, since
\begin{align*} 
{\bf u}_{2}^{3}(x)\big|_{\partial D_{1}}=-{\bf u}_{1}^{3}(-x)\big|_{\partial D_{2}}&=0,\quad {\bf u}_{2}^{3}(x)\big|_{\partial D_{2}}=-{\bf u}_{1}^{3}(-x)\big|_{\partial D_{1}}={\boldsymbol\psi}_{3},\\
{\bf u}_{2}^{3}(x)\big|_{\partial D}&=-{\bf u}_{1}^{3}(-x)\big|_{\partial D}=0,
\end{align*}
it follows that ${\bf u}_{2}^{3}(x)=-{\bf u}_{1}^{3}(-x)$ in $\Omega$. Thus, 
\begin{align}\label{origin_sym2D}
a_{11}^{\alpha 3}=-a_{22}^{\alpha 3},~ a_{12}^{\alpha 3}=-a_{21}^{\alpha 3},\quad\alpha=1,2,
\quad\mbox{and}~ 
a_{11}^{33}=a_{22}^{33},~ a_{12}^{33}=a_{21}^{33}.
\end{align}
	
Secondly, due to the symmetry of the domain with respect to $\{x_{2}=0\}$, let us set
\begin{align*}
{\bf u}_{2}^{1}(x_1,x_{2})&=\big(
({\bf u}_{2}^{1})^{(1)}(x_1,x_{2}),
({\bf u}_{2}^{1})^{(2)}(x_1,x_{2})\big)^{\mathrm T}\\
&=\big(({\bf u}_{1}^1)^{(1)}(x_1,-x_{2}),
-({\bf u}_{1}^{1})^{(2)}(x_1,-x_{2})\big)^{\mathrm T},
\end{align*}
\begin{align*}
{\bf u}_{2}^{2}(x_1,x_{2})&=\big(
({\bf u}_{2}^{2})^{(1)}(x_1,x_{2}),
({\bf u}_{2}^{2})^{(2)}(x_1,x_{2})\big)^{\mathrm T}\\
&=\big(
-({\bf u}_{1}^2)^{(1)}(x_1,-x_{2}),
({\bf u}_{1}^{2})^{(2)}(x_1,-x_{2})\big)^{\mathrm T},
\end{align*}
and
\begin{align*}
{\bf u}_{2}^{3}(x_1,x_{2})&=\big(
({\bf u}_{2}^{3})^{(1)}(x_1,x_{2}),
({\bf u}_{2}^{3})^{(2)}(x_1,x_{2})\big)^{\mathrm T}\\
&=\big(
-({\bf u}_{1}^3)^{(1)}(x_1,-x_{2}),
({\bf u}_{1}^{3})^{(2)}(x_1,-x_{2})\big)^{\mathrm T}.
\end{align*}
Thus, for example,
\begin{align*}
\left(2\mu e({\bf u}_{2}^{1}), e({\bf u}_{2}^{2})\right)=&\,\mu\Big(2\partial_{x_1}({\bf u}_{2}^{1})^{(1)}\partial_{x_1}({\bf u}_{2}^{2})^{(1)}+2\partial_{x_2}({\bf u}_{2}^{1})^{(2)}\partial_{x_2}({\bf u}_{2}^{2})^{(2)}\\
&\quad+\big(\partial_{x_1}({\bf u}_{2}^{1})^{(2)}+\partial_{x_2}({\bf u}_{2}^{1})^{(1)}\big)\cdot\big(\partial_{x_1}({\bf u}_{2}^{2})^{(2)}+\partial_{x_2}({\bf u}_{2}^{2})^{(1)}\big)\\=&\,-\left(2\mu e({\bf u}_{1}^{1}), e({\bf u}_{1}^{2})\right),
\end{align*}
so that we have
\begin{equation}\label{12_sym2D}
a_{22}^{12}=-a_{11}^{12},\quad a_{12}^{12}=-a_{21}^{12}.
\end{equation}
Similarly, 
\begin{equation}\label{123_sym}
a_{22}^{13}=-a_{11}^{13},\quad a_{12}^{13}=-a_{21}^{13},\quad\mbox{but}~a_{22}^{23}=a_{11}^{23},\quad a_{12}^{23}=a_{21}^{23}.
\end{equation}
Then, combining \eqref{origin_sym02D}--\eqref{123_sym}, we deduce
\begin{align}\label{syma112D}
a_{11}^{12}=a_{22}^{12}=a_{12}^{12}=a_{21}^{12}=
a_{11}^{23}=a_{22}^{23}= a_{12}^{23}=a_{21}^{23}=0.
\end{align}
Thus, from \eqref{123_sym} and \eqref{syma112D},
\begin{equation}\label{symaalphabeta2D}
a^{\alpha\beta}=a^{\alpha\beta}=0,\quad\alpha,\beta=1,2,3,~\alpha\neq\beta.
\end{equation}
Similarly, 
\begin{equation}\label{symaalphabeta122D}
a_*^{\alpha\beta}=a_*^{\alpha\beta}=0,\quad\alpha,\beta=1,2,3,~\alpha\neq\beta.
\end{equation}

Finally, on account of the symmetry condition of the boundary data $({\rm S_{{\boldsymbol\varphi}}})$, we set 
\begin{align*}
{\bf u}_0(x)=-{\bf u}_0(-x).
\end{align*}
Then 
\begin{equation}\label{|symtilb2D}
b_{1}^{\alpha}=-b_{2}^{\alpha},\quad \alpha=1,2,\quad b_{1}^{3}=b_{2}^{3}.
\end{equation}
Similarly,
\begin{equation}\label{|symtilb002D}
b_{1}^{*\alpha}=-b_{2}^{*\alpha},\quad \alpha=1,2,\quad b_{1}^{*3}=b_{2}^{*3}.
\end{equation}
	
Now substituting  \eqref{syma112D} and \eqref{|symtilb2D} directly into \eqref{systemC2D}, and using Cramer's rule, we deduce
\begin{equation}\label{C13C23}
C_1^\alpha=-C_2^\alpha,\quad \alpha=1,2,\quad\mbox{and}~ C_1^3=C_2^3.
\end{equation}
Then, by using  \eqref{symaalphabeta2D} and \eqref{|symtilb2D}, equation \eqref{C1+C2_1} becomes
\begin{align}\label{C1+C2_matrix2D}
AX=B,
\end{align}
where $A=(a^{\alpha\beta})_{\alpha,\beta=1}^{3}$ with $a^{\alpha\beta}=0$, $\alpha\neq \beta$, $X=(0,0,\frac{C_1^3+C_2^3}{2})^{\mathrm T}$, and $B=(0,0,2b_{1}^{3}-(C_1^1-C_2^1)(a_{11}^{13}+a_{12}^{13}))^{\mathrm T}$. Similarly, substituting \eqref{symaalphabeta122D} and  \eqref{|symtilb002D} directly into \eqref{equ_C*alpha}, we obtain
\begin{align}\label{C1+C2_*2D}
A_*X_*=B_*,
\end{align}
where $A_*=(a_*^{\alpha\beta})_{\alpha,\beta=1}^{3}$ with $a_*^{\alpha\beta}=0$, $\alpha\neq \beta$, $X_*=(0,0,C_{*}^{3})^{\mathrm T}$, and $B_*=(0,0,2b_{1}^{*3})^{\mathrm T}$.	Combining  \eqref{C1+C2_matrix2D} and \eqref{C1+C2_*2D}, we have
\begin{equation*}
A(X-X_*)=(0,0,\mathcal{B}^3)^{\mathrm T},
\end{equation*}
where $$\mathcal{B}^3=2(b_{1}^{3}-b_{1}^{*3})-(C_1^1-C_2^1)(a_{11}^{13}+a_{12}^{13})-C_*^3(a^{33}-a_*^{33}),$$ and $(a^{\alpha\beta})_{\alpha,\beta=1}^3$ is positive definite; the details can be found in \cite{Li2021}. It follows from   Proposition \ref{lemCialpha},  \eqref{es b1 b1* beta=1} and \eqref{lem valpha1} that
\begin{equation*}
\mathcal{B}^3=O(\sqrt{\varepsilon}).
\end{equation*}
By Cramer's rule, we obtain
\begin{equation*}
\left|\frac{C_1^3+C_2^3}{2}-C_*^3\right|\leq C\sqrt{\varepsilon}.
\end{equation*}
The proof of Proposition \ref{propC2D} is complete.
\end{proof}

\begin{proof}[Proof of Proposition \ref{protildeb}]
Note that
\begin{align*}
&C_{2}^{\alpha}\int_{\partial D_{1}}{\boldsymbol\psi}_\beta\cdot\sigma[{\bf u}^{\alpha},p^{\alpha}]\nu-C_{*}^{\alpha}\int_{\partial D_{1}^{0}}{\boldsymbol\psi}_\beta\cdot\sigma[{\bf u}^{*\alpha},p^{*\alpha}]\nu\\
=&\,C_{2}^{\alpha}\Bigg(\int_{\partial D_{1}}{\boldsymbol\psi}_\beta\cdot\sigma[{\bf u}^{\alpha},p^{\alpha}]\nu-\int_{\partial D_{1}^{0}}{\boldsymbol\psi}_\beta\cdot\sigma[{\bf u}^{*\alpha},p^{*\alpha}]\nu\Bigg)+\Big(C_{2}^{\alpha}-C_{*}^{\alpha}\Big)\int_{\partial D_{1}^{0}}\frac{\partial v^{*\alpha}}{\partial \nu}\Big|_{+}\cdot{\boldsymbol\psi}_\beta.
\end{align*}
Then
\begin{align*}
\tilde b_{1}^{\beta}[{\boldsymbol\varphi}]-\tilde b_{1}^{*\beta}[{\boldsymbol\varphi}]&=\int_{\partial D_{1}}{\boldsymbol\psi}_\beta\cdot\sigma[{\bf u}_{0},p_{0}]\nu-\int_{\partial D_{1}^{0}}{\boldsymbol\psi}_\beta\cdot\sigma[{\bf u}_{0}^*,p_{0}^*]\nu\nonumber\\
&\quad+\sum_{\alpha=1}^{3}C_{2}^{\alpha}\Bigg(\int_{\partial D_{1}}{\boldsymbol\psi}_\beta\cdot\sigma[{\bf u}^{\alpha},p^{\alpha}]\nu-\int_{\partial D_{1}^{0}}{\boldsymbol\psi}_\beta\cdot\sigma[{\bf u}^{*\alpha},p^{*\alpha}]\nu\Bigg)\nonumber\\
&\quad+\sum_{\alpha=1}^{3}\Big(C_{2}^{\alpha}-C_{*}^{\alpha}\Big)\int_{\partial D_{1}^{0}}\frac{\partial v^{*\alpha}}{\partial \nu}\Big|_{+}\cdot{\boldsymbol\psi}_\beta.
\end{align*}
By virtue of Lemma \ref{es b1 b1* beta=1} and Lemma \ref{lem valpha1},  to obtain \eqref{convtilbe2D}, it suffices to prove the convergence of $C_{2}^{\alpha}-C_{*}^{\alpha}$. We have from Proposition \ref{lemCialpha} and Proposition \ref{propC2D} that
\begin{equation*}
|C_{2}^\alpha-C_*^\alpha|=\left|\frac{C_1^\alpha+C_2^\alpha}{2}-C_*^\alpha-\frac{C_1^\alpha-C_2^\alpha}{2}\right|=\frac{|C_1^\alpha-C_2^\alpha|}{2}\leq C
\begin{cases}
\sqrt{\varepsilon},~\alpha=1,\\
\varepsilon^{3/2},~\alpha=2,
\end{cases}
\end{equation*}
and
\begin{equation*}
|C_{2}^3-C_*^3|=\left|\frac{C_1^3+C_2^3}{2}-C_*^3-\frac{C_1^3-C_2^3}{2}\right|=\left|\frac{C_1^3+C_2^3}{2}-C_*^3\right|\leq C\sqrt{\varepsilon}.
\end{equation*}
Due to Lemma \ref{es b1 b1* beta=1} and Lemma \ref{lem valpha1}, we obtain \eqref{convtilbe2D}.
The proof of Proposition \ref{protildeb} is finished.
\end{proof}

\subsection{The Completion of the Proof of Theorem \ref{mainthm2}}

\begin{proof}[Proof of Theorem \ref{mainthm2}.]
We first obtain from Proposition \ref{protildeb} that, if $\tilde b_{1}^{*1}[{\boldsymbol\varphi}]\neq0$, then there exists  a sufficiently small constant $\varepsilon_0>0$ such that for $0<\varepsilon<\varepsilon_0$, 
\begin{equation}\label{tildeb1}
|\tilde b_{1}^{1}[{\boldsymbol\varphi}]|\geq\frac{1}{2}|\tilde b_{1}^{*1}[{\boldsymbol\varphi}]|>0.
\end{equation}
By virtue of \eqref{equ-decompositon} and Proposition \ref{propC2D}, we obtain
\begin{equation}\label{eqsymm}
\sum_{\alpha=1}^{2}\left(C_{1}^{\alpha}-C_{2}^{\alpha}\right)a_{11}^{\alpha\beta}=\tilde b_1^\beta,\quad\beta=1,2.
\end{equation}
It follows from Lemma \ref{lema11} and \eqref{syma112D} that
\begin{equation*}
\frac{1}{C}\varepsilon^{-2}\leq\det A\leq C\varepsilon^{-2},
\end{equation*}
where $A=(a_{11}^{\alpha\beta})_{\alpha,\beta=1}^2$. 
By Cramer's rule and \eqref{tildeb1}, we get
\begin{align}\label{lowerC11}
|C_1^1-C_2^1|=\left|\frac{1}{\det A}\tilde b_1^1[{\boldsymbol\varphi}] a_{11}^{22}+O(\varepsilon^{3/2})\right|\geq\frac{|\tilde b_{1}^{*1}[{\boldsymbol\varphi}]|}{C}\sqrt{\varepsilon}.
\end{align}
Moreover, from Proposition \ref{propu12} and Proposition \ref{lemCialpha}, it follows that
\begin{equation*}
|(C_{1}^{2}-C_{2}^{2})\nabla{\bf u}_{1}^2(0,x_2)|\leq C\varepsilon^{1/2}.
\end{equation*}
Combining  Proposition \ref{propu11} and \eqref{lowerC11} yields
\begin{align*}
|\nabla{{\bf u}}(0,x_2)|&=\left|\sum_{\alpha=1}^{2}\left(C_{1}^{\alpha}-C_{2}^{\alpha}\right)\nabla{\bf u}_{1}^{\alpha}(0,x_2)
+\nabla {\bf u}_{b}(0,x_2)\right|\\
&\geq|(C_{1}^{1}-C_{2}^{1})\nabla{\bf u}_{1}^{1}(0,x_2)|-C\geq\frac{|\tilde b_{1}^{*1}[{\boldsymbol\varphi}]|}{C\sqrt{\varepsilon}}.
\end{align*}
Theorem \ref{mainthm2}  is proved.
\end{proof}

\begin{remark}\label{rmk-p}
Recalling that 
\begin{equation*}
p=\sum_{\alpha=1}^{2}\left(C_{1}^{\alpha}-C_{2}^{\alpha}\right)p_{1}^{\alpha}+p_{b},\quad\mbox{in}~\Omega,
\end{equation*}
where $
p_{b}:=\sum_{\alpha=1}^{3}C_{2}^{\alpha}(p_{1}^{\alpha}+p_{2}^{\alpha})+p_{0}$ is bounded. To establish the lower bound of $p$, it suffices to deal with $\sum_{\alpha=1}^{2}\left(C_{1}^{\alpha}-C_{2}^{\alpha}\right)p_{1}^{\alpha}$. For this, by using \eqref{p12} and the estimate of $|\nabla q_1^2|$ in Proposition \ref{propu12}, we find that 
\begin{align}\label{estp12}
|p_1^2(0,x_2)-(q_1^2)_{R}|\geq|\bar p_1^2(0,x_2)|-|q_1^2(0,x_2)-(q_1^2)_{R}|\geq\frac{3\mu}{\varepsilon^2}-\frac{C}{\varepsilon}\geq\frac{1}{C\varepsilon^2}.
\end{align}
Similar to \eqref{tildeb1}, if $\tilde b_{1}^{*2}[{\boldsymbol\varphi}]\neq0$, then it follows from Proposition \ref{protildeb} that,  there exists  a sufficiently small constant $\varepsilon_0>0$ such that for $0<\varepsilon<\varepsilon_0$, 
\begin{equation*}
|\tilde b_{1}^{2}[{\boldsymbol\varphi}]|\geq\frac{1}{2}|\tilde b_{1}^{*2}[{\boldsymbol\varphi}]|>0.
\end{equation*}
This together with Cramer's rule and \eqref{eqsymm} yields
\begin{align*}
|C_1^2-C_2^2|=\left|\frac{1}{\det A}\tilde b_1^2[{\boldsymbol\varphi}] a_{11}^{11}+O(\varepsilon^{2})\right|\geq\frac{|\tilde b_{1}^{*2}[{\boldsymbol\varphi}]|}{C}\varepsilon^{3/2}.
\end{align*}
Combining with \eqref{estp12}, we derive
\begin{equation}\label{C2p12}
|(C_{1}^{2}-C_{2}^{2})(p_1^2(0,x_2)-(q_1^2)_{R})|\geq\frac{|\tilde b_{1}^{*2}[{\boldsymbol\varphi}]|\varepsilon^{3/2}}{C\varepsilon^2}\geq \frac{|\tilde b_{1}^{*2}[{\boldsymbol\varphi}]|}{C\sqrt{\varepsilon}}.
\end{equation}
However, in view of Propositions \ref{propu11} and \ref{lemCialpha}, we have 
\begin{equation*}
|(C_{1}^{1}-C_{2}^{1})(p_1^1(0,x_2)-(q_1^1)_{R})|\leq \frac{C\sqrt{\varepsilon}}{\varepsilon}\leq\frac{C}{\sqrt{\varepsilon}},
\end{equation*}
which is of the same order as in \eqref{C2p12}. Thus, it is still quite challenging to prove the lower bound of $p$ in the presence of strictly convex inclusions. This also requires us to develop new technique to handle this problem in the future.
\end{remark}

Finally, under the symmetric assumptions on the domain and boundary data, we have the following point-wise estimate for the Cauchy stress tensor $\sigma[{\bf u},p]$. Before stating it, we introduce a constant 
\begin{equation*}
q_{R}=\sum_{\alpha=1}^{2}(C_1^\alpha-C_2^\alpha)(q_{1}^\alpha)_{R},
\end{equation*}
which is independent of $\varepsilon$, here $(q_{1}^\alpha)_{R}$ is defined in \eqref{defqialpha}. 

\begin{theorem}\label{mainthmsigma}(Cauchy Stress estimates)
Under the assumptions of Theorem \ref{mainthm2}, let ${\bf u}\in H^1(D;\mathbb R^2)\cap C^1(\bar{\Omega};\mathbb R^2)$ and $p\in L^2(D)\cap C^0(\bar{\Omega})$ be a solution to \eqref{Stokessys} and \eqref{compatibility}. Then we have 
\begin{equation}\label{equ-sigma2D}
|\sigma[{\bf u},p-q_{R}]|\leq \frac{C}{\sqrt\varepsilon}\|{\boldsymbol\varphi}\|_{C^{2,\alpha}(\partial D;\mathbb R^2)},\quad\mbox{in}~\Omega_R.
\end{equation}
Moreover, at the segment $\overline{P_{1}P_{2}}$, if $\tilde b_{1}^{*1}[{\boldsymbol\varphi}]\neq0$, then
$$|\sigma[{\bf u},p-q_{R}]|(0,x_{2})\geq\frac{|\tilde b_{1}^{*1}[{\boldsymbol\varphi}]|}{C\sqrt{\varepsilon}},\quad|x_{2}|\leq\varepsilon.$$
\end{theorem}

\begin{remark}
Here we show that the blow-up rate $\varepsilon^{-1/2}$  of $|\sigma[{\bf u},p-q_{R}]|$ in dimension two is optimal, which is consistent with the result in \cite{AKKY}. 
\end{remark}

\begin{proof}[Proof of Theorem \ref{mainthmsigma}]
Repeat  the proof of Theorem \ref{mainthm2D}, we have the estimate of $|\nabla{{\bf u}}(x)|$ as showed in \eqref{upper-u2D}.
For $|p|$,  combining with $C_1^3=C_2^3$ in \eqref{C13C23}, for any $x\in\Omega_R$, the estimate  \eqref{upper-p2D} becomes
\begin{align*}
|p(x)-q_{R}|&\leq\sum_{\alpha=1}^{2}\left|(C_{1}^{\alpha}-C_{2}^{\alpha}\right)(p_{1}^{\alpha}(x)-(q_{1}^\alpha)_{R})|+C\\
&\leq\,\frac{C\sqrt{\varepsilon}}{\varepsilon}+\frac{C\varepsilon^{3/2}}{\varepsilon^2}+C\leq \frac{C}{\sqrt{\varepsilon}},
\end{align*}
where $q_{R}=\sum_{\alpha=1}^{2}(C_1^\alpha-C_2^\alpha)(q_{1}^\alpha)_{R}$. Thus, \eqref{equ-sigma2D} is proved.

For the lower bound, by virtue of \eqref{lowerC11} and $|\partial_{x_2}({\bf u}_1^1)^{(1)}|(0,x_2)\geq\frac{1}{\varepsilon}$, we have for any  $(0,x_2)\in\Omega_R$,
\begin{align*}
&\Big|\sigma[{\bf u},p-q_{R}]\Big|(0,x_{2})=\Big|2\mu e({\bf u})-(p-q_{R})\mathbb{I}\Big|(0,x_{2})\\
\geq&\,\Big|\sum_{\alpha=1}^{2}(C_{1}^{\alpha}-C_{2}^{\alpha})\Big(2\mu e({\bf u}_{1}^{\alpha})-(p_{1}^{\alpha}-(q_{1}^\alpha)_{R})\mathbb{I}\Big)\Big|(0,x_{2})-C\\
\geq&\,2\mu\Big|(C_{1}^{1}-C_{2}^{1}) e_{12}({\bf u}_{1}^{1})\Big|(0,x_{2})-C\geq\frac{|\tilde b_{1}^{*1}[{\boldsymbol\varphi}]|}{C\sqrt{\varepsilon}}.
\end{align*}
The proof of Theorem \ref{mainthmsigma} is finished.
\end{proof}

\noindent{\bf{\large Acknowledgements.}}
H.G. Li was partially supported by NSF of China (11971061).
L.J. Xu was partially supported by NSF of China (12301141).


\begin{thebibliography}{99}


\bibitem{ADN1959} S. Agmon, A. Douglis, L. Nirenberg, Estimates Near the Boundary for Solutions of Elliptic Partial Differential Equations Satisfying General Boundary Conditions I, Comm. Pure Appl. Math., 12 (1959), 623-727.

\bibitem{ADN1964} S. Agmon, A. Douglis, L. Nirenberg, Estimates Near the Boundary for Solutions of Elliptic Partial Differential Equations Satisfying General Boundary Conditions II, Comm.Pure Appl. Math., 17 (1964), 35-92

\bibitem{AKKL} H. Ammari, H. Kang, K. Kim, H. Lee, Strong convergence of the solutions of the linear elasticity and uniformity of asymptotic expansions in the presence of small inclusions. J. Differential Equations 254 (2013), no. 12, 4446-4464.

\bibitem{AKKY} H. Ammari, H. Kang, D. Kim, S. Yu, Quantitative estimates for stress concentration of the Stokes flow between adjacent circular cylinders. SIAM J. Math. Anal., 55 (2023), no.4,  3755-3806.

\bibitem{AKL} H. Ammari, H. Kang, M. Lim, Gradient estimates to the conductivity problem, Math. Ann. 332 (2005) 277-286.


\bibitem{AKL3} H. Ammari, H. Kang, H. Lee, J. Lee, M. Lim, Optimal estimates for the electrical field in two dimensions, J. Math. Pures Appl. 88 (2007) 307-324.

\bibitem{Bab} I. Babu\u{s}ka, B. Andersson, P. Smith, K. Levin, Damage analysis of fiber composites. I. Statistical analysis on fiber scale, Comput. Methods Appl. Mech. Eng. 172 (1999) 27-77.

\bibitem{BLL}  J.G. Bao, H.G. Li, and Y.Y. Li. Gradient estimates for solutions of the Lam\'{e} system with partially infinite coefficients, Arch. Ration. Mech. Anal. 215 (2015), 307-351.

\bibitem{BLL2} J.G. Bao; H.G. Li; Y.Y. Li, Gradient estimates for solutions of the Lam\'e system with partially infinite coefficients in dimensions greater than two. Adv. Math. 305 (2017), 298-338.


\bibitem{BLY} E. Bao, Y.Y. Li, B. Yin, Gradient estimates for the perfect conductivity problem, Arch. Ration. Mech. Anal. 193 (2009) 195-226.

\bibitem{BT} E. Bonnetier, F. Triki, On the spectrum of the Poincar\'e variational problem for two close-to-touching inclusions in 2D, Arch. Ration. Mech. Anal. 209 (2013) 541-567.

\bibitem{BV} E. Bonnetier, M. Vogelius, An elliptic regularity result for a composite medium with touching fibers of circular cross-section, SIAM J. Math. Anal. 31 (2000) 651-677.

\bibitem{BC} B. Budiansky, G.F. Carrier, High shear stresses in stiff fiber composites, J. Appl. Mech. 51 (1984) 733-735.

\bibitem{CL} Y. Chen, H.G. Li,  Estimates and Asymptotics for the stress concentration between closely spaced stiff $C^{1, \gamma}$ inclusions in linear elasticity. J. Funct. Anal. 281 (2021), 109038.

\bibitem{CS1} G. Ciraolo, A. Sciammetta, Gradient estimates for the perfect conductivity problem in anisotropic media, J. Math. Pures Appl. 127 (2019) 268-298.

\bibitem{CS2} G. Ciraolo, A. Sciammetta, Stress concentration for closely located inclusions in nonlinear perfect conductivity problems, J. Differ. Equ. 266 (2019) 6149-6178.

\bibitem{Dong} H. Dong, Gradient estimates for parabolic and elliptic systems from linear laminates. Arch Ration Mech Anal, 2012, 205: 119-149.

\bibitem{DLY} H. Dong, Y.Y. Li, Z.L. Yang, Optimal gradient estimates of solutions to the insulated
conductivity problem in dimension greater than two. arXiv:2110.11313v1 [math.AP].
To appear in J Eur Math Soc, 2022.

\bibitem{DLY2} H. Dong, Y.Y. Li, Z.L. Yang, Gradient estimates for the insulated conductivity problem:
the non-umbilical case. arXiv:2203.10081v1 [math.AP]. To appear in Journal de Mathematiques Pures et Appliquees.



\bibitem{G2011} G.P. Galdi. An introduction to the mathematical theory of the Navier-Stokes equations: steady-state problems. Springer, Cham (2011).


\bibitem{HanLin} Q. Han, F.H. Lin, Elliptic partial differential equations. Second edition. Courant Lecture Notes in Mathematics, 1. Courant Institute of Mathematical Sciences, New York; American Mathematical Society, Providence, RI, 2011. x+147 pp. ISBN: 978-0-8218-5313-9.

\bibitem{H} M. Hillairet, Lack of collision between solid bodies in a 2D constant-density incompressible
viscous flow. Communications in Partial Differential Equations 32 (2007)
1345-1371.

\bibitem{HK} M. Hillairet and T. Kela\"i, Justification of lubrication approximation: an application to fluid/solid interactions. Asymptot. Anal. 95 (2015) 187-241.

\bibitem{Jef1} G. B. Jeffery, Plane stress and plane strain in bipolar coordinates, Phil. Trans. Roy. Soc. London A 221 (1921), 265-293.

\bibitem{Jef2} G. B. Jeffrey, The rotation of two circular cylinders in a viscous fluid, Proc. Roy. Soc. A 101 (1922), 169-174.

\bibitem{K} H. Kang, Quantitative analysis of field concentration in presence of closely located inclusions of high contrast, Proceedings of the International Congress of Mathematicians 2022.

\bibitem{KLeY}H. Kang, H. Lee, K. Yun, Optimal estimates and asymptotics for the stress concentration between closely located stiff inclusions, Math. Ann. 363 (2015) 1281-1306.

\bibitem{KLiY} H. Kang, M. Lim, K. Yun, Asymptotics and computation of the solution to the conductivity equation in the presence of adjacent inclusions with extreme conductivities, J. Math. Pures Appl. 99 (2013) 234-249.

\bibitem{KLiY2}H. Kang, M. Lim, K. Yun, Characterization of the electric field concentration between two adjacent spherical perfect conductors, SIAM J. Appl. Math. 74 (2014) 125-146.


\bibitem{KY} H. Kang, S. Yu, Quantitative characterization of stress concentration in the presence of closely spaced hard inclusions in two-dimensional linear elasticity, Arch. Ration. Mech. Anal. 232 (2019), 121-196.

\bibitem{KangYu} H. Kang, S. Yu, Singular functions and characterizations of field concentrations: a survey. Anal. Theory Appl. 37 (2021), no. 1, 102-113.

\bibitem{Kratz} W. Kratz, On the maximum modulus theorem for Stokes function, Appl. Anal. 58 (1995), 293-302. 

\bibitem{Lady1959} O.A. Ladyzhenskaya, Investigation of the Navier-Stokes Equation for a Stationary Flow of an Incompressible Fluid, Uspehi Mat. Nauk., (3) 14 (1959) 75-97.


\bibitem{Lady} O.A. Ladyzhenskaya,  The mathematical theory of viscous incompressible flow. Mathematics and its applications, vol 2, 2nd edn. Gordon and Breach, New York (1969) 

\bibitem{Li2021} H.G. Li, Lower bounds of gradient's blow-up for the Lam\'e system with partially infinite coefficients, J. Math. Pures Appl. 149 (2021) 98-134.

\bibitem{Li} H.G. Li, Asymptotics for the electric field concentration in the perfect conductivity problem, SIAM J. Math. Anal. 52 (2020) 3350-3375.


\bibitem{Li-Li} H.G. Li, Y.Y. Li Gradient estimates for parabolic systems from composite material. Sci. China Math. 60 (2017), no. 11, 2011-2052.

\bibitem{LLBY} H.G. Li, Y.Y. Li, E.S. Bao, B. Yin, Derivative estimates of solutions of elliptic systems in narrow regions, Q. Appl. Math. 72(3) (2014) 589-596.



\bibitem{LX2}  
H.G. Li; L.J. Xu, Estimates for stress concentration between two adjacent rigid inclusions in Stokes flow. arXiv: 2310.09498v1 [math. AP].

\bibitem{LXZ} H.G. Li, L.J. Xu, P.H. Zhang, Stress blow-up analysis when a suspending rigid particle approaches the boundary in Stokes flow: 2D case. SIAM J. Math. Anal. 55 (2023), no. 5, 4493-4536.

\bibitem{LXZ2} H.G. Li, L.J. Xu, P.H. Zhang, Stress blow-up analysis when suspending rigid particles approach boundary in 3D Stokes flow. arXiv:2306.03591 [math.AP]. 

\bibitem{Mazya} V. Maz'ya, J. Rossmann,  A maximum modulus estimate for solutions of the Navier-Stokes system in domains of polyhedral type. Math. Nachr. 282 (2009), no. 3, 459-469. 


\bibitem{LN} Y.Y. Li, L. Nirenberg, Estimates for elliptic system from composite material, Commun. Pure Appl. Math. 56 (2003) 892-925.

\bibitem{LV}Y.Y. Li, M. Vogelius, Gradient estimatesfor solutions to divergence form elliptic equations with discontinuous coefficients, Arch. Ration. Mech. Anal. 153 (2000) 91-151.

\bibitem{LY} Y.Y. Li; Z.L. Yang, Gradient estimates of solutions to the insulated conductivity problem
in dimension greater than two. Math Ann., (2023) 385:  1775-1796.

\bibitem{Solonni1966} V.A. Solonnikov, General Boundary Value Problems for Douglis-Nirenberg Elliptic Systems II, Trudy Mat. Inst. Steklov, 92, 233-297; English Transl.: Proc. Steklov Inst. Math, 92, (1966), 212-272.

\bibitem{T1984} R. Temam, Navier-Stokes Equations, North-Holland, Amsterdam, 1984.

\bibitem{Wann} G. H. Wannier, Hydrodynamics of lubrication, Q. Appl. Math. 8 (1950), 7-32.

\bibitem{Wen} B. Weinkove, The insulated conductivity problem, effective gradient estimates
and the maximum principle. Math Ann., (2023) 385: 1-16.

\bibitem{Yun} K. Yun, Estimates for electric fields blown up between closely adjacent conductors with arbitrary shape, SIAM J. Appl. Math. 67 (2007) 714-730.

\bibitem{Yun2} K. Yun, An optimal estimate for electric fields on the shortest line segment between two spherical insulators in three dimensions, J. Differ. Equ. 261 (2016) 148-188.

\end{thebibliography}
\end{document}